\documentclass[final,onefignum,onetabnum]{siamart190516}

\usepackage{lipsum}
\usepackage{amsfonts}
\usepackage{graphicx}
\usepackage{epstopdf}
\usepackage{algorithmic}
\usepackage{amssymb}
\usepackage{hyphenat}
\usepackage{pdflscape}
\usepackage{siunitx}
\usepackage[normalem]{ulem}
\ifpdf
  \DeclareGraphicsExtensions{.eps,.pdf,.png,.jpg}
\else
  \DeclareGraphicsExtensions{.eps}
\fi

\newcommand{\R}{\mathbb{R}}

\newcommand{\mb}{\mathbb}
\newcommand{\abs}[1]{\left| {#1} \right|}
\newcommand{\bpa}[1]{\left( {#1} \right)}
\newcommand{\bbr}[1]{\left[ {#1} \right]}
\newsiamremark{assum}{Assumption}
\newcommand{\var}{\mbox{Var}}

\newsiamremark{remark}{Remark}
\newsiamremark{hypothesis}{Hypothesis}
\crefname{hypothesis}{Hypothesis}{Hypotheses}
\newsiamthm{claim}{Claim}

\headers{Adaptive FD Interval Estimation}{H.-J.M. Shi, Y. Xie, M.Q. Xuan, J. Nocedal}

\title{Adaptive Finite-Difference Interval Estimation for Noisy Derivative-Free Optimization \thanks{Submitted to the editors 10/12/2021.
\funding{Shi and Xie were supported by the Office of Naval Research grant N00014-14-1-0313 P00003. Xuan was supported by the National Science Foundation grant DMS-1620022. Nocedal was supported by AFOSR grant FA95502110084, and by National Science Foundation grant DMS-1620022.}}}

\author{Hao-Jun Michael Shi\thanks{Department of Industrial Engineering and Management Sciences, Northwestern University, Evanston IL (\url{hjmshi@u.northwestern.edu}, \url{ycxie@u.northwestern.edu}, \url{qxuan@u.northwestern.edu}, \url{j-nocedal@northwestern.edu)}
        }
      \and
      Yuchen Xie\footnotemark[2]
      \and
       Melody Qiming Xuan\footnotemark[2]
      \and
       Jorge Nocedal\footnotemark[2]}

\usepackage{amsopn}

\ifpdf
\hypersetup{
  pdftitle={Adaptive Finite-Difference Interval Estimation for Noisy Derivative-Free Optimization},
  pdfauthor={Hao-Jun Michael Shi, Yuchen Xie, Melody Qiming Xuan, Jorge Nocedal}
}
\fi

\begin{document}

\maketitle

\begin{abstract}
A common approach for minimizing a smooth nonlinear function is to employ finite-difference approximations to the gradient. While this can be easily performed when no error is present within the function evaluations, when the function is noisy, the optimal choice requires information about the noise level and higher-order derivatives of the function, which is often unavailable. Given the noise level of the function, we propose a bisection search for finding a finite-difference interval for any finite-difference scheme that balances the \textit{truncation error}, which arises from the error in the Taylor series approximation, and the \textit{measurement error}, which results from noise in the function evaluation. Our procedure produces reliable estimates of the finite-difference interval at low cost without explicitly approximating higher-order derivatives. We show its numerical reliability and accuracy on a set of test problems. When combined with L-BFGS, we obtain a robust method for minimizing noisy black-box functions, as illustrated on a subset of unconstrained CUTEst problems with synthetically added noise. 
\end{abstract}



\section{Introduction}
\label{intro}
\setcounter{equation}{0}
A powerful approach for derivative-free optimization is to utilize finite differences to approximate gradients, and to employ these approximations within a known nonlinear optimization method; see e.g., \cite{shi2021numerical}. These methods operate by spending at least $n + 1$ function evaluations at each iteration, where $n$ is the total number of variables. This lies in contrast to interpolation-based methods, which utilize prior function evaluations with only one new evaluation at each iteration; see \cite{conn2009introduction,larson2019derivative}. Therefore, in order for the finite-difference approach to be effective, one must ensure that the quality of the gradient is satisfactory and significant progress is made at each iteration of the optimization algorithm.

Often, black-box functions to be optimized are contaminated by stochastic or computational noise. This noise could arise naturally from modeling randomness within a simulation, or as a bi-product of an adaptive computation, for example through the early termination of an iterative solver. The presence of noise has largely prevented finite-difference methods from gaining more popularity within the derivative-free optimization community, as the choice of the finite-difference interval becomes increasingly critical as the noise level increases.

In this paper, we propose a procedure for computing the finite-difference interval  that is more robust than techniques proposed in the literature. It estimates the finite-difference interval directly by balancing truncation and measurement error, rather than by first computing an estimate of a higher derivative of the function, as is common. The proposed procedure applies to any finite-difference scheme, including first-order forward and central differences, as well as higher-order schemes. 

To outline our procedure, we first note that finite-difference interval estimation requires knowledge of both the noise level and higher-order derivative of the function \cite{more2012estimating}. While the former may be known \emph{a priori} or can be estimated by sampling or computing difference tables \cite{more2011estimating}, the latter quantity is not normally available to the user. 

Let us consider, for example, the problem of estimating the derivative of a smooth univariate function $\phi : \R \rightarrow \R$. Assume that we are only provided noisy function evaluations of the form
\begin{equation}
    f(t) = \phi(t) + \epsilon(t)
\end{equation}
where $\epsilon: \R \rightarrow \R$ models the error. Let us assume for simplicity that the noise is bounded.
The simplest and cheapest finite-difference approximation to the first derivative is the forward-difference approximation. If $\phi^{(d)}$ denotes the $d$-th order derivative of $\phi$, then the forward-difference approximation is computed by
\begin{equation}\label{eq:forward diff}
    \phi^{(1)}(t) \approx \frac{f(t + h) - f(t)}{h} \triangleq f^{(1)}(t; h)
\end{equation}
where $h > 0$ is the finite-difference interval. (Note the slight abuse of notation by denoting $f^{(d)}$ as the finite-difference approximation to the $d$-th order derivative.) With no noise, excluding round-off error, one would ideally choose $h$ as small as possible, the common practical choice being $h = \max\{1, |x|\} \sqrt{\epsilon_M}$, where $\epsilon_M$ is machine precision, to handle rounding errors \cite{berahas2019derivative}. However, this choice is poor under the presence of large errors, as is well-known.

To see this, consider the following decomposition of the error in the forward-difference approximation:
\begin{equation}
    \abs{f^{(1)}(t; h) - \phi^{(1)}(t)} \leq \abs{\frac{\phi(t + h) - \phi(t)}{h} - \phi^{(1)}(t)}  + \abs{\frac{\epsilon(t + h) - \epsilon(t)}{h}}.
\end{equation}
We will call the error induced by the first term \textit{truncation error} since it arises from truncation of the Taylor series, and the error induced by the second term \textit{measurement error} due to error in the function evaluations. 

Note that if $h$ is small, then the truncation error is small but the measurement error may be large. On the other hand, if $h$ is too large, the measurement error may be small but the truncation error may be too large. Therefore, the optimal $h$ trades off these two terms by making the error from each of these two sources equal. The adaptive procedure proposed in this paper for estimating the finite-difference interval in the presence of noise  properly balances these two different sources of error by using a bisection technique. The procedure aims to be: (1) reliable, that is, applicable to most practical problems of interest, and more robust than the techniques proposed in the literature; (2) accurate, producing near-optimal estimates of the finite-difference interval; and (3) efficient, employing the least number of function evaluations possible. 

This paper is organized into five sections. We present the notation and literature review in the rest of this section. In Section \ref{sec:forward diff}, we introduce our finite-difference interval estimation procedure for the forward-difference case. In Section \ref{sec:general fd}, we present the generalized procedure for arbitrary finite-difference schemes and provide theoretical guarantees for the termination of our procedure. Extensive numerical results on synthetic problems with injected noise are provided in Section \ref{sec:experiments}, and concluding remarks are made in Section \ref{sec:final remark}.

\subsection{Literature Review}

The problem of estimating derivatives, particularly in the presence of rounding errors, is a fundamental question within numerical analysis and scientific computing. Fornberg proposed a stable algorithm for generating finite-difference formulas on arbitrarily spaced grids \cite{fornberg1988generation}. Lyness and Moler observed that the Cauchy integral theorem allows one to evaluate the $d$-th derivative of a complex function as a closed complex integral via numerical integration techniques \cite{lyness1967numerical}. This was simplified and extended by Squire and Trapp who observed that one could avoid cancellation error by using complex perturbations in the Taylor expansion, called complex step differentiation \cite{squire1998using}. This has more recently led to extensions of the complex step to evaluating the Hessian by Hare and Srivastava \cite{hare2020applying}. Brekelmans et al. compared design of experiments schemes against standard finite-difference schemes within the stochastic noise regime \cite{brekelmans2005gradient}.

To handle rounding errors, Curtis and Reid describe a heuristic that estimates the truncation and rounding errors using central and forward-difference estimates. The ratio between the two estimates of these errors are used to determine the finite-difference interval \cite{curtis1974choice}. Stepleman and Winarsky use a set of decreasing central-difference intervals. The optimal interval is obtained by the smallest interval that does not violate monotonic decrease in the absolute difference between consecutive central-difference estimates \cite{stepleman1979adaptive}. Gill, Murray, Saunders, and Wright introduced an adaptive procedure for computing forward-difference intervals by utilizing a ratio to determine the second derivative \cite{GillMurrWrig81,GillMurrSaunWrig83}. Their procedure has some similarities with our approach, which we discuss in Section~\ref{subsec:comparison}. Barton proposed an adaptive procedure for handling rounding or multiplicative errors by interpreting the function values as correct up to a fixed number of significant digits and ensuring that at least a certain number of significant digits change from the resulting difference interval \cite{barton1992computing}. Most recently, Mor\'e and Wild proposed a heuristic for estimating the second derivative for determining the forward-difference interval that checks two conditions: (1) if the noise dominates the second-order derivative; and (2) if the forward and backward difference is too large relative to the function values \cite{more2012estimating}. A comparison of the resulting errors between finite-difference and simplex gradients were analyzed in \cite{berahas2019theoretical}.

Incorporating finite differences into optimization methods have also had a long history. Kiefer and Wolfowitz first applied finite differences to stochastic approximation \cite{kiefer1952stochastic}. Kelley developed an implicit filtering BFGS method that utilizes finite differences in the case where noise decays as the iterates converge to the solution \cite{choi2000superlinear,kelley2011implicit}. Berahas et al. proposed a finite-difference L-BFGS method that incorporates \texttt{ECNoise} and a heuristic by Mor\'e and Wild to estimate the second derivative \cite{berahas2019derivative,more2011estimating,more2012estimating}. Most recently, Shi et al. tested finite-difference methods within the unconstrained, least squares, and constrained settings assuming knowledge of the noise level \cite{shi2021numerical}.

\subsection{Notation}

In the following sections, we will use Bachmann-Landau notation liberally. Suppose $f:\R \rightarrow \R^{\geq 0},~ g:\R \rightarrow \R$. We will write $g(h) = \mathcal{O}(f(h))$ if there exist constants $C_1, C_2 > 0$ and $\delta > 0$ such that $C_1 f(h) \leq \abs{g(h)} \leq C_2 f(h)$ for all $\abs{h} \leq \delta$. We write $g(h) = o(f(h))$ if for every $\epsilon > 0$ there exists a $\delta > 0$ such that $\abs{g(h)} \leq \epsilon {f(h)}$ for all $\abs{h} \leq \delta$. Finally, we write $g(h) = O(f(h))$ if there exist constants $C > 0$ and $\delta > 0$ such that $\abs{g(h)} \leq C f(h)$ for all $\abs{h} \leq \delta$. 

We will use $\phi^{(d)}: \R \rightarrow \R$ to denote the $d$-th order derivative of $\phi$. For a given vector $x \in \R^n$, $[x]_i$ denotes the $i$-th component of $x$. Similarly, for a given matrix $A \in \R^n$, $[A]_{ij}$ denotes the $(i, j)$-th entry of $A$. We will use $\|\cdot\|$ to denote the standard Euclidean norm unless otherwise specified.

\section{An Adaptive Forward-Difference Interval Estimation Procedure} \label{sec:forward diff} \ \\
Suppose we are interested in determining the finite-difference interval for the forward-difference approximation of the first derivative of $\phi$. Since the Taylor expansion of the function $\phi$ is given by
\begin{equation*}
    \phi(t + h) = \phi(t) + \phi^{(1)}(t) h + \frac{\phi^{(2)}(t)}{2} h^2 + o(h^2),
\end{equation*}
the total error can be bounded by
\begin{equation}\label{eq:error bound}
    |\phi^{(1)}(t) - f^{(1)}(t; h)| \leq \underbrace{\frac{|\phi^{(2)}(t)| h}{2}}_{T_1} + \underbrace{\frac{2 \epsilon_f}{h}}_{T_2} + o(h),
\end{equation}
where $T_1$ denotes the error arising from truncation of the Taylor series, which we call \textit{truncation error}, and $T_2$ is the error arising from the function evaluations, also called \textit{measurement error}.  
Minimizing the right hand side, ignoring the higher-order terms, leads to the near-optimal interval
\begin{equation} \label{eq:forward opt}
    h^* \approx 2 \sqrt{\frac{\epsilon_f}{|\phi^{(2)}(t)|}}.
\end{equation} 
This formula requires an estimate of the second derivative $|\phi^{(2)}(t)|$.
We now propose a procedure that yields an interval $h = \mathcal{O}\left(\sqrt{\frac{\epsilon_f}{|\phi^{(2)}(t)|}}\right)$ without estimating $|\phi^{(2)}(t)|$ separately. 

Our procedure balances the truncation $T_1$ and measurement error $T_2$ in \eqref{eq:error bound}. To do so, it estimates the ratio between these two errors directly and attempts to find an interval $h$ for which this ratio is close to some constant value.  We claim that the ratio $T_1 / T_2$  can be approximated using noisy function evaluations $f(\cdot)$ and the noise level $\epsilon_f$, through a \textit{testing ratio} such as
\begin{equation}\label{eq:test ratio}
    r(h; f, t, \epsilon_f) = \frac{|f(t + 4h) - 4 f(t + h) + 3 f(t)|}{8 \epsilon_f}.
\end{equation}
Our algorithm achieves this by finding an interval $h$ such that
\begin{equation} \label{eq:interval}
    r(h; f, t, \epsilon_f) \in [r_l, r_u]
\end{equation}
for $r_l > 1$ and $r_u > r_l + 2$.

To see why this procedure works to give us a near-optimal $h$ as in \eqref{eq:forward opt}, note that
\begin{equation}
    \phi(t + 4h) - 4 \phi(t + h) + 3 \phi(t) = 6 \phi^{(2)}(t) h^2 + o(h^2).
\end{equation}
Therefore, if we expand \eqref{eq:test ratio}, we obtain:
\begin{equation}
    r(h; f, t, \epsilon_f) = \left|\frac{3 \phi^{(2)}(t) h^2}{4 \epsilon_f} + \frac{\epsilon(t + 4h) - 4 \epsilon(t + h) + 3 \epsilon(t)}{8 \epsilon_f} + o(h^2) \right|.
\end{equation}
Since $\left|\frac{\epsilon(t + 4h) - 4 \epsilon(t + h) + 3 \epsilon(t)|}{8 \epsilon_f} \right| \leq 1$ by the fact that $|\epsilon(t)| \leq \epsilon_f$ for all $t \in \R$, by  imposing \eqref{eq:interval} and ignoring the $o(h^2)$ term, we approximately have
\begin{equation} \label{eq:forward int0}
    \frac{3 |\phi^{(2)}(t)| h^2}{4 \epsilon_f} \in [r_l - 1, r_u + 1]. 
\end{equation}
Thus, we obtain an interval $h$ such that 
\begin{equation} \label{eq:forward int}
h \in \frac{2}{\sqrt{3}} \sqrt{\frac{\epsilon_f}{|\phi^{(2)}(t)|}} \cdot [\sqrt{r_l - 1}, \sqrt{r_u + 1}]. 
\end{equation}
Therefore, if $r_l$ and $r_u$ are chosen properly, such as $r_l = 1.5$ and $r_u = 6$, we obtain $h \in [\sqrt{0.5}, \sqrt{7}] \cdot \sqrt{\frac{\epsilon_f}{|\phi^{(2)}(t)|}}$, which is the same order as the optimal finite-difference interval \eqref{eq:forward opt}, differing only by a small constant factor. 

To find an $h$ that satisfies \eqref{eq:interval}, we will perform a bisection search on $h$. In particular, if $r(h; f, t, \epsilon_f) < r_l$, then the numerator of the testing ratio is dominated by noise, indicating that $h$ is too small. On the other hand, if $r(h; f, t, \epsilon_f) > r_u$, then the numerator significantly dominates the noise, which implies that $h$ is too large. Our procedure for forward differences is summarized in Algorithm \ref{algo:fd_est_proc}.\\

\begin{algorithm}[H]
\caption{Adaptive Forward-Difference Interval Estimation}
\label{algo:fd_est_proc}

\hspace*{\algorithmicindent} \textbf{Input:} {One-dimensional noisy function $f: \mathbb{R} \to \mathbb{R}$; noise level $\epsilon_f > 0$; lower- and upper-bound $(r_l, r_u) = (1.5, 6)$;}

\hspace*{\algorithmicindent} \textbf{Output:} {Finite-difference interval $h$ such that \eqref{eq:interval} holds.}

\begin{algorithmic}[1]
\STATE{$h \gets \frac{2}{\sqrt{3}} \sqrt{\epsilon_f}$;}
\STATE{$l \gets 0$, $u \gets +\infty$;}
\WHILE{\textbf{True}}
\STATE{Evaluate $r(h; f, t, \epsilon_f) = \frac{\abs{f(t + 4 h) - 4 f(t + h) + 3 f(t)}}{8 \epsilon_f}$;}
\IF{$r(h; f, t, \epsilon_f) < r_l$}
\STATE{$l \gets h$;}
\ELSIF{$r(h; f, t, \epsilon_f) > r_u$}
\STATE{$u \gets h$;}
\ELSE
\STATE{ \textbf{break;} }
\ENDIF
\IF{$u = +\infty$}
\STATE{$ h \gets 4 h$;} \label{line: u}
\ELSIF{$l = 0$}
\STATE{$h \gets h / 4$;} \label{line: l}
\ELSE
\STATE{$h \gets (l+u)/2$;}
\ENDIF
\ENDWHILE
\RETURN $h$
\end{algorithmic}
\end{algorithm}

By scaling $h$ by a factor of 4 in lines~\ref{line: u} and \ref{line: l} in Algorithm~\ref{algo:fd_est_proc}, the new trial $h$ only requires a single new function evaluation to check the testing ratio \eqref{eq:test ratio} when $h$ is updated monotonically.

In addition, the testing ratio is affine-invariant with respect to the function of interest in the sense that $r(h; f, t, \epsilon_f)$ remains unchanged if applied to a modified function $\tilde{f}(t) = a f(t) + b$ for $a \neq 0$ and $b \in \R$ with noise level $|a| \epsilon_f$, i.e. $r(h; \tilde{f}, t, |a| \epsilon_f) = r(h; f, t, \epsilon_f)$. Therefore, the finite-difference interval $h$ will correctly remain unchanged under this transformation.

\subsection{Comparison to Prior Methods} \label{subsec:comparison}

Although the definition of the testing ratio appears similar to the ratio in Gill et al. \cite{GillMurrSaunWrig83}, which is defined as the inverse ratio
\begin{equation}\label{reverse ratio}
    \frac{4 \epsilon_f}{|f(t + \tilde{h}) - 2 f(t) + f(t - \tilde{h})|},
\end{equation}
they use the finite difference interval $\tilde{h}$ to estimate the second derivative $|\phi^{(2)}(t)|$. In contrast, our approach utilizes a bisection search as in Algorithm \ref{algo:fd_est_proc} to find the difference interval $h$ directly without estimating the second derivative. The direct estimation of $h$ avoids the nested estimation as in Gill et al., which can cause larger errors.

We also note that one may reduce the cost of our procedure by reusing prior function evaluations within the bisection search, with an appropriate choice of the scaling factor for $h$. Additionally, one can perform forward differences without additional function evaluations as $f(t+h)$ and $f(t)$ are already evaluated in our testing ratio \eqref{eq:test ratio}. This is an advantage over the procedure of Gill et al. See Section \ref{sec:general fd} for further discussion on this topic.

Our technique also differs from Mor\'e and Wild's procedure \cite{more2012estimating}. Their procedure estimates the second derivative $\phi^{(2)}(t)$ by
\begin{equation} \label{eq:second der}
    \phi^{(2)}(t) \approx \frac{f(t + \tilde{h}) - 2 f(t) + f(t - \tilde{h})}{\tilde{h}^2} \triangleq f^{(2)}(t; \tilde{h}),
\end{equation}
and then inserts this estimate into the optimal formula \eqref{eq:forward int}. The difference interval $\tilde{h}$ is required to satisfy
\begin{align}
    |f(t + \tilde{h}) - 2 f(t) + f(t - \tilde{h})| & \geq \tau_1 \epsilon_f \label{eq:more-wild 1}\\
    |f(t \pm \tilde{h}) - f(t)| & \leq \tau_2 \max\{|f(t)|, |f(t \pm \tilde{h})|\} \label{eq:more-wild 2}
\end{align}
with $\tau_1 \gg 1$ and $\tau_2 \in (0, 1)$. Their method attempts to satisfy this within two trials as follows:
\begin{enumerate}
    \item Set $\tilde{h}_1 = \sqrt[4]{\epsilon_f}$ and compute $\mu_1 = |f^{(2)}(t; \tilde{h}_1)|$ using \eqref{eq:second der}. If conditions \eqref{eq:more-wild 1} and \eqref{eq:more-wild 2} are satisfied for $\tilde{h}_1$, return $\mu_1$. 
    \item Set $\tilde{h}_2 = \sqrt[4]{\epsilon_f / \mu_1}$ and compute $\mu_2 = |f^{(2)}(t; \tilde{h}_2)|$ using \eqref{eq:second der}. If conditions \eqref{eq:more-wild 1} and \eqref{eq:more-wild 2} are satisfied for $\tilde{h}_2$, return $\mu_2$.
    \item If $|\mu_1 - \mu_2| \leq \frac{1}{2} \mu_2$, return $\mu_2$.
\end{enumerate}
If the heuristic is unable to return an estimate $\mu$ of $\phi^{(2)}(t)$ after two trials, this is considered as a failure.

While \eqref{eq:more-wild 1} appears similar to the testing ratio, it is better interpreted as ensuring that noise does not dominate the second-derivative estimation due to the large choice of $\tau_1 = 100$. The second condition \eqref{eq:more-wild 2} is not affine-invariant in the sense that adding a sufficiently large $b$ can force the condition to be satisfied. This is undesirable as perturbations of the function should not change the overall behavior of the method.

\section{Generalized Finite-Difference Interval Estimation} \label{sec:general fd}

Typically, finite-difference interval estimation procedures for numerical optimization focus on forward differences \cite{barton1992computing,GillMurrSaunWrig83,more2012estimating}. However, in the noisy regime, higher-order finite-difference approximations, such as central differences, can yield more accurate approximations; see \cite{shi2021numerical}. As a result, in order to attain the highest possible accuracy, one must design methods that efficiently find a near-optimal difference interval for more general finite-difference schemes. To handle this, we propose a generalization of the forward-difference case, Algorithm \ref{algo:fd_est_proc}, for $d$-th order derivatives.

Consider a finite-difference approximation scheme $S = (w, s)$ defined over $m$ points, where we approximate $\phi^{(d)}(t)$ using the equation
\begin{equation}\label{eq:scheme_def_w_s}
    f_S^{(d)}(t; h) = \frac{\sum_{j = 1}^m w_j \cdot f(t + h s_j)}{h^{d}} \approx \phi^{(d)}(t)
\end{equation}
where $w \in \R^m$ and $s \in \R^m$ are the associated weights and shifts of the finite-difference scheme. As in the forward-difference setting, we will use a slight abuse of notation by denoting the finite-difference approximation as $f^{(d)}_S(t; h)$. The \textit{forward-difference} scheme for approximating the first derivative (i.e., $d = 1$) is obtained by defining $s = (0,1)^T$ and $w = (-1, 1)^T$, while \textit{central-difference} scheme for $d = 1$ is defined by $s = (-1, 1)^T$ and $w = (-\frac{1}{2}, \frac{1}{2})^T$. The standard \textit{second-order central-difference} scheme (for $d=2$) is defined as $s = (-1, 0, 1)^T$ and $w = (1, -2, 1)^T$. 

In order for the finite-difference scheme to be valid, the coefficients $w$ and shifts $s$ must be chosen such that the Taylor expansion of the finite-difference approximation over the function $\phi$ satisfies
\begin{equation}\label{eq:def_error_order_q}
    \sum_{j = 1}^m w_j \cdot \phi(t + h s_j) = \phi^{(d)}(t) h^{d} + c_q \phi^{(q)}(t) h^q + o\bpa{h^{q}}.
\end{equation}
Here $q \geq d + 1$ denotes the order of the remainder term\footnote{Note that the order of accuracy can be higher than $d + 1$ for certain schemes, such as central-difference approximations.}. This ensures that $f_S^{(d)}(t; h) \approx \phi^{(d)}(t)$. In order to guarantee this, the finite-difference scheme $S$ must satisfy
\begin{equation*}
    \frac{1}{l!} \sum_{j = 1}^m w_j s_j^{l} = 
    \begin{cases}
    1 & \mbox{ if } l = d \\
    0 & \mbox{ if } l < q, ~ l \neq d
    \end{cases}
\end{equation*}
and as a result
\begin{equation*}
    c_q = \frac{1}{q!} \sum_{j = 1}^m w_j  s_j^{q}.
\end{equation*}
(See the Appendix \ref{app:fd tables} for more detail on how generic finite-difference schemes are derived.)
Therefore, in the presence of noise, the worst-case error for the finite-difference scheme of interest can be bounded by
\begin{equation*}
    |f^{(d)}(t; h) - \phi^{(d)}(t)| \leq \abs{c_q} \abs{\phi^{(q)}(t)} h^{q - d} + \|w\|_1 \epsilon_f h^{-d} + o\bpa{h^{q - d}}.
\end{equation*}
One can define an approximately optimal choice of $h$:
\begin{equation}\label{eq:h_opt_scheme}
    h^* \approx \abs{\frac{d}{q - d} \cdot \frac{\|w\|_1 \epsilon_f}{c_q \phi^{(q)}(t)}}^{1 / q}.
\end{equation}

While $\epsilon_f$ is assumed to be known and $d$, $q$, $w$ and $c_q$ are available, the $q$-th order derivative $\phi^{(q)}(x)$ is unknown and often difficult to estimate. Following the idea from the forward-difference case, we propose a procedure for estimating \eqref{eq:h_opt_scheme} directly. We first construct a testing ratio $r_S$ associated with scheme $S$:
\begin{equation}\label{eq:general test ratio}
    r_S(h; f, t, \epsilon_f) = \frac{\abs{\sum_{j = 1}^{\tilde{m}} \tilde{w}_j \cdot f(t + h \tilde{s}_j)}}{\epsilon_f}
\end{equation}
where $\tilde{w}, \tilde{s} \in \R^{\tilde{m}}$ where $\tilde{m} \geq q + 1$, $\tilde{s}_j \neq \tilde{s}_k$ for all $j \neq k$, and $\tilde{w}$ and $\tilde{s}$ satisfies
\begin{equation}\label{eq:testing_ratio_taylor}
\sum_{j = 1}^{\tilde{m}} \tilde{w}_j \cdot \phi(t + h \tilde{s}_j) = c_r \phi^{(q)}(t) h^q + o\bpa{h^{q}}, ~~~ c_r = \frac{1}{q!}\sum_{j = 1}^{\tilde{m}} \tilde{w}_j \tilde{s}_j^{q} \neq 0.
\end{equation}
Without loss of generality, we assume that $\tilde{w}$ satisfies $\|\tilde{w}\|_1 = 1$. This can be done by normalizing $\tilde{w}$. We then perform a bisection search to find an interval $h$ that satisfies
\begin{equation}\label{eq:general interval}
    r_S(h; f, t, \epsilon_f) \in [r_l, r_u]
\end{equation}
for some $r_l > 1$ and $r_u > r_l + 2$. The procedure is summarized in
Algorithm \ref{algo:gen_est_proc}. 

\begin{algorithm}[H]
\caption{Adaptive Finite-Difference Interval Estimation}
\label{algo:gen_est_proc}

\hspace*{\algorithmicindent} \textbf{Input:} {One-dimensional noisy function $f: \mathbb{R} \to \mathbb{R}$; noise level $\epsilon_f > 0$; testing ratio $r_S(h; f, t, \epsilon_f)$ for scheme $S$; lower- and upper-bound $r_l$ and $r_u$ satisfying $1 < r_l < r_u - 2$; initial interval $h_0$; scaling factor $\eta > 1$;}

\hspace*{\algorithmicindent} \textbf{Output:}{Finite-difference interval $h$ such that \eqref{eq:general interval} holds.}

\begin{algorithmic}[1]
\STATE{$h \gets h_0$;}
\STATE{$l \gets 0$, $u \gets +\infty$;}
\WHILE{\textbf{True}}
\STATE{Evaluate $r_S(h; f, t, \epsilon_f) $;}
\IF{$r_S(h; f, t, \epsilon_f) < r_l$}
\STATE{$l \gets h$;}
\ELSIF{$r_S(h; f, t, \epsilon_f) > r_u$}
\STATE{$u \gets h$;}
\ELSE
\STATE{ \textbf{break;} }
\ENDIF
\IF{$u = +\infty$}
\STATE{$h \gets \eta h$;}
\ELSIF{$l = 0$}
\STATE{$h \gets h / \eta$;}
\ELSE
\STATE{$h \gets (l+u)/2$;}
\ENDIF
\ENDWHILE
\RETURN $h$
\end{algorithmic}
\end{algorithm}

By \eqref{eq:testing_ratio_taylor}, we can see that
\begin{align}
    r_S(h; f, t, \epsilon_f) & = \abs{\frac{c_r \phi^{(q)}(t) h^q}{\epsilon_f} + \frac{\sum_{j = 1}^{\tilde{m}} \tilde{w}_j \cdot \epsilon(t + h \tilde{s}_j)}{\epsilon_f} + o\bpa{h^q}} \\
    & = \abs{\frac{c_r \phi^{(q)}(t) h^{q}}{\epsilon_f} + \Delta + o\bpa{h^{q}}}, \qquad \Delta \triangleq \frac{\sum_{j = 1}^{\tilde{m}} \tilde{w}_j \cdot \epsilon(t + h \tilde{s}_j)}{\epsilon_f}.
\end{align}
Note that by definition of $\Delta$, $\abs{\Delta} \leq 1$. This is a consequence of the requirement that $\|\tilde{w}\|_1 = 1$ and that $|\epsilon(t)| \leq \epsilon_f$ for all $t \in \R$. Therefore, if we have $r_S(h; f, t, \epsilon_f) \in [r_l, r_u]$ and if we the ignore $o\bpa{h^q}$ term, then we (approximately) have
\begin{equation} \label{eq: ratio interval}
    \abs{\frac{c_r \phi^{(q)}(t) h^q}{\epsilon_f}} \in [r_l - 1, r_u + 1],
\end{equation}
i.e.,
\begin{equation}\label{eq:general opt}
    h \in \bbr{\left(\frac{r_l - 1}{\abs{c_r}} \frac{\epsilon_f}{\abs{\phi^{(q)}(t)}}\right)^{1/q}, \left(\frac{r_u + 1}{\abs{c_r}} \frac{\epsilon_f}{\abs{\phi^{(q)}(t)}} \right)^{1/q}}.
\end{equation}
Note from \eqref{eq:h_opt_scheme} that $h$ has the same dependence on $\epsilon_f$ and $\phi^{(q)}(t)$ as $h^*$. As in the forward-difference case, our algorithm is invariant to affine transformations with respect to the function. \\

\textit{Example 1 (First-Order Central Difference).} Consider the first-order central-difference scheme for approximating the first derivative:
\begin{equation}
    f^{(1)}_S(t; h) = \frac{f(t + h) - f(t - h)}{2h},
\end{equation}
where $s = (-1, 1)^T$ and $w = (-\frac{1}{2}, \frac{1}{2})^T$. The Taylor expansion of the numerator is given as:
\begin{equation*}
    \frac{\phi(t + h) - \phi(t - h)}{2 h} = \phi^{(1)}(t) + \frac{\phi^{(3)}(t) h^2}{6} + o(h^2).
\end{equation*}
The full error of the derivative approximation and the approximate optimal choice of $h$ are:
\begin{equation*}
    \abs{f_S^{(1)}(t; h) - \phi^{(1)}(t)} \leq \frac{\abs{\phi^{(3)}(t)} h^2}{6} + \frac{\epsilon_f}{h} + o(h^2), ~~~~~ h^* \approx \sqrt[3]{\frac{3 \epsilon_f}{\abs{\phi^{(3)}(t)}}}.
\end{equation*}
One example of a valid testing ratio is:
\begin{equation}\label{eq:cd test ratio ex1}
    r_S(h; f, t, \epsilon_f) = \frac{|f(t + 3h) - 3 f(t + h) + 3 f(t - h) - f(t - 3h)|}{8 \epsilon_f}.
\end{equation}

\textit{Example 2 (Second-Order Central Difference).} Consider the second-order central-difference scheme for approximating the second derivative:
\begin{equation}
    f^{(2)}_S(t; h) = \frac{f(t + h) - 2 f(t) + f(t - h)}{h^2},
\end{equation}
where $s = (-1, 0, 1)^T$ and $w = (-\frac{1}{2}, \frac{1}{2})^T$. The Taylor expansion of the numerator is given as:
\begin{equation*}
    \frac{\phi(t + h) - 2 \phi(t) + \phi(t - h)}{h^2} = \phi^{(2)}(t) + \frac{\phi^{(4)}(t) h^2}{24} + o(h^2).
\end{equation*}
The full error of the derivative approximation and the approximate optimal choice of $h$ are:
\begin{equation*}
    \abs{f_S^{(2)}(t; h) - \phi^{(2)}(t)} \leq \frac{\abs{\phi^{(4)}(t)} h^2}{24} + \frac{4 \epsilon_f}{h^2} + o(h^2), ~~~~~ h^* \approx 2 \sqrt[4]{\frac{6 \epsilon_f}{\abs{\phi^{(4)}(t)}}}.
\end{equation*}
One example of a valid testing ratio is:
\begin{equation}\label{eq:cd test ratio ex2}
    r_S(h; f, t, \epsilon_f) = \frac{|f(t + 2h) - 4 f(t + h) + 6 f(t) - 4 f(t - h) + f(t - 2h)|}{16 \epsilon_f}.
\end{equation}

\subsection{Practical Considerations} \label{sec:practical}
We make a few observations geared to making the procedure both efficient and robust.
\medskip

\noindent\textit{I. Choice of $r_l$ and $r_u$.} Ideally, one should choose $r_l$ and $r_u$ such that they are close to the optimal ratio
\begin{equation*}
    r^* = {\frac{d}{q - d}\cdot \abs{\frac{c_r}{c_q}} \cdot \|w\|_1}
\end{equation*}
in order to yield an $h$ that is close to $h^*$ in \eqref{eq:h_opt_scheme}. However, this is not directly possible in the presence of noise, which requires that $1 < r_l < r_u - 2$ in order to ensure finite-termination; see Section \ref{sec:finite term}. Therefore, we ideally want to select $(r_l, r_u)$ sufficiently large such that $1 < r_l < r_u - 2$ and, if possible, such that $r^*$ is logarithmically centered within the interval $[r_l, r_u]$:

\begin{equation} \label{eq:log center}
    \sqrt{r_l r_u} = r^* \iff \log r^* = \frac{\log r_l + \log r_u}{2}
\end{equation}
with 
\begin{equation} \label{eq:ideal rl ru}
    (r_l, r_u) = \left(\frac{r^*}{\beta}, \beta r^* \right) \iff (\log r_l, \log r_u) = (\log r^* - \log \beta, \log r^* + \log \beta)
\end{equation}
for some $\beta > 1$. Note that having \eqref{eq:ideal rl ru} is not always possible since we require that $1 < r_l < r_u - 2$. Therefore, we use the values:
\begin{equation}\label{eq:rl ru}
    r_l = \max \left\{1 + \eta, \frac{r^*}{\beta} \right\}, ~~~~~ r_u = \max \left\{3 (1 + \eta), \beta r^* \right\}
\end{equation}
for some $\eta > 0$. (In our experiments, we set $\eta = 0.1$ and $\beta = 2$.)

Note that when $r_l, r_u > r^*$, the algorithm may overestimate $|\phi^{(q)}(t)|$ and hence underestimate $h$. In order to avoid this in practice, we have found that it is preferable to choose a testing ratio such that the optimal ratio $r^* \geq \beta (1 + \eta)$. This could be done by choosing a different $\tilde{w}$ in the testing ratio. \\

\noindent\textit{II. Generation of Testing Ratio.} Although many choices of $r_S$ are possible for any finite-difference scheme $S$, it would be useful to have a method for automatically generating valid testing ratios that efficiently utilize function values. A simple yet useful way to construct $r_S$ is through the formula
\begin{equation}\label{eq:gen_ratio}
    r_S^\alpha(h; f_S, t, \epsilon_f) = \frac{\abs{\left(f^{(d)}_S(t; h) - \alpha^{-d} f^{(d)}_{S}(t; \alpha h) \right) h^d}}{A \epsilon_f},
\end{equation}
where $\alpha \neq 1$ and $A$ is computed by normalizing the coefficients such that $\|\tilde{w}\|_1 = 1$ is satisfied. 

This approach is guaranteed to generate a valid testing ratio $r_S$ for any $\alpha \neq 1$ since it cancels out the $\phi^{(d)}(t)$ term in the Taylor expansion, leaving only the relevant higher-order term of order $q$ of interest. In particular, since
\begin{align*}
    \sum_{j = 1}^m w_j \cdot \phi(t + h s_j) & = \phi^{(d)}(t) h^d + c_q \phi^{(q)}(t) h^q + o(h^q) \\
    \sum_{j = 1}^m w_j \cdot \phi(t + \alpha h s_j) & = \phi^{(d)}(t) (\alpha h)^d + c_q \phi^{(q)}(t) (\alpha h)^q + o(h^q),
\end{align*}
we obtain that
\begin{align*}
    \left(f^{(d)}_S(t; h) - \alpha^{-d} f^{(d)}_{S}(t; \alpha h) \right) h^d & = \sum_{j = 1}^m \left(w_j \cdot \phi(t + h s_j) - \alpha^{-d} w_j \cdot \phi(t + \alpha h s_j)\right) \\
    & = c_q (1 - \alpha^{q - d}) \phi^{(q)}(t) h^q + o(h^q),
\end{align*}
which satisfies \eqref{eq:testing_ratio_taylor} with an effective $c_r = c_q (1 - \alpha^{q - d}) / A$ as desired.

With this design of the testing ratio, the function values required for the corresponding finite-difference scheme are already evaluated within the testing ratio. We can therefore obtain the finite-difference approximation using previously computed function values at no additional cost.

A couple of rules of thumb can be applied for choosing $\alpha$. First, one should ideally generate a testing ratio such that there exists a valid $(r_l, r_u)$ such that \eqref{eq:ideal rl ru} can be satisfied. This could be done by selecting larger $\alpha$. Second, one can reuse prior function evaluations within the bisection search when monotonically increasing or decreasing $h$ through specific choices of $\alpha$, namely by choosing $\alpha$ as a power of $\eta$. This was done with $\eta = 4$ in the forward-difference algorithm (Algorithm \ref{algo:fd_est_proc}). \\

\noindent\textit{III. Initialization of $h_0$.} Since the difference interval $h$ that satisfies the procedure is approximately of the form \eqref{eq:general opt}, it is preferable to initialize $h_0 = \mathcal{O}(\epsilon_f^{1/q})$. Two possible choices are $h_0 = \epsilon_f^{1/q}$ or $\left(\frac{d}{q - d} \cdot \frac{\|w\|_1}{\abs{c_q}} \cdot \epsilon_f \right)^{1/q}$. The latter is based on the assumption that $|\phi^{(q)}(t)| \approx 1$. If instead the finite-difference interval is re-estimated within an optimization algorithm, we can initialize $h_0$ as the difference interval $h$ used at the prior iteration of the optimization algorithm.

We observe that on rare occasions a poor initial choice of $h_0$ can result in large error in the derivative approximation. This occurs when the initial choice of $h_0$ is too large to capture the local behavior of the function. Reducing the initial interval $h_0$ resolves this issue. \\

\noindent\textit{IV. Handling of Special Cases.} The Taylor expansion analysis elucidates two possible failure cases for our procedure. In particular, observe that
\begin{equation*}
    r_S(h; f, t, \epsilon_f) = \abs{\frac{c_r \phi^{(q)}(t) h^q}{\epsilon_f} + \Delta + o(h^q)}.
\end{equation*}
If $h$ is large (for example, when the noise level $\epsilon_f$ is high), the higher-order terms $o(h^q)$ can dominate the other terms in the Taylor series expansion. This can yield poor estimates of $h$ even if the condition $r_S(h; f, t, \epsilon_f) \in [r_l, r_u]$ is satisfied. In our numerical experiments, we have not found this to be a common issue.

The more common case is when $\phi^{(q)}(t) \approx 0$. In this case, $r_S$ will be dominated by $\Delta$. In this case, $r_S(h; f, t, \epsilon_f) < r_l$ for all $h$ and $h$ will thus monotonically increase until the maximum number of iterations is reached (which we set \texttt{max\_iter} to 20). This occurs, for example, with any $(q - 1)$-th degree polynomial. In this case, the method provides a warning but does not flag this as a failure. Note that in this case, $h$ is a good choice because letting $h^* \rightarrow \infty$ would allow for infinite reduction in the noise. \\

\noindent\textit{V. Sensitivity to $\epsilon_f$.} In practice, one must either have a priori knowledge of the noise level, or estimate the noise level using an uncertainty quantification technique. In the latter case, it is likely that $\epsilon_f$ is misestimated, resulting in larger error in the derivative. Note that if $\epsilon_f$ is overestimated, the algorithm is still guaranteed to terminate in a finite number of iterations, but may generate a larger interval $h > h^*$. On the other hand, if $\epsilon_f$ is underestimated, the algorithm is not guaranteed to terminate in a finite number of iterations, and, if the algorithm succeeds, will likely underestimate $h < h^*$. \\

\subsection{Finite Termination} \label{sec:finite term}

Next, we prove a finite termination theorem for Algorithm \ref{algo:gen_est_proc}. We start by making the following assumptions:
\begin{assum}
There exists an $\epsilon_f \geq 0$ (called the \textit{noise level} of the function) such that
\begin{equation} \label{eq:bounded-noise}
    |\epsilon(t)| \leq \epsilon_f ~~~~~ \forall t \in \mathbb{R}.
\end{equation}
\end{assum}

\begin{assum}\label{ass:bdd_diff_testing_ratio}
The testing ratio $r_S$ satisfies:
\begin{equation*}
    \abs{r_S(h; \phi, t, \epsilon_f) - r_S(h; f, t, \epsilon_f)} \leq 1, ~~ \forall t \in \mathbb{R}, ~ h > 0.
\end{equation*}
\end{assum}
Recall that the testing ratio is defined by \eqref{eq:general test ratio}. This assumption is satisfied by our requirement that $\|\tilde{w}\|_1 = 1$ and that $|\epsilon(t)| \leq \epsilon_f$.

\begin{assum}\label{ass:suff_fin_term}
$r_S(h; \phi, t, \epsilon_f)$ is continuous with $r_S(0; \phi, t, \epsilon_f) = 0$, and there exists an integer $K \in \mb{N}$ such that
\begin{equation*}
    r_S(2^K h_0; \phi, t, \epsilon_f) \geq r_u - 1.
\end{equation*}
\end{assum}
Notice that since $\phi$ is a continuous function, $r_S(h;\phi, t, \epsilon_f)$ is also continuous by its definition. Assuming that $\epsilon_f > 0$, the requirement that $r_S(0; \phi, t, \epsilon_f) = 0$ is also satisfied by validity of the testing ratio \eqref{eq:testing_ratio_taylor}. The last part of Assumption \ref{ass:suff_fin_term}, while technical, is satisfied, for example, when $\abs{\phi^{(q)}(\xi)} \geq \eta > 0$ for all $\xi \in [\min_j \{t + h \tilde{s}_j\}, \max_j \{t + h \tilde{s}_j\}]$. With these assumptions, we can now show finite termination.

\begin{theorem}
Suppose that Assumptions \ref{ass:bdd_diff_testing_ratio} and \ref{ass:suff_fin_term} are satisfied and that $h\geq 0$. In addition, suppose $r_u$ and $r_l$ are chosen such that $0 < r_l < r_u - 2$. Then, Algorithm \ref{algo:gen_est_proc} will terminate successfully in a finite number of iterations.
\end{theorem}

\begin{proof}
Assume by contradiction that Algorithm \ref{algo:gen_est_proc} does not terminate finitely. We denote the variables $l, u, h$ used \emph{at the beginning of} the $k$-th iteration of Algorithm \ref{algo:gen_est_proc} as $l_k, u_k, h_k$, respectively. Obviously, we have
\begin{equation*}
    0 \leq l_k \leq h_k \leq u_k, ~ \forall k \in \mb{N},
\end{equation*}
and
\begin{equation*}
    l_k \leq l_{k+1} < u_{k+1} \leq u_k , ~ \forall k \in \mb{N}.
\end{equation*}

First, we show that $r_S(l_k; \phi, t, \epsilon_f) < r_l + 1$ for all $k \in \mb{N}$, by induction on $k$. Clearly this is true for $k=0$ since $l_0 = 0$, and we have $r_S(0; \phi, t, \epsilon_f) = 0$ by Assumption \ref{ass:suff_fin_term}. Suppose the statement holds for $k \leq K$. We have two cases: (1) $r_S(h_K; f, t, \epsilon_f) < r_l$, which by Assumption \ref{ass:bdd_diff_testing_ratio} implies $r_S(h_K; \phi, t, \epsilon_f) \leq r_S(h_K; f, t, \epsilon_f) + 1 < r_l + 1$. In this case $l_{K+1} = h_K$, so $r_S(l_{K+1}; \phi, t, \epsilon_f) = r_S(h_K; \phi, t, \epsilon_f) < r_l + 1$. (2) $r_S(h_K; f, t, \epsilon_f) > r_u$, in which case $l_{K+1} = l_{K}$ so by the induction hypothesis $r_S(l_{K+1}; \phi, t, \epsilon_f) = r_S(l_{K}; \phi, t, \epsilon_f) < r_l + 1$. Therefore the induction hypothesis holds for $(K+1)$-th iteration.

By a similar argument, we can show that either $u_k = +\infty$, or $u_k < +\infty$ and $r_S(u_k; \phi, t, \epsilon_f) > r_u - 1$ for all $k \in \mb{N}$.

In summary, we can show that for all $k \in \mb{N}$, we have
\begin{align}
    \text{either}~~& r_S(l_{k}; \phi, t, \epsilon_f) < r_l + 1 < r_u - 1 < r_S(u_{k}; \phi, t, \epsilon_f), \\ 
    ~~\text{or}~~& r_S(l_{k}; \phi, t, \epsilon_f) < r_l + 1 ~~\text{and}~~u_k = +\infty.
\end{align}

Next, we claim that there exists $K_1 \in \mb{N}$ such that $u_k < +\infty$ for $k \geq K_1$. Suppose this is not the case, then we have $r_S(h_k; f, t, \epsilon_f) < r_l, ~\forall k \in \mb{N}$. In this case, we have $h_{k+1} = 2l_{k+1} = 2h_k$, so $h_k = 2^k h_0$ for all $k \in \mb{N}$. By Assumption \ref{ass:suff_fin_term}, there exists $K \in \mb{N}$ such that $r_S(h_K; \phi, t, \epsilon_f) \geq r_u - 1$, and since $r_S(h_K; f, t, \epsilon_f) \geq r_S(h_K; \phi, t, \epsilon_f) - 1$, we have $r_S(h_K; f, t, \epsilon_f) \geq r_u - 2 > r_l$, contradicting the inequality $r_S(h_k; f, t, \epsilon_f) < r_l, ~\forall k \in \mb{N}$. This proves the existence of $K_1$.

We are now ready to present the contradiction. For $k \geq K_1$, since $u_k < \infty$, we have
\begin{equation*}
    u_{k+1} - l_{k+1} = \frac{1}{2}\bpa{u_k - l_k}
\end{equation*}
This implies that $u_k - l_k \to 0$. Since $r_S(h; \phi, t, \epsilon_f)$ (as a function of $h$) is continuous and $u_{K_1} < +\infty$, $[0, u_{K_1}]$ is compact so $r_S(h; \phi, t, \epsilon_f)$ (as a function of $h$) is \emph{uniformly continuous} on $[0, u_{K_1}]$. Note that $l_k, u_k \in [0, u_{K_1}]$ for $k \geq {K_1}$, therefore we have
\begin{equation*}
    r_S(u_k; \phi, t, \epsilon_f) - r_S(l_k; \phi, t, \epsilon_f) \to 0
\end{equation*}
This contradicts the fact that 
\begin{equation*}
    r_S(l_{k}; \phi, t, \epsilon_f) < r_l + 1 < r_u - 1 < r_S(u_{k}; \phi, t, \epsilon_f), ~ \forall k \in \mb{N}, ~ k\geq K_1
\end{equation*}
Therefore, Algorithm \ref{algo:gen_est_proc} must terminate finitely. Clearly, whenever it terminates, the output $h_R$ must satisfy 
\begin{equation*}
    r_S(h_R; f, t, \epsilon_f) \in [r_l, r_u].
\end{equation*}
\end{proof}

\section{Numerical Experiments} \label{sec:experiments}

In this section, we present numerical results dem\hyp{}onstrating the reliability of our finite-difference interval estimation procedure. We first utilize the method for computing first and second derivatives ($d=1,2$) of commonly tested functions, with added noise. We then insert our procedure into a standard L-BFGS implementation and illustrate its performance on a subset of synthetic noisy \texttt{CUTEst} problems \cite{gould2015cutest}. All methods were implemented in Python~3.

\subsection{Finite-Difference Interval Estimation} 

We begin by testing our proposed procedure on several univariate functions. We focus on the case where $d=1$ and $2$ as this is most relevant to optimization. We  test Algorithm \ref{algo:gen_est_proc} using 6 different estimating schemes, shown in Table \ref{tab:schemes_tested}. The testing ratios are generated using formula \eqref{eq:gen_ratio} with different choices of $\alpha$. The $\alpha$ for each scheme is chosen as the smallest integer such that $r^* > \beta = 2$.

\begin{table}[tbhp]
{\footnotesize
  \caption{Schemes for approximating the $d$-th order derivative used in the experiments. The scheme is defined by $S = (w, s)$ as in \eqref{eq:scheme_def_w_s}; $q$ is defined in \eqref{eq:def_error_order_q}.}\label{tab:schemes_tested}
\begin{center}
\begin{tabular}{cccccccl} \hline
label & $d$ & $s$ & $w$ & $q$ & $\alpha$ & $r^*$ & Comment \\ \hline
	\texttt{FD} & 1 & $(0, 1)$ & $(-1, 1)$ & 2 & $4$ & $3$ & FD \\
	\texttt{CD} & 1 & $(-1, 1)$ & $(-1/2, 1/2)$ & 3 & 3 & $3$ & CD \\
	\texttt{FD\_3P} & 1 & $(0, 1, 2)$ & $(-3/2, 2, -1/2)$ & 3 & 3 & $3.69$ & FD w/ 3 points \\
	\texttt{FD\_4P} & 1 & $(0, 1, 2, 3)$ & $(-11 / 6, 3, - 3 / 2, 1 / 1)$ & 4 & 3 & $8.25$ & FD w/ 4 points \\
	\texttt{CD\_4P} & 1 & $(-2, -1, 1, 2)$ & $(1 / 12, -2 / 3, 2 / 3, -1 / 12)$ & 5 & 2 & $2.5$ & CD w/ 4 points \\
	\texttt{L2\_CD} & 2 & $(-1, 0, 1)$ & $(1, -2, 1)$ & 4 & 2 & $3$ & 2nd-order CD \\ \hline
\end{tabular}
\end{center}
}
\label{tab:schemes}
\end{table}

For a specific testing function $\phi$ at point $t$, with noise $\epsilon_f$, and  scheme $S = (w, s)$, we define the \emph{worst case relative error} induced by $h$, as:
\begin{equation} \label{eq:worst rel error}
	\delta_S(h; \phi, t, \epsilon_f) = \frac{1}{\abs{\phi^{(d)}(t)}} \bbr{\abs{\frac{\sum_{j=1}^p w_j \phi(t + s_j h)}{h^d} - \phi^{(d)}(t)} + \|w\|_1 \frac{\epsilon_f}{h^d}}.
\end{equation}
This function captures the worst case relative error of the estimation scheme $S$ on the noisy function $f$ at $t$, for a given differencing interval $h$. The value of $h$ that minimizes $\delta_S(h; \phi, t, \epsilon_f)$ is the optimal $h$. Notice that $\delta_S(h; \phi, t, \epsilon_f)$ is a deterministic function that does not rely on the realization of actual noise in $f(t)$.

In the experiments reported below, we manually inject uniformly distributed, stochastic noise into $\phi$, 
\begin{equation*}
	f(t) = \phi(t) + \epsilon(t), ~~ \epsilon(t) \sim \text{Uniform}(-\epsilon_f, \epsilon_f),
\end{equation*}
independent of all other quantities.
We let $h_\dagger$ denote the output of Algorithm~\ref{algo:gen_est_proc}.

\subsubsection{Robustness to Different Noise Levels}

We test our adaptive procedure for both forward and central differences ({\tt FD, CD}) on the simple function $\phi(t) = \cos(t)$ for different noise levels, at $t = 1.0$. The plots of the worst case relative error $\delta_S(h; \phi, t, \epsilon_f)$ and the obtained intervals $h_{\dagger}$ are illustrated in Figure \ref{fig:cos_func_noise_level} for a range of noise levels $\epsilon_f$. The dots represent the worst case relative error $\delta_S(h; \phi, t, \epsilon_f)$ at $h_\dagger$. Our results demonstrate that our method performs consistently well across a range of noise levels, since $h_{\dagger}$ approximately minimizes $\delta_S(h; \phi, t, \epsilon_f)$ in all cases. Results for
all other finite-difference schemes listed in Table \ref{tab:schemes_tested} are reported the Appendix \ref{app:experiments}.

\begin{figure}[tbhp]
    \centering
    \includegraphics[width=0.49\linewidth]{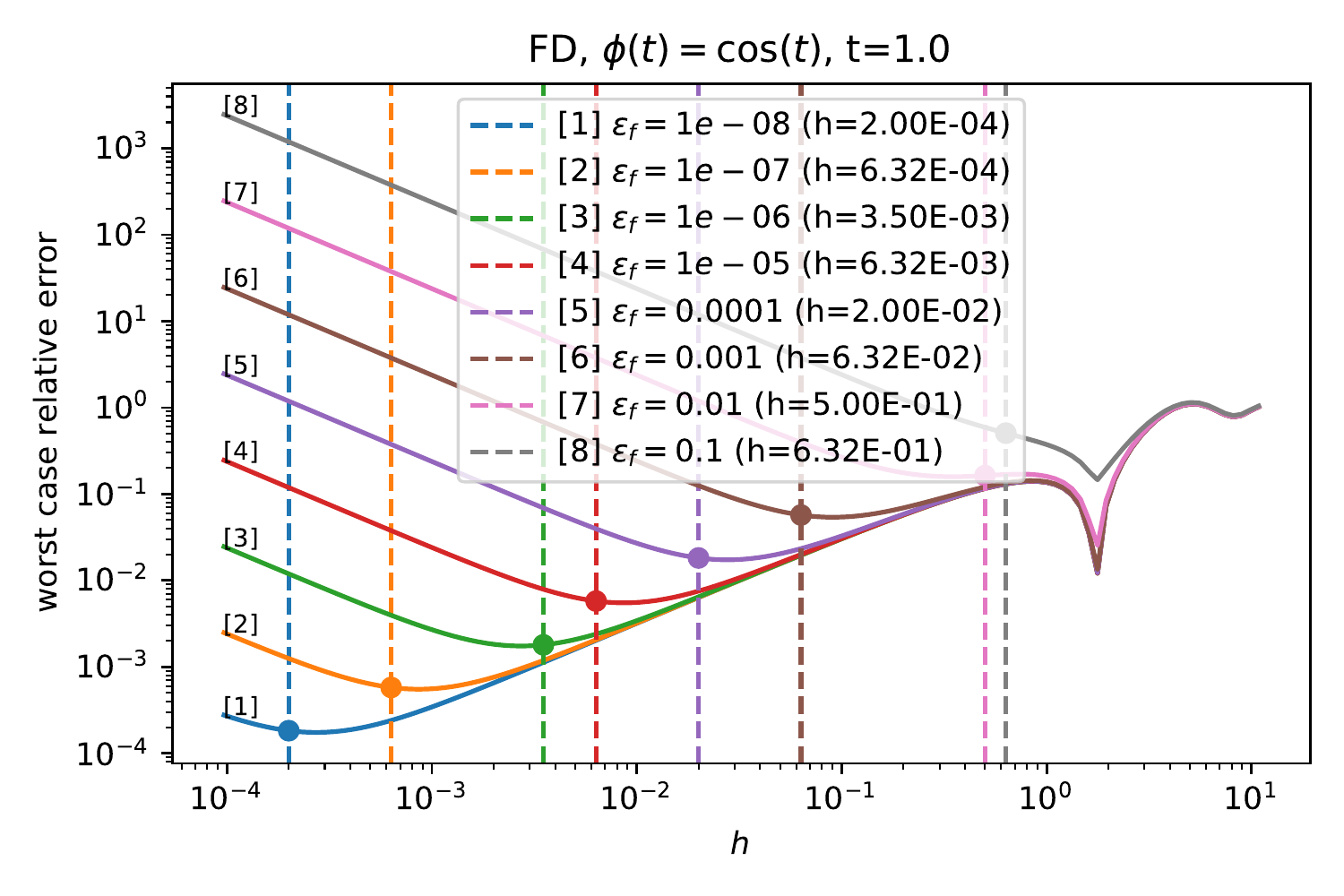}
	\includegraphics[width=0.49\linewidth]{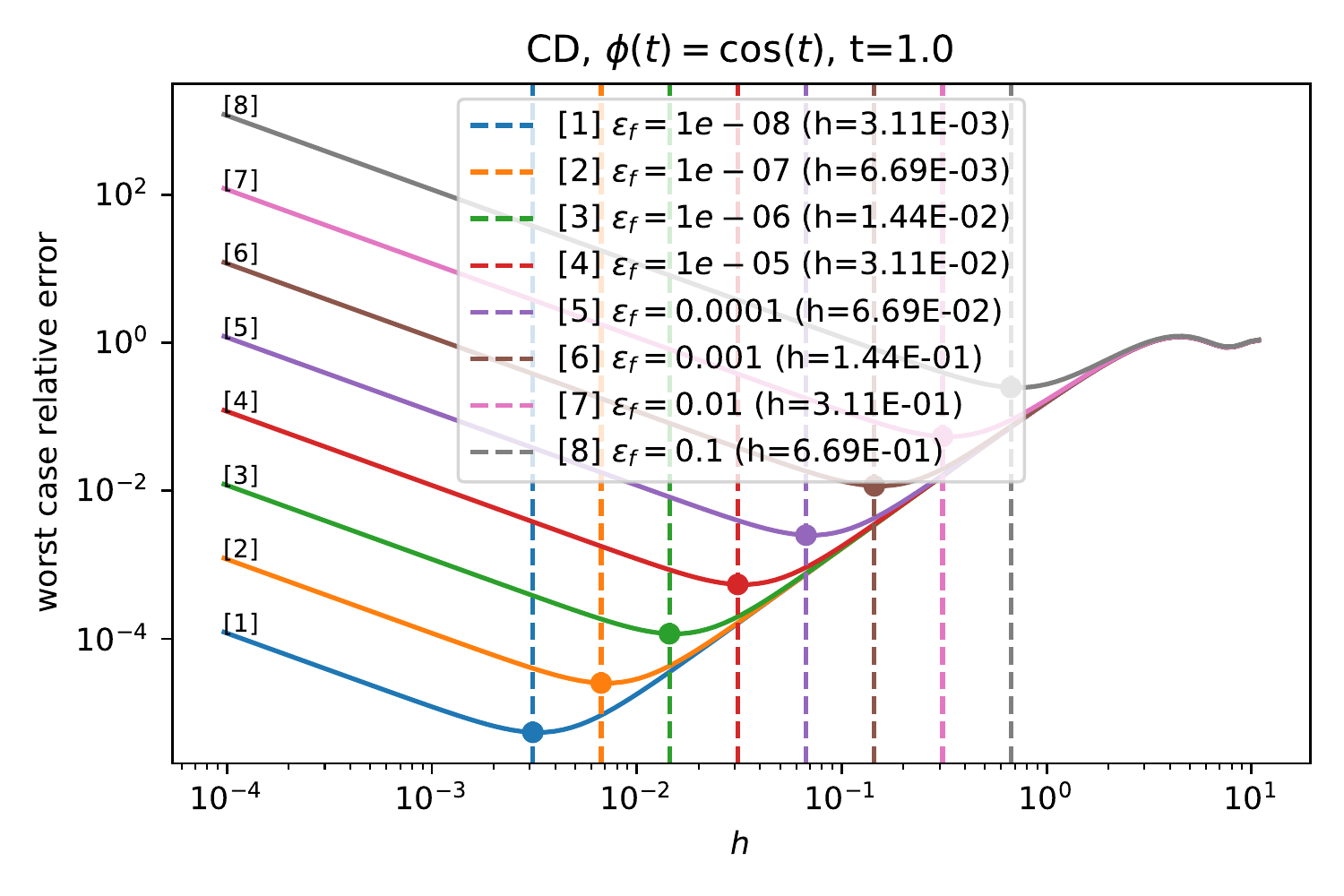}
    \\
    \caption{Worst Case Relative Error Plot for Forward Differences (Left) and Central Difference (Right). We plot $\delta_S(h; \phi, t, \epsilon_f)$ against $h$ on function $\phi(t) = \cos(t)$, at $t = 1$, for $\epsilon_f$ varying from $10^{-8}$ to $10^{-1}$. Each vertical dashed line represents the $h_{\dagger}$ output by Algorithm \ref{algo:gen_est_proc}. For each curve, the dots represent $\delta_S(h; \phi, t, \epsilon_f)$ evaluated at $h = h_{\dagger}$. Note how close are the dots to the minimum values of the functions $\delta_S$.}
    \label{fig:cos_func_noise_level}
\end{figure}

\subsubsection{Difficult and Special Examples}

In this subsection, we consider the following difficult functions studied in \cite{GillMurrSaunWrig83} and \cite{shi2021numerical}.
\begin{enumerate}
	\item $\phi(t) = \bpa{e^t - 1}^2$, at $t = -8$. This function has extremely small first and second-order derivatives at $t = -8$, but quickly increases as $t$ increases beyond $t = 0$; a naive choice of $h = \sqrt{\epsilon_f / |\phi^{(2)} (t)|}$ for forward differences can result in an extremely large $h$ and lead to huge error. 
	\item $\phi(t) = e^{100 t}$, at $t = 0.01$. This function is difficult because its higher-order derivatives increase rapidly. This can easily lead to inaccurate higher-order approximations, especially when $h$ is chosen to be large. 
	\item $\phi(t) = t^4 + 3 t^2 - 10 t$, at $t = 0.99999$. This function is considered difficult because $\phi^\prime(1) = 0$, and represents a case where the estimated derivative is very close to $0$. In addition, this function is a fourth-order polynomial, so the optimal $h$ for \texttt{CD\_4P} is $+\infty$. 
	\item $\phi(t) = 10000 t^3 + 0.01 t^2 + 5 t$, at $t = 10^{-9}$. This function has approximate central symmetry at $t = 0$, which can lead to issues for adaptive procedures such as those proposed in \cite{GillMurrSaunWrig83}, because a near-optimal interval may fail to satisfy the test.
\end{enumerate}

For each example, we fix $\epsilon_f = 10^{-3}$, and perform our estimation procedure for different schemes to obtain $h_\dagger$. Again, we plot the worst case relative error $\delta_S(h; \phi, t, \epsilon_f)$ and the interval $h_\dagger$. The dots denote the error $\delta_S(h; \phi, t, \epsilon_f)$ at $h_\dagger$. The results can be found in Figure \ref{fig:special}. As demonstrated by the plots, our procedure is able to produce a near-optimal $h_\dagger$ that approximately minimizes $\delta_S(h; \phi, t, \epsilon_f)$ across all estimating schemes for the selected challenging problems.

For the first two examples, we see that our procedure is able to estimate the derivative well even when the function increases rapidly. These are cases where using our adaptive procedure is significantly more effective than computing an interval based on higher-order derivative information at the point of interest, as observed in \cite{shi2021numerical}. 

It is also interesting to observe the results for the two other examples. For $\phi(t) = t^4 + 3 t^2 - 10 t$ and scheme \texttt{CD\_4P}, our procedure generates a large $h_\dagger$; this is consistent with the fact that scheme \texttt{CD\_4P} has $q = 5$, and $\phi^{(5)}(\xi) = 0$ for all $\xi$ on this example, which implies that we should choose $h$ to be as large as possible. This similarly holds true for the schemes \texttt{FD\_4P} and \texttt{CD\_4P} on the function $\phi(t) = 10000 t^3 + 0.01 t^2 + 5 t$. While in theory we should choose $h = \infty$ in such cases, we can observe in Figure \ref{fig:special} that this is not the case. When plotting the worst-case relative error, we see that there exists a large $h$ such that $\delta_S(h; \phi, t, \epsilon_f)$ is minimized, beyond which the relative error begins to sharply increase in a discontinuous manner. This phenomenon is due to round-off error. When $h$ becomes too large, round-off error, which is approximately on the order of $\max_j \abs{\phi(t + s_j h)} \epsilon_M$, will dominate $\epsilon_f$.

\begin{figure}[tbhp]
    \centering
    \includegraphics[width=0.49\linewidth]{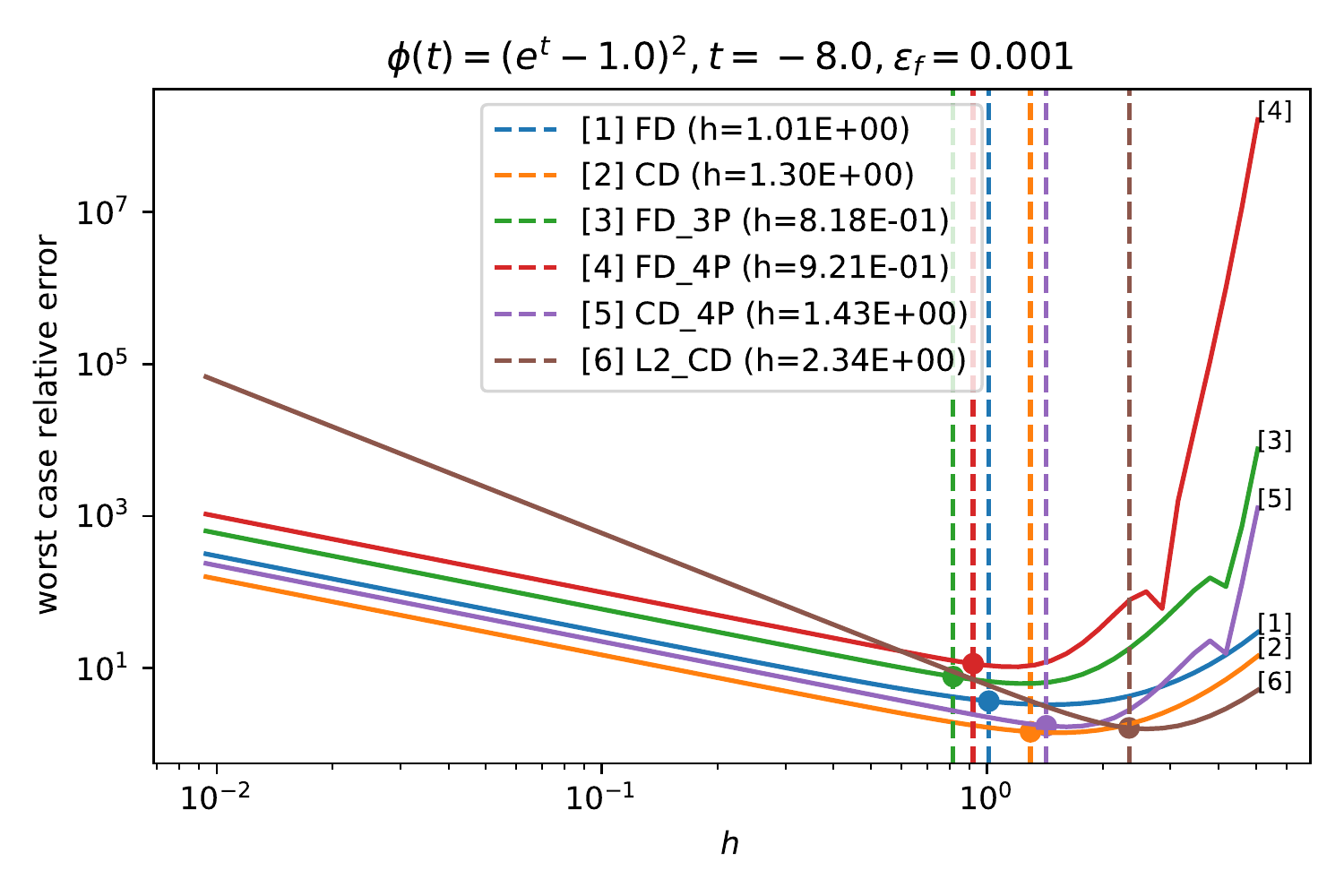}
	\includegraphics[width=0.49\linewidth]{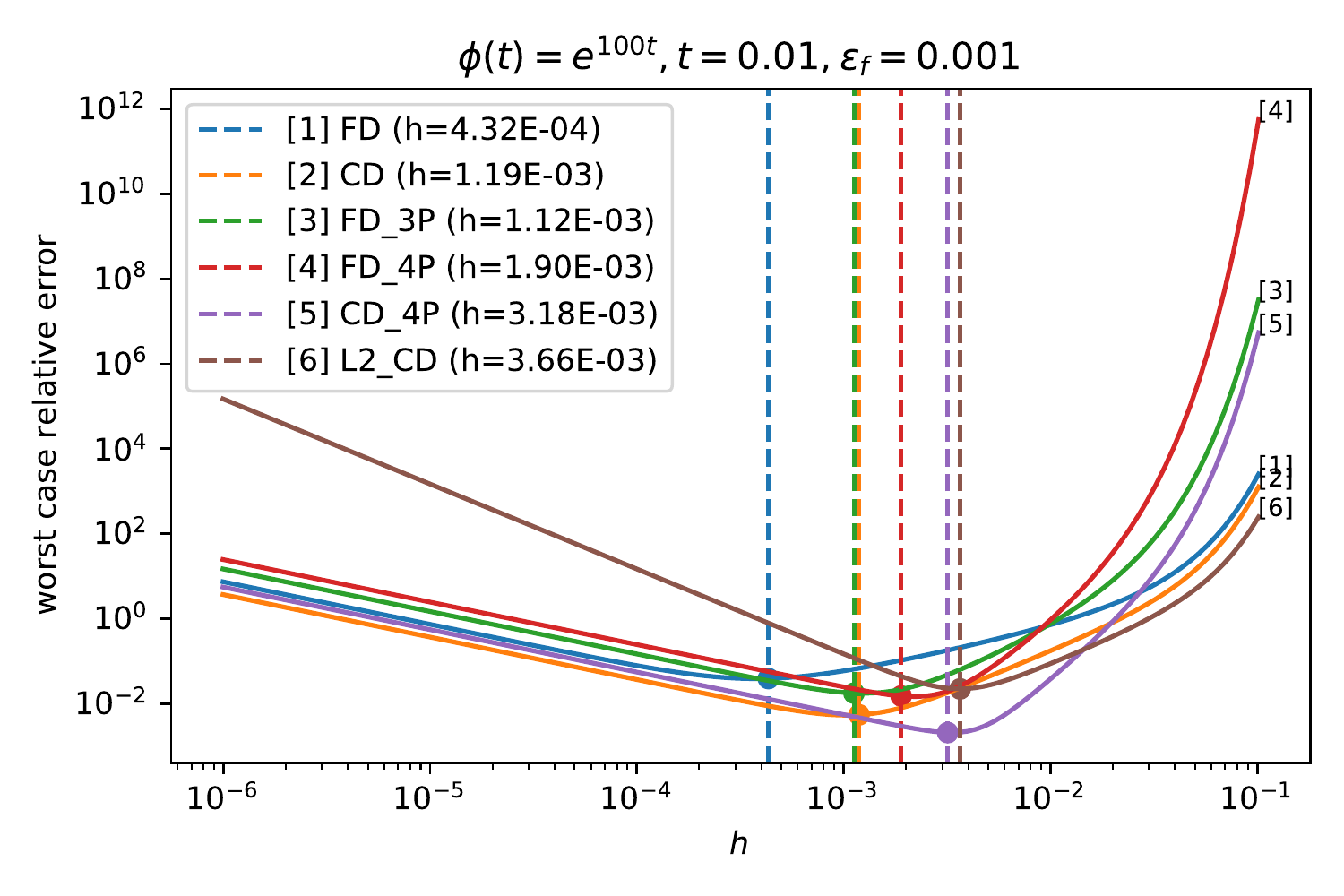}
    \\
    \includegraphics[width=0.49\linewidth]{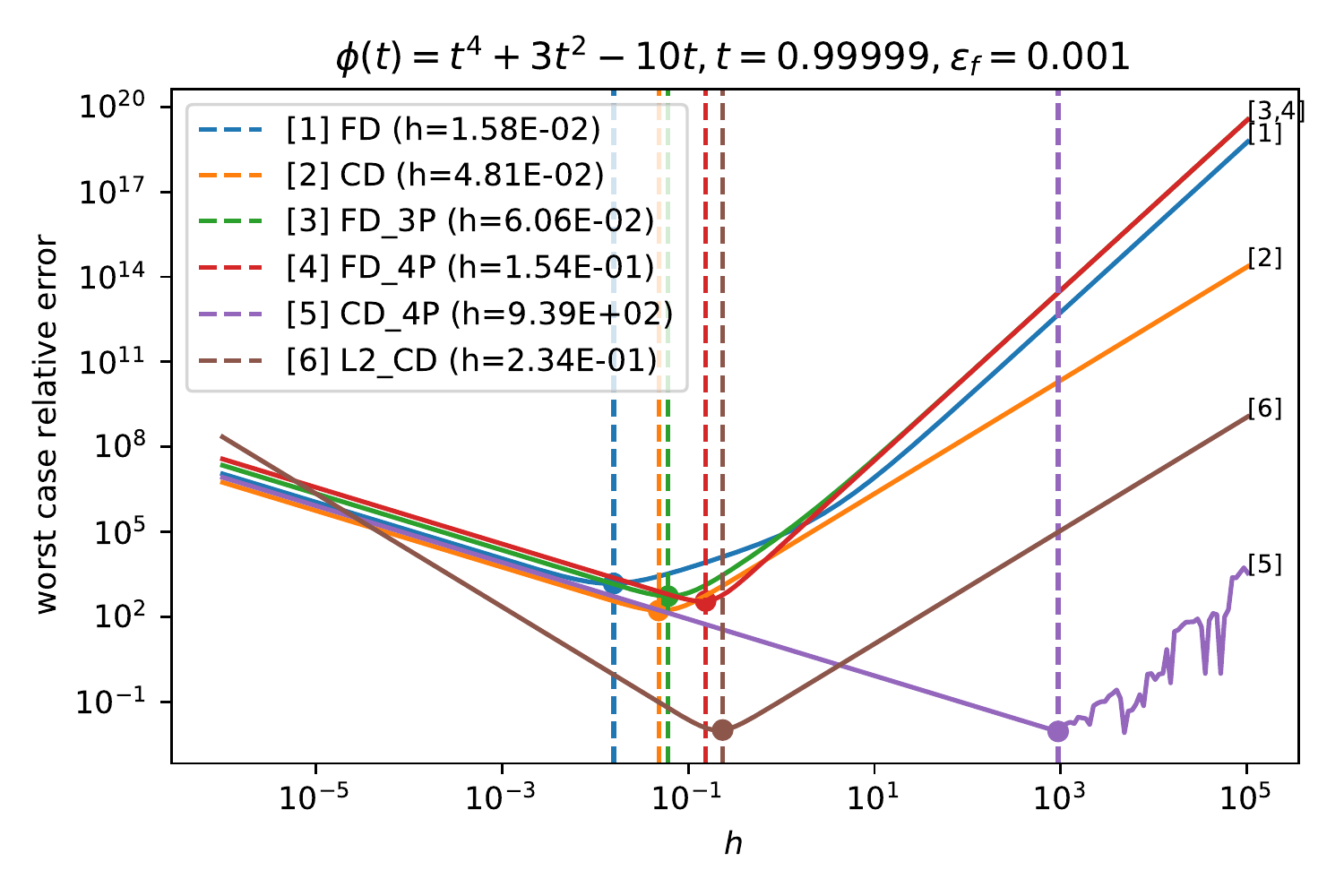}
	\includegraphics[width=0.49\linewidth]{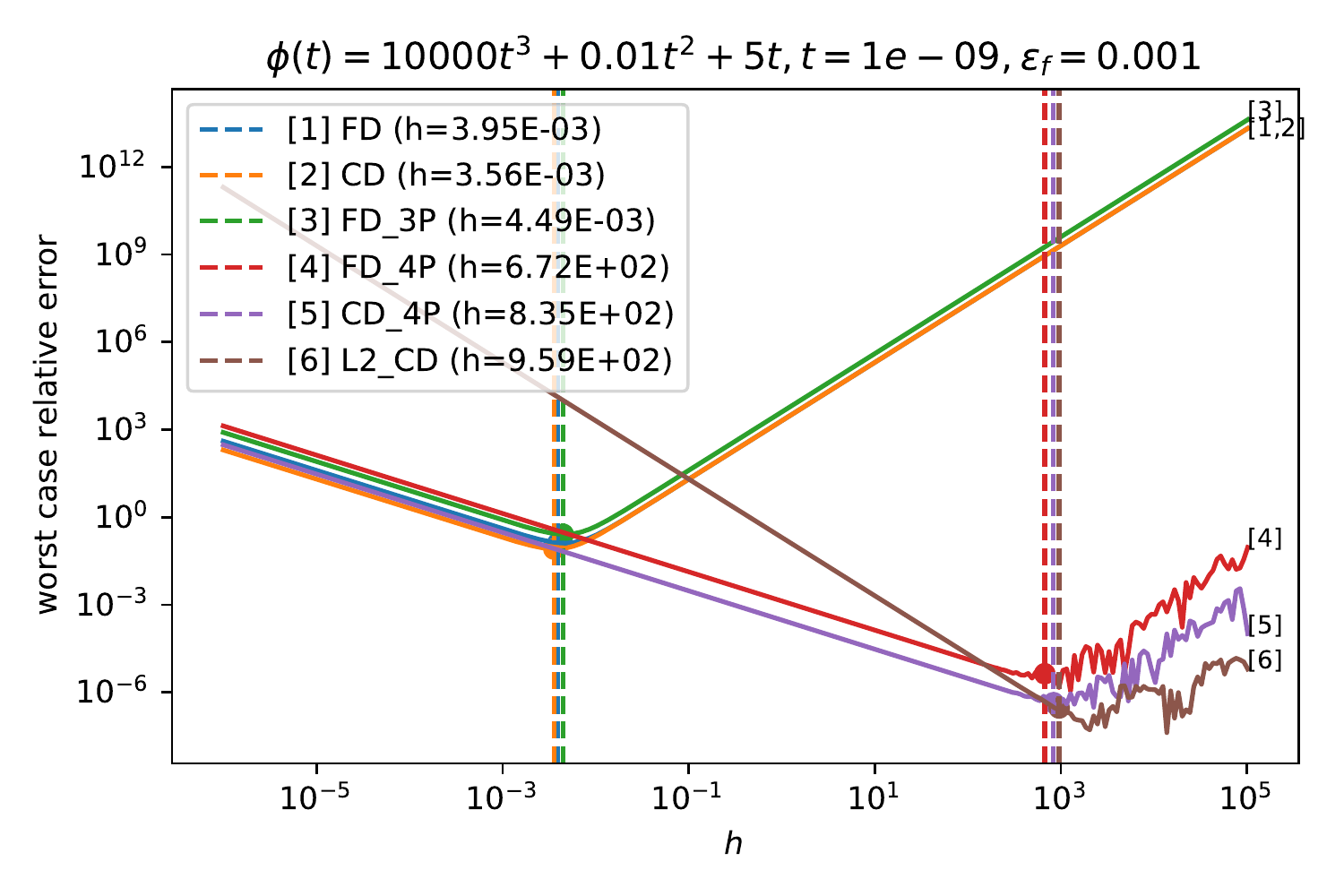}
	\\
	\caption{Worst Case Relative Error Plot for 4 Difficult Examples with $\epsilon_f = 10^{-3}$. We plot $\delta_S(h; \phi, t, \epsilon_f)$ against $h$ on 4 selected functions, for the 6 schemes in Table \ref{tab:schemes}. Each vertical dashed line represents the $h_{\dagger}$ output by Algorithm \ref{algo:gen_est_proc}. For each curve, the dots represent $\delta_S(h; \phi, t, \epsilon_f)$ at $h = h_{\dagger}$ for the corresponding estimating scheme. Note how close are the dots to the minimum values of the functions $\delta_S$. }
    \label{fig:special}
\end{figure}
 
\subsection{Finite-Difference L-BFGS}

We implemented our procedure within the L-BFGS method. We perform  tests on a subset of \texttt{CUTEst} problems \cite{gould2015cutest} detailed in Table \ref{tab:cutest probs}, injecting noise as before, with $\epsilon_f \in \{10^{-1}, 10^{-3}, 10^{-5}, 10^{-7}\}$.
We assume that the noise level $\epsilon_f$ is known to the algorithm and requires no function evaluations.

\begin{table}[tbhp]
	{\footnotesize
	\caption{Subset of unconstrained \texttt{CUTEst} problems and their problem dimensions \cite{gould2015cutest}.}
    \label{tab:cutest probs}
    \begin{center}
	\begin{tabular}{lrlrlrlr}
		\hline
        Problem & Dim ($n$) & Problem & Dim ($n$) & Problem & Dim ($n$) & Problem & Dim ($n$)\\
        \hline
        \texttt{AIRCRFTB} & 5 & \texttt{CRAGGLVY} & 100 & \texttt{FREUROTH} & 100 & \texttt{PFIT4LS} & 3 \\
        \texttt{ALLINITU} & 4 & \texttt{CUBE} & 2 & \texttt{GENROSE} & 100 & \texttt{QUARTC} & 100 \\
        \texttt{ARWHEAD} & 100 & \texttt{DENSCHND} & 3 & \texttt{GULF} & 3 & \texttt{SINEVAL} & 2 \\
        \texttt{BARD} & 3 & \texttt{DENSCHNE} & 3 & \texttt{HAIRY} & 2 & \texttt{SINQUAD} & 100 \\
        \texttt{BDQRTIC} & 100 & \texttt{DIXMAANH} & 90 & \texttt{HELIX} & 3 & \texttt{SISSER} & 2 \\
        \texttt{BIGGS3} & 3 & \texttt{DQRTIC} & 100 & \texttt{NCB20B} & 100 & \texttt{SPARSQUR} & 100 \\
        \texttt{BIGGS5} & 5 & \texttt{EDENSCH} & 36 & \texttt{NONDIA} & 100 & \texttt{TOINTGSS} & 100 \\
        \texttt{BIGGS6} & 6 & \texttt{EIGENALS} & 110 & \texttt{NONDQUAR} & 100 & \texttt{TQUARTIC} & 100\\
        \texttt{BOX2} & 2 & \texttt{EIGENBLS} & 110 & \texttt{OSBORNEA} & 5 & \texttt{TRIDIA} & 100 \\
        \texttt{BOX3} & 3 & \texttt{EIGENCLS} & 30 &  \texttt{OSBORNEB} & 11 & \texttt{WATSON} & 31 \\
        \texttt{BRKMCC} & 2 & \texttt{ENGVAL1} & 100 &  \texttt{PENALTY1} & 100 & \texttt{WOODS} & 100 \\
        \texttt{BROWNAL} & 100 & \texttt{EXPFIT} & 2 & \texttt{PFIT1LS} & 3 & \texttt{ZANGWIL2} & 2 \\
        \texttt{BROWNDEN} & 4 & \texttt{FLETCBV3} & 100 & \texttt{PFIT2LS} & 3 \\
        \texttt{CLIFF} & 2 & \texttt{FLETCHBV} & 100 & \texttt{PFIT3LS} & 3 \\
        \hline
    \end{tabular}
    \end{center}
    }
\end{table}

The L-BFGS method has the form
\begin{equation}
    x_{k + 1} = x_k - \alpha_k H_k g(x_k),
\end{equation}
where $g(x_k)$ is a finite-difference approximation to the gradient, $H_k$ is the L-BFGS matrix with memory of 10 (see \cite{mybook}), and  $\alpha_k$ is a steplength selected by a relaxed Armijo-Wolfe line search designed to handle noise. 

To describe the line search, let $\alpha_k^j$ denote the $j$th trial steplength at iteration $k$. Similar to Shi et al.\cite{shi2020noise}, the Armijo condition is relaxed as follows:
\begin{equation}
    f(x_k + \alpha_k^j p_k)
    \begin{cases}
    \leq f(x_k) + c_1 \alpha_k^j g(x_k)^T p_k & \mbox{ if } j = 0, \, g(x_k)^T p_k < - \epsilon_g(x_k) \|p_k\| \\
    \leq f(x_k) + c_1 \alpha_k^j g(x_k)^T p_k + 2 \epsilon_f & \mbox{ if } j \geq 1,  \, g(x_k)^T p_k < - \epsilon_g(x_k) \|p_k\| \\
    < f(x_k) & \mbox{ if } g(x_k)^T p_k \geq - \epsilon_g(x_k) \|p_k\|,
    \end{cases}
\end{equation}
where $c_1 = 10^{-4}$ and $\epsilon_g(x_k)$ is the estimated gradient error described below. Thus, we relax the line search only when the gradient is reliable; otherwise, we enforce simple decrease. We test the Wolfe condition,
\begin{equation}
    \nabla f(x_k + \alpha_k^j p_k)^Tp_k \geq c_2  \nabla f(x_k)^Tp_k, \qquad c_2 = 0.9,
\end{equation}
by estimating  directional derivatives using finite differences along $p_k$, as in \cite{shi2020noise}.

\subsubsection{Forward Differences}

In the first set of experiments, the gradient approximation $g(x_k)$ is obtained by forward differences, 
\[   [g(x_k)]_i  = \frac{f(x + h_i e_i) - f(x)}{h_i} , \quad i=1, \ldots n,
\]
where  $h_i$ is determined by one of the following three strategies.

\medskip
1. {\tt Fixed}. The interval $h$ is fixed across all components $i$ and all iterations. This strategy tries to emulate the common practice of hand-tuning $h_i$ at the start using  problem specific information. We simulate this using the formula
\begin{equation}
h = 2 \sqrt{\frac{\epsilon_f}{L_2}}, ~~~~~ \mbox{where } L_2 = \max\left\{10^{-1}, \sqrt{\frac{\sum_{i = 1}^n [\nabla^2 \phi(x_0)]_{ii}^2}{n}}\right\},
\end{equation}
which assumes that the diagonals of the Hessian are known.
The gradient error is approximated assuming $L_2$ is correct, that is, $\epsilon_g(x) = 2 \sqrt{n L_2 \epsilon_f}$. We created this idealized option for benchmarking purposes only.

\medskip
2. {\tt MW}. The Mor\'e-Wild heuristic for estimating and interval $h_i$ for every component $i$ of the gradient \cite{more2012estimating}. We set $L_2 = \max\{10^{-1}, \hat{L}_2 \}$, where $\hat{L}_2$ is the estimate given by the Mor\'e and Wild heuristic. If the heuristic fails, we set $L_2 = 10^{-1}$. We employ the maximum of $10^{-1}$ and $\hat{L}_2$ to safeguard against failures caused by small $L_2$ estimations. Other threshold values are possible but we have found $10^{-1}$ to work the best for this set of test problems. The gradient error is estimated similar to {\tt Fixed} but componentwise, i.e., $\epsilon_g(x) = 2 \sqrt{\epsilon_f \sum_{i = 1}^n L_{2, i}}$.

\medskip
3. {\tt Adaptive} Our adaptive procedure for estimating $h_i$ along each component using Algorithm~\ref{algo:fd_est_proc}.

\medskip
We  chose not to compare against Gill et al. \cite{GillMurrSaunWrig83} as we regard the Mor\'e-Wild ({\tt MW}) heuristic to be an improvement over their approach. For the {\tt MW} and {\tt Adaptive} strategies, we  re-estimate the second derivative or finite-difference interval whenever a partial derivative needs to be approximated. For example, when computing the full gradient, we  estimate the finite-difference interval along each coordinate direction separately. 

We present results for a few representative problems in Figure \ref{fig:fd l-bfgs}. The L-BFGS method described above is terminated if no further progress is made on the objective function over 5 consecutive iterations. Figure \ref{fig:fd l-bfgs} plots the optimality gap $\phi(x_k) - \phi^*$ against the number of function evaluations. The optimal value $\phi^*$ is obtained by solving the original problem to completion without noise with L-BFGS. 

While we found the {\tt MW} strategy to work well for $\epsilon_f < 10^{-1}$, this heuristic fails frequently for the case where $\epsilon_f = 10^{-1}$. (This can be seen in the complete results presented in the Appendix \ref{app:experiments}.) For this reason, we report results for $\epsilon_f = 10^{-1}$ and $10^{-5}$, to demonstrate the robustness of our algorithm compared to {\tt MW} heuristic, for different noise levels. When the {\tt MW} heuristic succeeds, we observe that our algorithm ({\tt Adaptive})  is able to more efficiently achieve comparable accuracy to the {\tt MW} heuristic, while attaining more accurate solutions than using a fixed interval for some problems. The lack of accuracy in the {\tt Fixed} strategy can be explained by the inability for a fixed interval to adapt to changes in the Hessian over the course of the iteration --- an exception being the TRIDIA problem, which is very well scaled.

\begin{figure}[tbhp]
    \centering
    \includegraphics[width=0.24\textwidth]{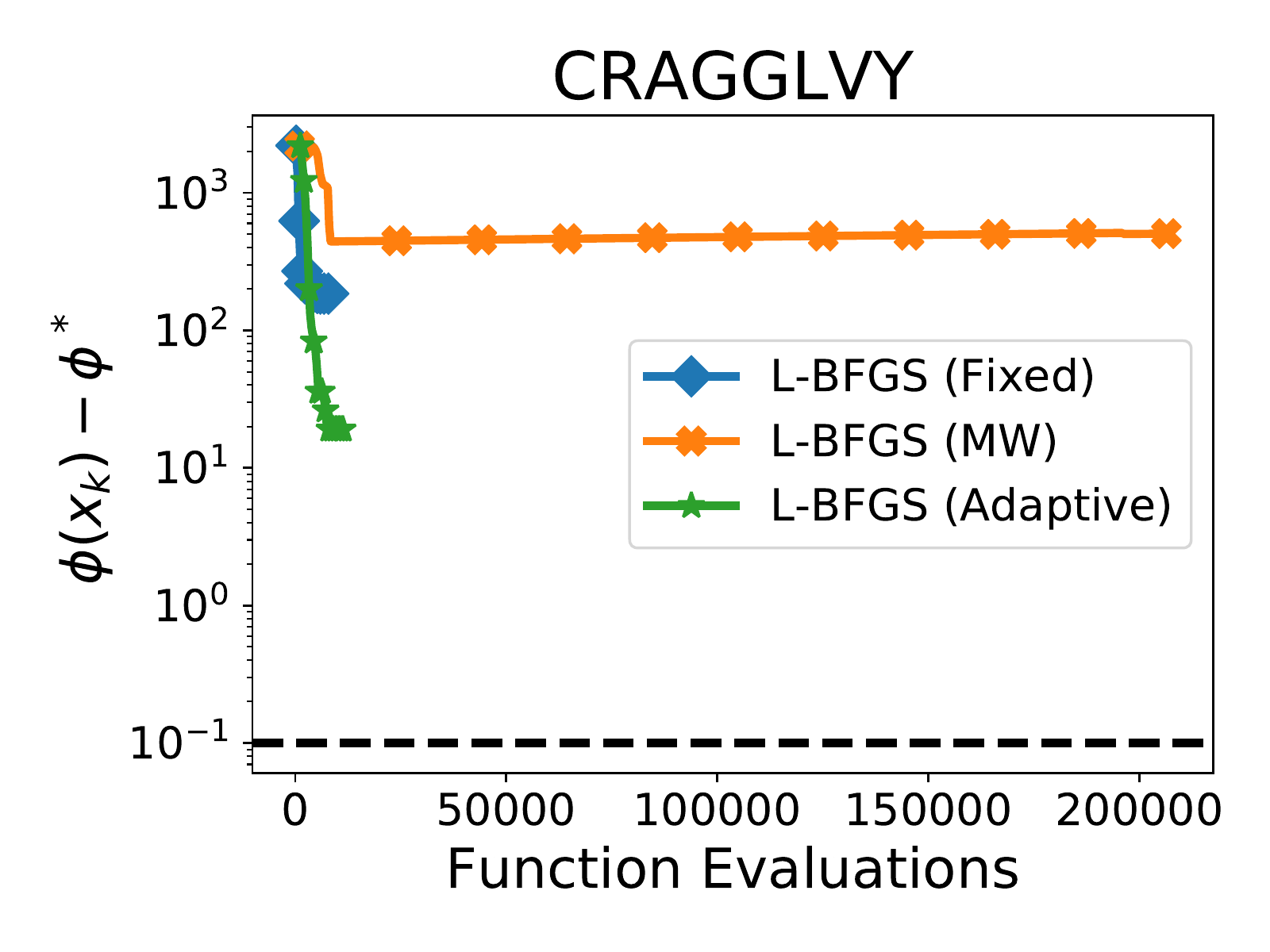}
    \includegraphics[width=0.24\textwidth]{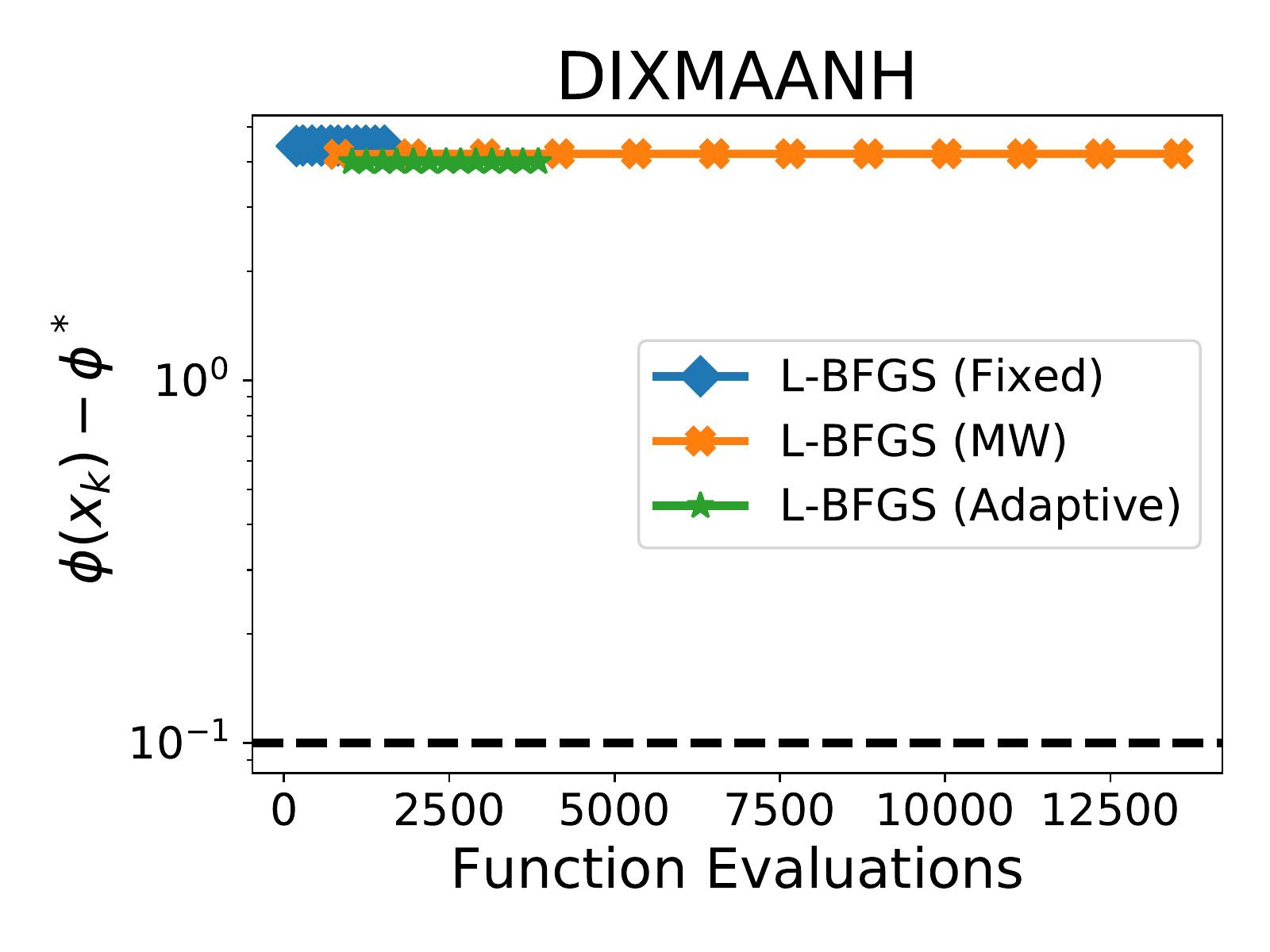}
    \includegraphics[width=0.24\textwidth]{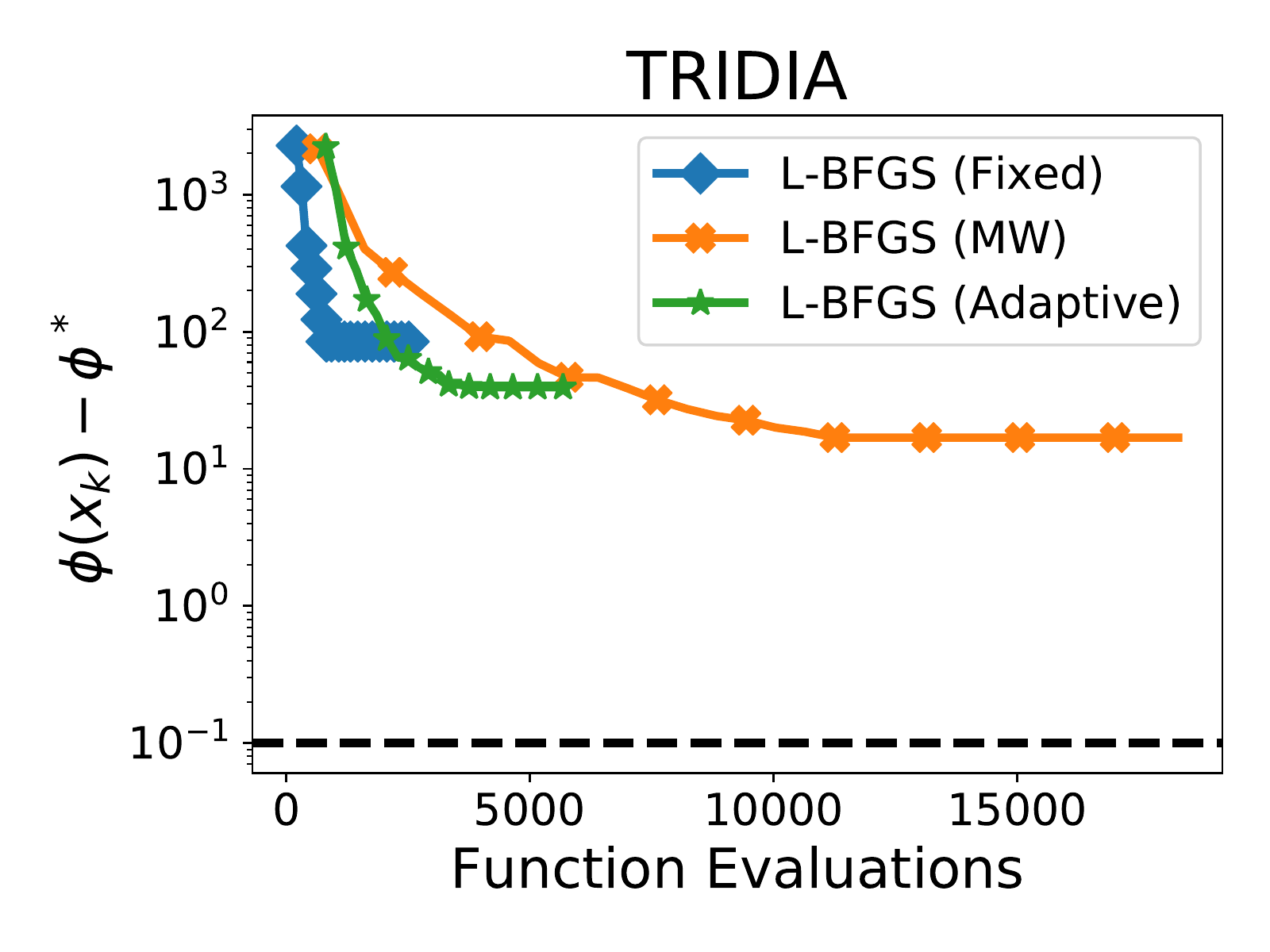}
    \includegraphics[width=0.24\textwidth]{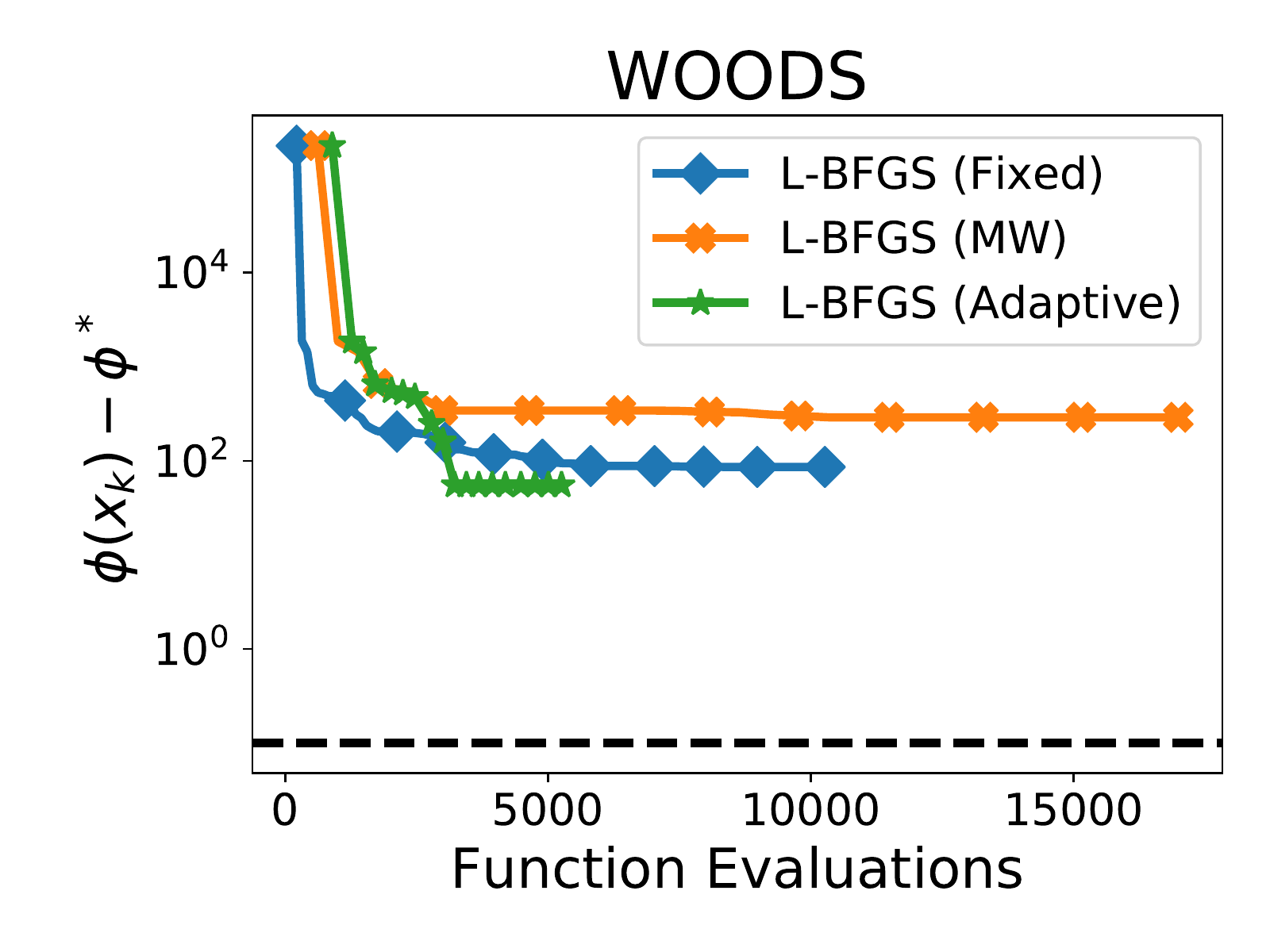}
    \includegraphics[width=0.24\textwidth]{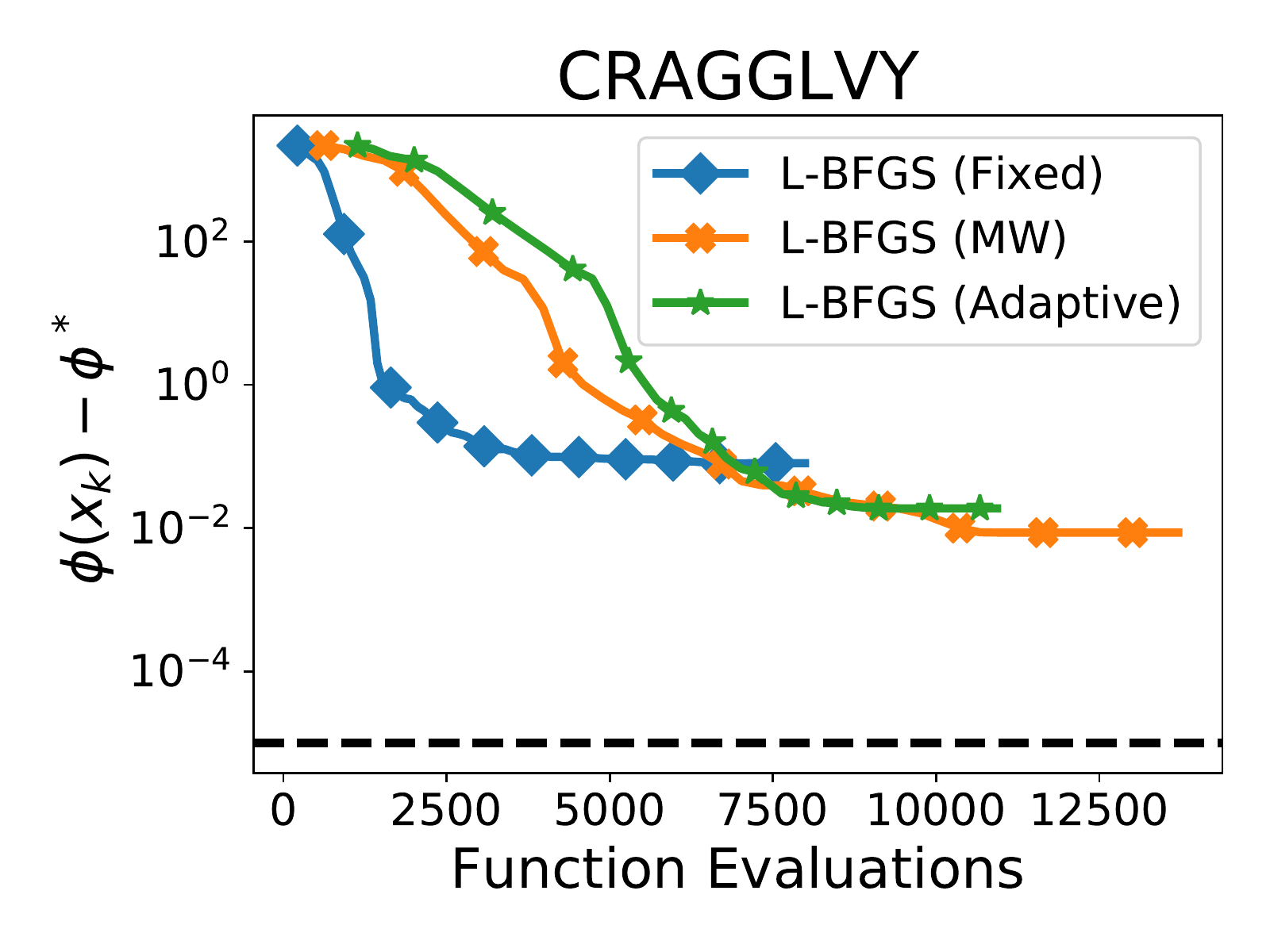}
    \includegraphics[width=0.24\textwidth]{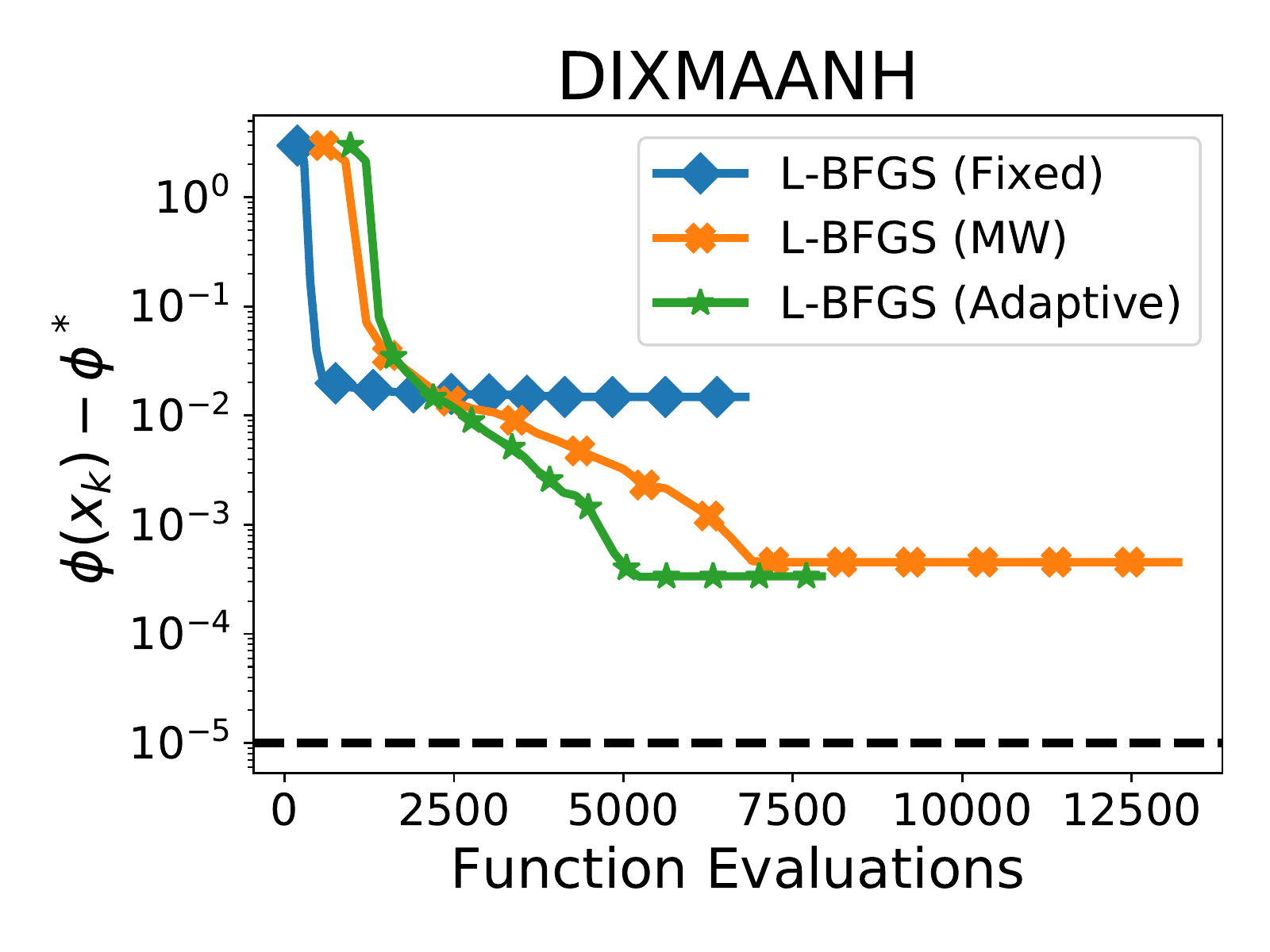}
    \includegraphics[width=0.24\textwidth]{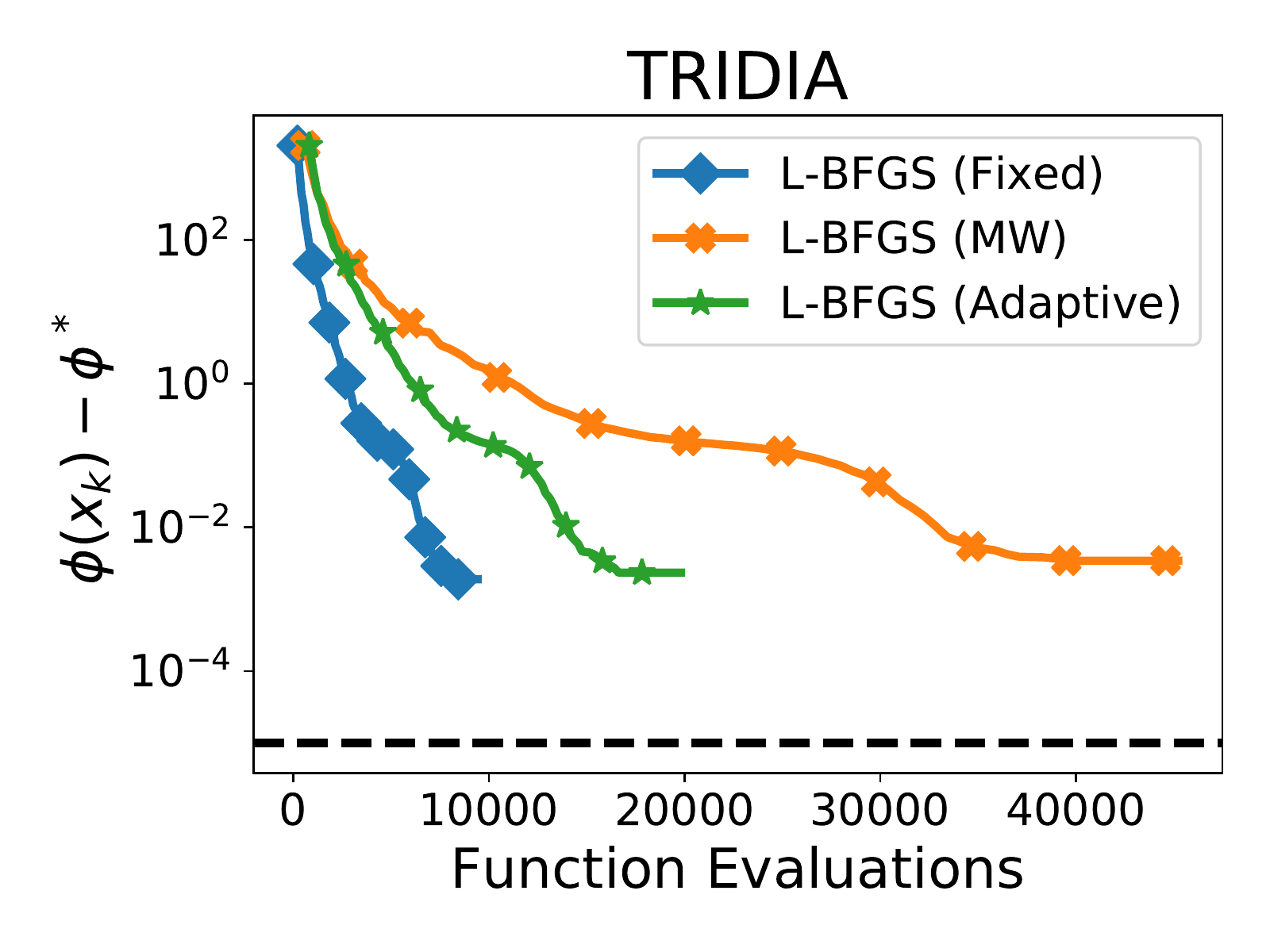}
    \includegraphics[width=0.24\textwidth]{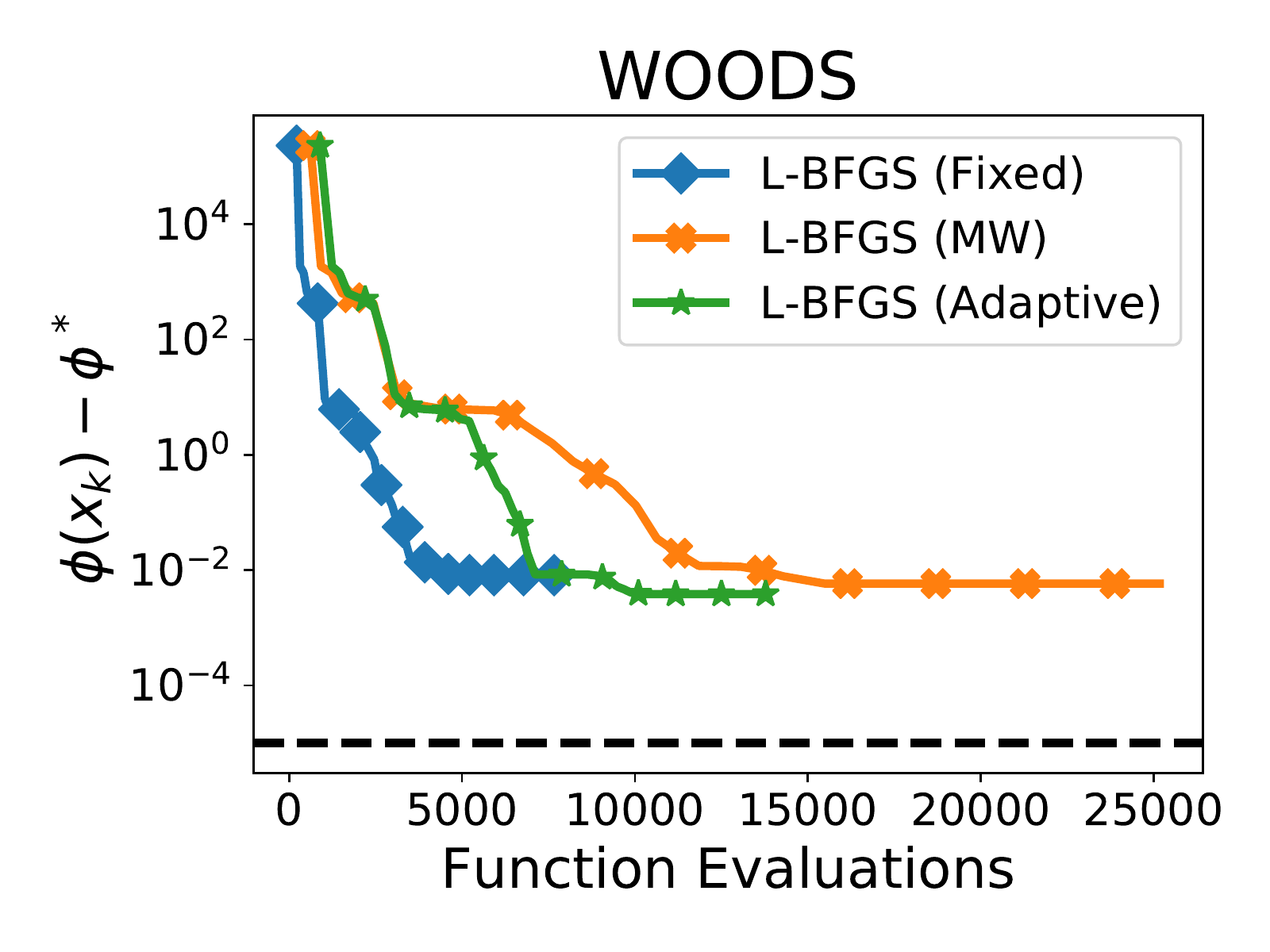}
    \caption{Comparison of forward-difference L-BFGS methods with difference intervals determined using a fixed interval, the Mor\'e and Wild heuristic, and our adaptive algorithm. Comparisons are made on representative problems with noise level $\epsilon_f = 10^{-1}$ (top) and $10^{-5}$ (bottom). We plot the true function value against the function evaluations. The dashed black line shows the noise level $\epsilon_f$ of the function.}
    \label{fig:fd l-bfgs}
\end{figure}

\subsubsection{Central Differences}

In the second set of experiments, we employ central differences,
\[ [g(x; h)]_i  = \frac{f(x + h_i e_i) - f(x - h_i e_i)}{2 h_i}.
\]
The differencing interval is determined via a {\tt Fixed} strategy or the {\tt Adaptive} procedure described in Algorithm~\ref{algo:gen_est_proc}. (The Mor\'e-Wild heuristic only offers estimates of the second order derivatives and thus does not apply to this case.) For the {\tt Fixed} strategy, we choose 
\begin{equation}
    h = \sqrt[3]{\frac{3 \epsilon_f}{L_3}}, ~~~ \mbox{where } L_3 = \max\left\{10^{-1}, \sqrt{\frac{1}{n} \sum_{i = 1}^n \left(\frac{[\nabla^2 \phi(x_0 + \tilde{h} e_i)]_{ii} - [\nabla^2 \phi(x_0)]_{ii} }{\tilde{h}}\right)^2}\right\}
\end{equation}
and $\tilde{h} = \max\{1, |[x_0]_i|\} \sqrt{\epsilon_M}$. Note that noiseless forward differences are applied to the true Hessian to estimate the third derivative along each coordinate direction at the initial point. This synthetic {\tt Fixed} strategy is presented for benchmarking purposes; it is not generally viable in practice.

Representative results are shown in Figure \ref{fig:cd l-bfgs}. Similar to the forward-difference case, our algorithm is able to obtain higher accuracy in the solution compared to the {\tt Fixed} strategy, but at higher cost as expected. As demonstrated by Figure \ref{fig:cd l-bfgs}, there is a stronger effect of noise when the objective value is about the same magnitude as the noise level, leading to oscillations in the function. The complete set of results for all problems and noise levels is presented in the Appendix \ref{app:experiments}. 

\begin{figure}
    \centering
    \includegraphics[width=0.24\textwidth]{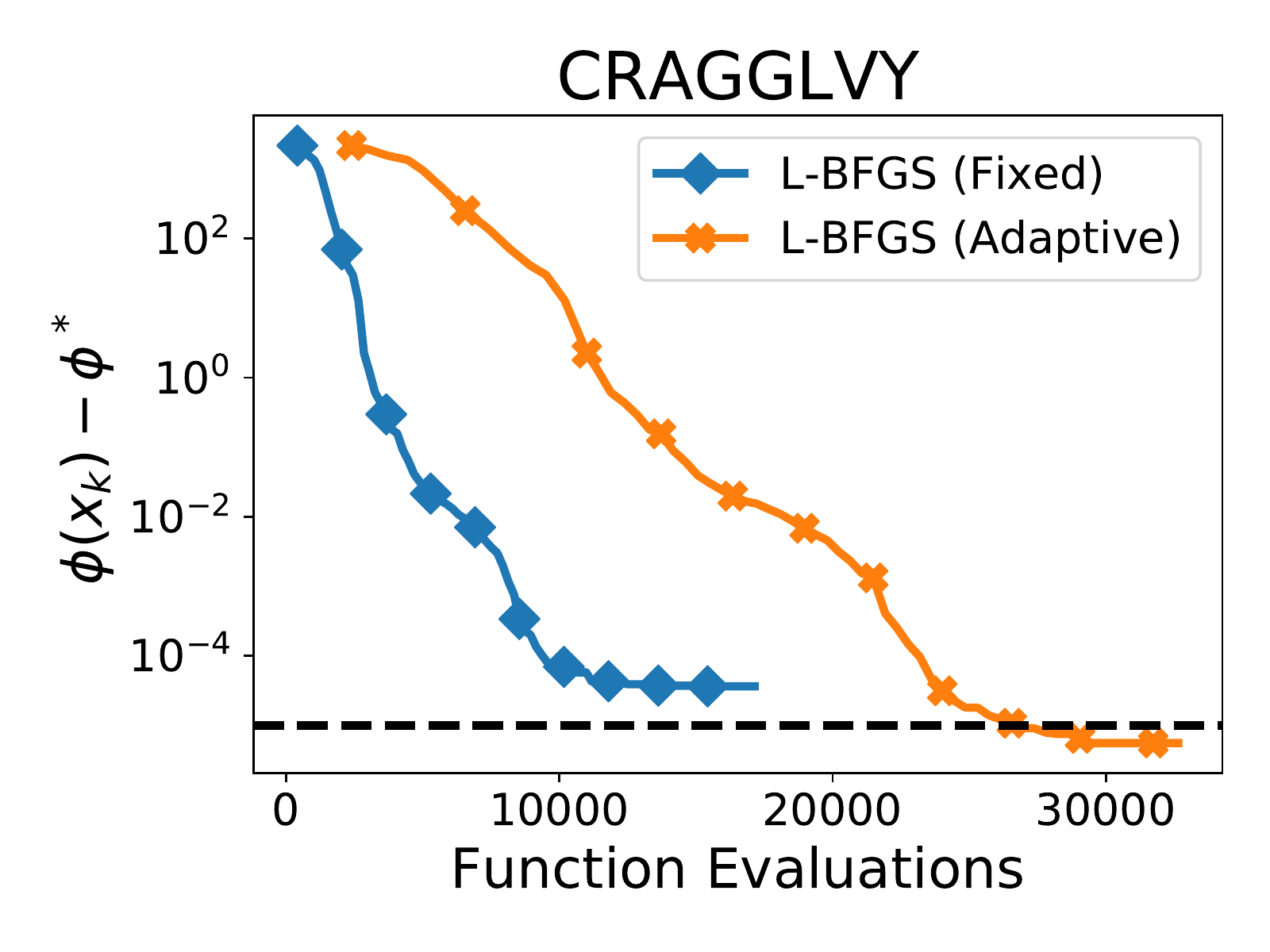}
    \includegraphics[width=0.24\textwidth]{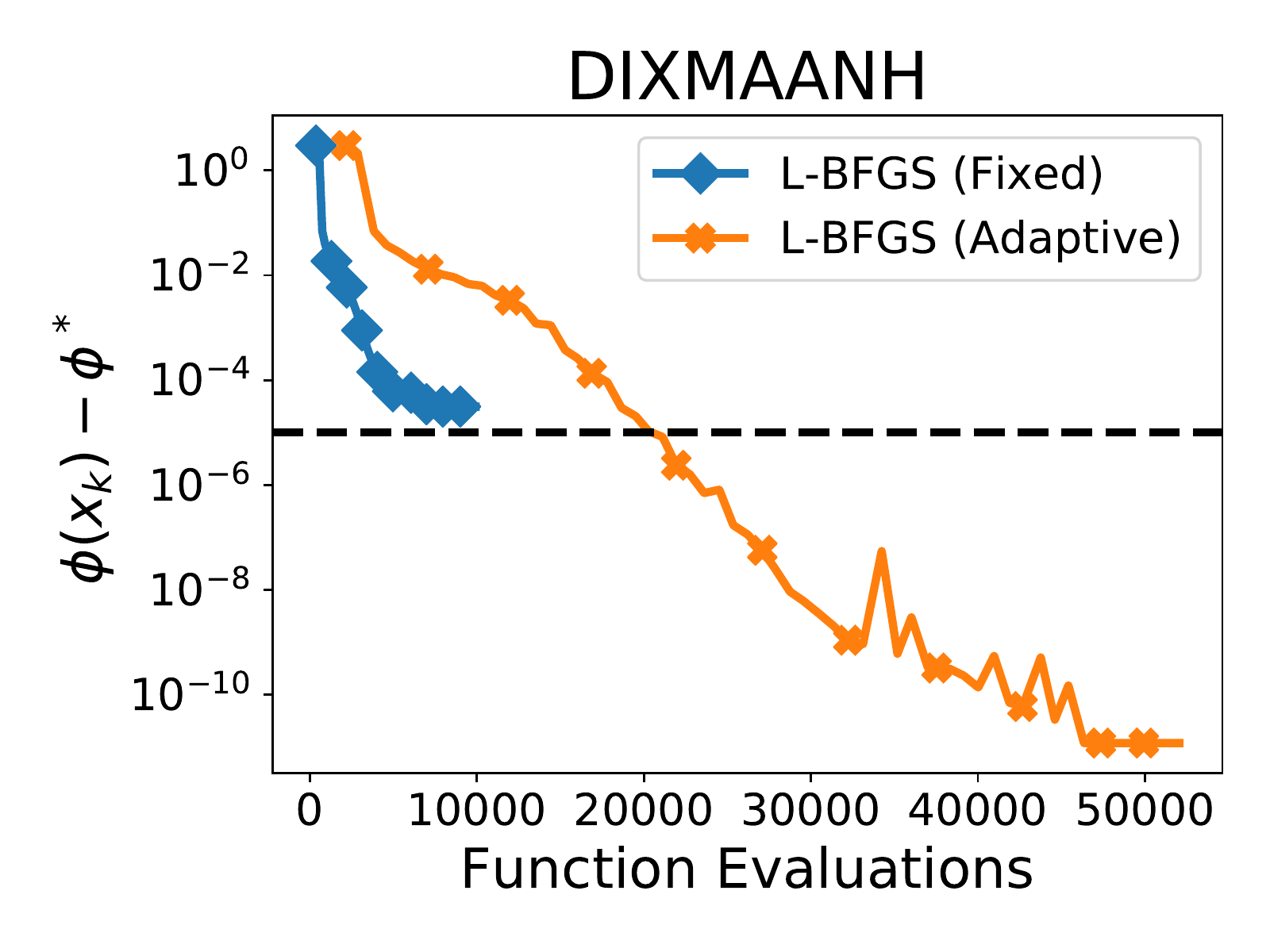}
    \includegraphics[width=0.24\textwidth]{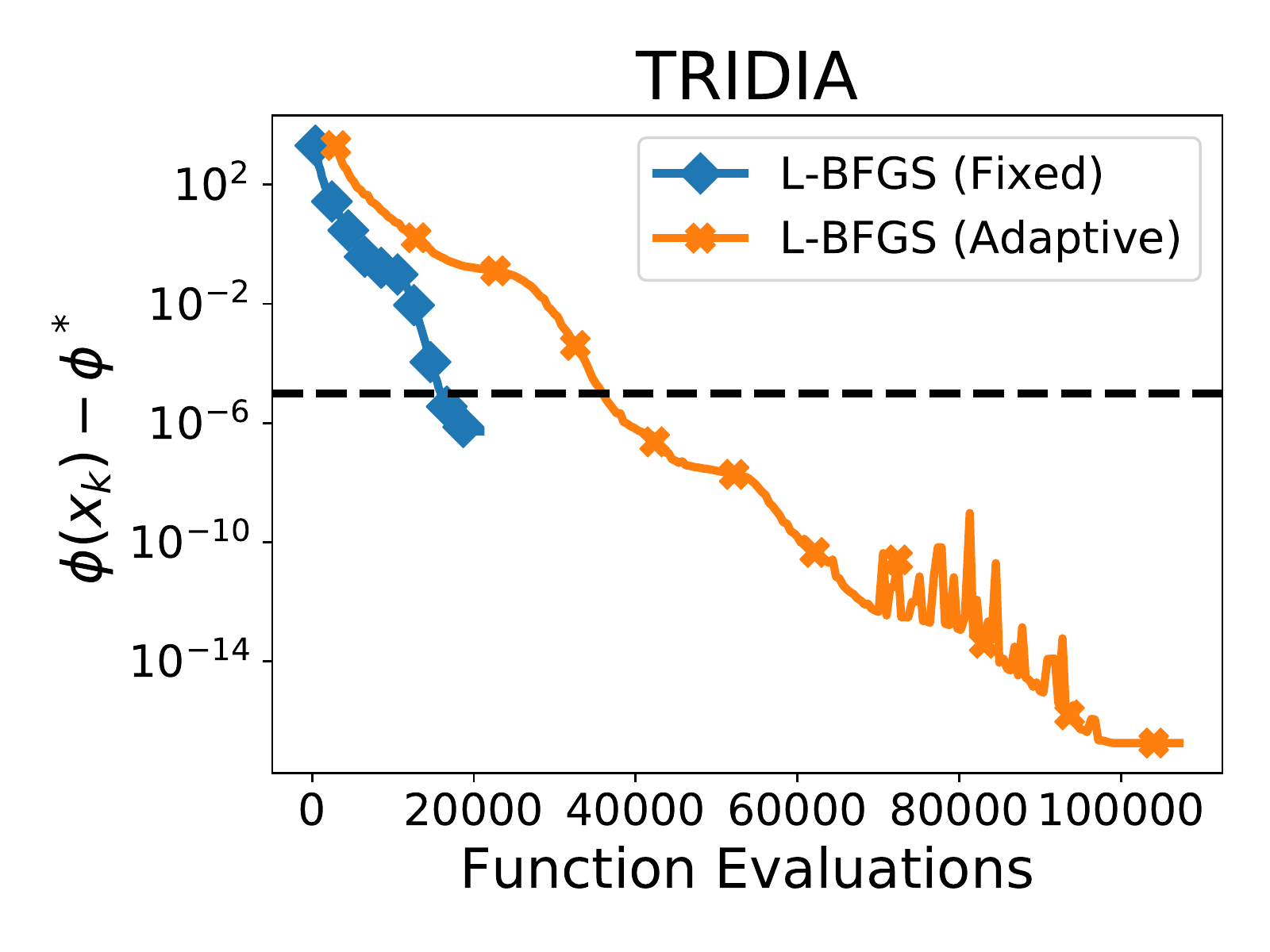}
    \includegraphics[width=0.24\textwidth]{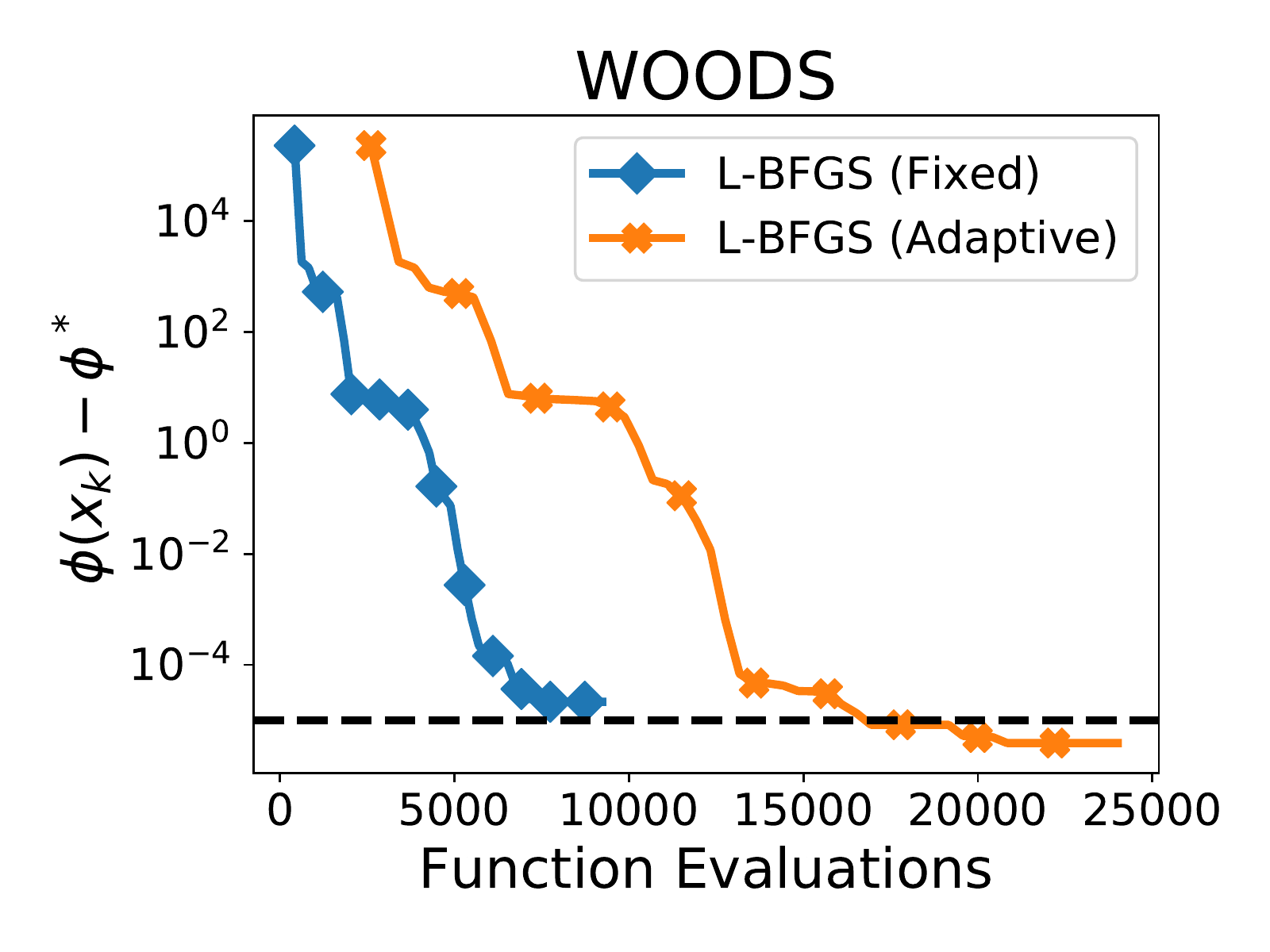}
    \caption{Comparison of central-difference L-BFGS methods with difference intervals determined using a fixed interval and our adaptive algorithm. Comparisons are made on representative problems with noise level $\epsilon_f = 10^{-5}$. We plot the true function value against the function evaluations. The dashed black line shows the noise level $\epsilon_f$ of the function.}
    \label{fig:cd l-bfgs}
\end{figure}

As demonstrated in our experiments, reusing previous difference intervals from prior iterations allows us to reduce the cost of the estimation procedure. Additional savings can be achieved by re-estimating the difference interval periodically;  for simple problems only a few times during the course of the optimization will suffice. 

\section{Final Remarks} \label{sec:final remark}

We have developed a principled and robust procedure for determining the difference interval for estimating gradients in optimization methods, assuming that the noise level is known. Our procedure applies to any finite-difference scheme, including central- and higher-order difference schemes. It performs a bisection search on a ratio that balances the truncation and measurement errors and attains a near-optimal difference interval. While Mor\'e and Wild \cite{more2012estimating} improves upon Gill et al.'s approach \cite{GillMurrSaunWrig83}, and is cheaper than our method, it is not as accurate and robust. These two qualities are essential to make finite-difference approximations reliable enough to be used in established nonlinear optimization algorithms for solving noisy problems.  

In our presentation, we assumed bounded noise, but our methods may be applied to unbounded stochastic noise with finite variance. By applying Chebyshev's inequality, our results on finite difference derivative estimation hold with high probability.

\section*{Acknowledgments}
We are grateful to Oliver Zhuoran Liu and Shigeng Sun for their feedback on this work.

\clearpage

\bibliographystyle{siamplain}
\bibliography{references}

\newpage
\appendix

\section{Finite-Difference Formula Derivation and Tables}
\label{app:fd tables}

We summarize the different standard finite-difference schemes with equidistant points, their theoretical error, optimal steplength, and optimal error in terms of the noise level $\epsilon_f$ and local bound on the $q$-th derivative $L_q$ for a smooth univariate function $\phi : \R \rightarrow \R$ in Tables \ref{tab:det-fd} and \ref{tab:det-cd}. For completeness, we provide a complete derivation of the errors for a generic finite-difference approximation to the $d$-th order derivative below. 

We will use $f : \R \rightarrow \R$ to denote the noisy function evaluations $f(t) = \phi(t) + \epsilon(t)$. We will consider two settings for $\epsilon(t)$: (1) we will assume that $\epsilon(t)$ is bounded, i.e., $|\epsilon(t)| \leq \epsilon_f$ for all $t$; (2) we will assume that $\epsilon(t)$ is a random variable with $\mathbb{E}[\epsilon(t)] = 0$ and $\mathbb{E}[\epsilon(t)^2] = \sigma_f^2$ for all $t$. The tables vary the number of evaluated points $m$ and is dependent on the local Lipschitz constant $L_q \geq 0$ which bounds the $q$-th derivative
$$|\phi^{(q)}(t + h_0 s)| \leq L_q$$
for all $s \in [s_1, s_m]$, where $q$ is the order of the remainder term in the Taylor expansion.

In the most general case, given distinct shifts $\{s_j\}_{j = 1}^m$ and points $\{t_1, ..., t_m\} = \{t + h s_1, t + h s_2, ..., t + h s_m\}$, one can derive a generic finite-difference method to approximate the $d$-th derivative of the form:
\begin{equation*}
  \phi^{(d)}(t) \approx \frac{\sum_{j = 1}^m w_j f(t + s_j h)}{h^d} = f^{(d)}(t; h).
\end{equation*}
We will assume without loss of generality that $s_1 < s_2 < ... < s_m$. First, note that $f^{(d)}$ can be decomposed into a noiseless finite-difference formula and its corresponding error:
\begin{equation*}
    f^{(d)}(t; h) = \frac{\sum_{j = 1}^m w_j \phi(t + s_j h)}{h^d} + \frac{\sum_{j = 1}^m w_j \epsilon(t + s_j h)}{h^d}.
\end{equation*}
Considering the noiseless finite-difference term, since the function is smooth, one can write the Lagrange remainder form of the Taylor series expansions for each function evaluation without noise as:
\begin{equation*}
    \phi(t + h s_j) = \sum_{l = 0}^{q - 1} \frac{1}{l!} \phi^{(l)}(t) s_j^l + \frac{1}{q!} \phi^{(q)}(\xi_j) s_j^q
\end{equation*}
for $\xi_j \in [t, t + h s_j]$ for $j = 1, ..., m$. Therefore, if the weights $w$ satisfy
\begin{equation*}
    \frac{1}{d!} \sum_{j = 1}^m w_j s_j^l = 
    \begin{cases}
    0 & \mbox{ for } l \neq d, ~ l = 0, 1, ..., q - 1 \\
    1 & \mbox{ for } l = d
    \end{cases}
\end{equation*}
then
\begin{equation*}
    \frac{\sum_{j = 1}^m w_j \phi(t + s_j h)}{h^d} = \phi^{(d)}(t) + \frac{h^{q - d}}{q!} \sum_{j = 1}^m w_j \phi^{(q)}(\xi_j) s_j^q.
\end{equation*}
This can be written compactly by the linear system of equations:
\begin{equation*}
    V(s)^T w = d! \cdot e_{p - d}
\end{equation*}
where $V(s) \in \R^{m \times q}$ is the Vandermonde matrix defined as
\begin{equation*}
    V(s) =
    \begin{bmatrix}
        s_1^{q - 1} & s_1^{q - 2} & \hdots & s_1^0 \\
        s_2^{q - 1} & s_2^{q - 2} & \hdots & s_2^0 \\
        \vdots & \vdots & \ddots & \vdots \\
        s_m^{q - 1} & s_m^{q - 2} & \hdots & s_m^0
    \end{bmatrix}
\end{equation*}
and $e_{p - d} \in \R^p$ is the $(p - d)$-th coordinate vector.

To derive a reasonable bound on the total error, suppose we are given $h_0 > 0$ and a bound on $\phi^{(q)}$
\begin{equation*}
    |\phi^{(q)}(t + s h_0)| \leq L_{q}
\end{equation*}
for all $s \in [s_1, s_m]$. If we assume that the error is bounded, i.e., $|\epsilon(t)| \leq \epsilon_f$, then one can then bound the error in the approximation by:
\begin{equation*}
    |f^{(d)}(t; h) - \phi^{(d)}(t)| \leq \frac{L_{q} h^{q - d}}{q!} \sum_{j = 1}^m |w_j s_j^{q}| + \frac{\|w\|_1 \epsilon_f}{h^d} = \epsilon_g(h)
\end{equation*}
for all $0 < h \leq h_0$.
If we assume instead that $\var(\epsilon(t)) = \sigma_f^2$, then we can similarly show
\begin{equation*}
    \mathbb{E}[(f^{(d)}(t; h) - \phi^{(d)}(t))^2] \leq \frac{L_{q}^2 h^{2(q - d)}}{(q!)^2} \sum_{j = 1}^m w_j^2 s_j^{2q} + \frac{\|w\|_2^2 \sigma_f^2}{h^{2d}} = \sigma_g^2(h)
\end{equation*}
for all $0 < h \leq h_0$.

The above Taylor series analysis is pessimistic in that it requires multiple $\xi_j$ points, and therefore yields a loose bound when applying the triangle inequality. Instead, one can consider the derivation of finite-difference schemes for approximating the \textit{first} derivative at an interpolation point using Lagrange polynomials, which yields a tighter bound on the error.

As above, suppose we are given distinct points $\{t_1, ..., t_m\} = \{t + h s_1, ..., t + h s_m\}$ and we are interested in approximating $\phi^{(1)}(t)$.
Recall that the Lagrange basis polynomials are defined as:
\begin{equation*}
    \psi_{p, j}(\tilde{t}) = \frac{\prod_{k \neq j} (\tilde{t} - t_k)}{\prod_{k \neq j} (t_j - t_k)} = \frac{\omega_m(\tilde{t})}{\omega_m^{(1)}(t_j) (\tilde{t} - t_j)}, ~~~~~ \omega_m(\tilde{t}) = \prod_{j = 1}^m (\tilde{t} - t_j).
\end{equation*}
Then the Lagrange interpolation is defined as:
\begin{equation*}
    \ell(\tilde{t}) = \sum_{j = 1}^m \psi_{m, j}(\tilde{t}) \phi(t_j).
\end{equation*}
It is well-known that the remainder is
\begin{equation*}
    \phi(\tilde{t}) - \ell(\tilde{t}) = \frac{\omega_m(\tilde{t})}{m!} \phi^{(m)}(\xi)
\end{equation*}
for some $\xi \in [t_1, t_m]$. Note that the finite-difference formula can simply be obtained by differentiating the Lagrange polynomial
\begin{equation*}
    \ell^{(1)}(\tilde{t}) = \sum_{j = 1}^m \psi_{m, j}^{(1)}(\tilde{t}) \phi(t_j).
\end{equation*}
Therefore, the finite-difference coefficients are obtained by evaluating $\psi_{m, j}^{(1)}(\tilde{t})$. The error is also obtained by noting
\begin{equation*}
    \phi^{(1)}(\tilde{t}) = \ell^{(1)}(\tilde{t}) + \frac{\omega_m^{(1)}(\tilde{t})}{m!} \phi^{(m)}(\xi) + \frac{\omega_m(\tilde{t})}{m!} \phi^{(m)}(\xi) \frac{d \xi}{d x}.
\end{equation*}
Since
\begin{equation*}
    \omega_m^{(1)}(\tilde{t}) = \sum_{j = 1}^m \prod_{k \neq j} (\tilde{t} - t_k),
\end{equation*}
plugging in $\tilde{t} = t_i$ for any $i = 1, ..., m$, we get the following equality
\begin{equation*}
    \phi^{(1)}(t_i) = \ell^{(1)}(t_i) + \frac{\omega_m^{(1)}(t_i)}{m!} \phi^{(m)}(\xi) = \ell^{(1)}(t_i) + \prod_{j \neq i} (t_i - t_j) \frac{\phi^{(m)}(\xi)}{m!}.
\end{equation*}

Given $h_0 > 0$ and a bound on $\phi^{(m)}$
\begin{equation*}
    |\phi^{(m)}(t + h_0 s)| \leq L_{m}
\end{equation*}
for all $s \in [s_1, s_m]$ and assuming $t = t_i$ is one of the interpolation points, we obtain the bound
\begin{equation*}
    |\phi^{(1)}(t) - \ell^{(1)}(t)| \leq \frac{L_{m} h^{m - 1}}{m!} \left|\prod_{j \neq i} s_j \right|
\end{equation*}
and if we incorporate the error in the function evaluations, we obtain a error and variance bounds of
\begin{align*}
     |f^{(1)}(t; h) - \phi^{(1)}(t)| & \leq \frac{L_{m} h^{m - 1}}{m!} \left|\prod_{j \neq i} s_j \right| + \frac{\|w\|_1 \epsilon_f}{h} = \epsilon_g(h) \\
     \mathbb{E}[(f^{(1)}(t; h) - \phi^{(1)}(t))^2] & \leq \frac{L_m^2 h^{2(m - 1)}}{(m!)^2} \prod_{j \neq i} s_j^2 + \frac{\|w\|_2^2 \sigma_f^2}{h^2} = \sigma_g^2(h)
\end{align*}
for all $0 < h \leq h_0$.

\begin{landscape}
\begin{table}
{\footnotesize
\caption{Table containing the finite-difference formula, deterministic error bound $|f^{(1)}(t; h) - \phi^{(1)}(t)| \leq \epsilon_g(h)$ for generic $h$, optimal $h^*$, and optimal error $\epsilon_g(h^*)$ for forward-difference schemes with number of points $m \in \{2, 3, 4, 5\}$.} \label{tab:det-fd}
\begin{center}
\begin{tabular}{ccccc}
\hline
{$m$} & {$f^{(1)}(t; h)$} & {$\epsilon_g(h)$} & {$h^*$} & {$\epsilon_g(h^*)$} \\
\hline
{$2$} & {$\frac{f(t + h) - f(t)}{h}$} & {$\frac{L_2 h}{2} + \frac{2 \epsilon_f}{h}$} & {$2 \sqrt{\frac{\epsilon_f}{L_2}}$} & {$2 \sqrt{L_2 \epsilon_f}$} \\
{$3$} & {$\frac{-3 f(t) + 4 f(t + h) - f(t + 2 h)}{2 h}$} & {$\frac{L_3 h^2}{3} + \frac{4 \epsilon_f}{h}$} & {$\sqrt[3]{\frac{6 \epsilon_f}{L_3}}$} & {$6^{2 / 3} L_3^{1/3} \epsilon_f^{2/3}$} \\
{$4$} & {$\frac{-11 f(t) + 18 f(t + h) - 9 f(t + 2h) + 2f(t + 3h)}{6h}$} & {$\frac{L_4 h^3}{4} + \frac{20 \epsilon_f}{3 h}$} & {$\sqrt[4]{\frac{80 \epsilon_f}{9 L_4}}$} & {$\frac{8 \cdot 5^{3/4}}{3 \sqrt{3}} L_4^{1/4} \epsilon_f^{3/4}$} \\
{$5$} & {$\frac{-25 f(t) + 48 f(t + h) - 36 f(t + 2h) + 16 f(t + 3h) - 3 f(t + 4h)}{12 h}$} & {$\frac{L_5 h^4}{5} + \frac{32 \epsilon_f}{3 h}$} & {$\sqrt[5]{\frac{40 \epsilon_f}{3 L_5}}$} & {$4 \left(\frac{5}{3}\right)^{4 / 5} 2^{2/5} L_5^{1/5} \epsilon_f^{4/5}$} \\
\hline
\end{tabular}
\end{center}
}
\end{table}

\begin{table}
{\footnotesize 
\caption{Table containing the finite-difference formula, deterministic error bound $|f^{(1)}(t; h) - \phi^{(1)}(t)| \leq \epsilon_g(h)$ for generic $h$, optimal $h^*$, and optimal error $\epsilon_g(h^*)$ for central-difference schemes with number of points $m \in \{2, 4, 6\}$.} \label{tab:det-cd}
\begin{center}
\begin{tabular}{ccccc}
\hline
{$m$} & {$f^{(1)}(t; h)$} & {$\epsilon_g(h)$} & {$h^*$} & {$\epsilon_g(h^*)$} \\
\hline
{$2$} & {$\frac{f(t + h) - f(t - h)}{2h}$} & {$\frac{L_3 h^2}{6} + \frac{\epsilon_f}{h}$} & {$\sqrt[3]{\frac{3 \epsilon_f}{L_3}}$} & {$\frac{3^{2/3}}{2} L_3^{1/3} \epsilon_f^{2/3}$} \\
{$4$} & {$\frac{f(t - 2h) - 8 f(t - h) + 8f(t + h) - f(t + 2h)}{12 h}$} & {$\frac{L_5 h^4}{30} + \frac{3 \epsilon_f}{2 h}$} & {$\sqrt[5]{\frac{45 \epsilon_f}{4 L_5}}$} & {$\frac{1}{4} \left(\frac{3}{2}\right)^{4/5} 5^{4/5} L_5^{1/5} \epsilon_f^{4/5}$} \\
{$6$} & {$\frac{-f(t - 3h) + 9f(t - 2h) - 45f(t - h) + 45 f(t + h) - 9f(t + 2h) + f(t + 3h)}{60h}$} & {$\frac{L_7 h^6}{140} + \frac{11 \epsilon_f}{6 h}$} & {$\sqrt[7]{\frac{385 \epsilon_f}{9 L_7}}$} & {$\frac{77^{6 / 7}}{12 \cdot 3^{5/7} \cdot \sqrt[7]{5}} L_7^{1/7} \epsilon_f^{6/7}$} \\
\hline
\end{tabular}
\end{center}
}
\end{table}

\end{landscape}


\begin{landscape}
\begin{table}
{\scriptsize
\caption{Table containing the finite-difference formula, MSE error bound $\mathbb{E}[(f^{(1)}(t; h) - \phi^{(1)}(t))^2] \leq \sigma^2_g(h)$ for generic $h$, optimal $h^*$, and optimal error $\sigma_g(h^*)$ for forward-difference schemes with number of points $m \in \{2, 3, 4, 5\}$.} \label{tab:sto-fd}
\begin{center}
\begin{tabular}{ccccc}
\hline
{$p$} & {$f^{(1)}(t; h)$} & {$\sigma_g^2(h)$} & {$h^*$} & {$\sigma_g(h^*)$} \\
\hline
{$1$} & {$\frac{f(t + h) - f(t)}{h}$} & {$ \frac{L_2^2 h^2}{4} + \frac{2 \epsilon_f^2}{h^2}$} & {$8^{1/4} \sqrt{\frac{\epsilon_f}{L_2}}$} & {$2^{1/4} \sqrt{L_2 \epsilon_f}$} \\
{$2$} & {$\frac{-3 f(t) + 4 f(t + h) - f(t + 2 h)}{2 h}$} & {$\frac{L_3^2 h^4}{9} + \frac{13 \epsilon_f^2}{2h^2}$} & {$(\frac{3}{2})^{1/3}13^{1/6}\sqrt[3]{\frac{\epsilon_f}{L_3}}$} & {$ \frac{\sqrt[6]{3} \sqrt[3]{13}}{2^{2/3}}  L_3^{1/3} \epsilon_f^{2/3}$} \\
{$3$} & {$\frac{-11 f(t) + 18 f(t + h) - 9 f(t + 2h) + 2f(t + 3h)}{6h}$} & {$\frac{L_4^2 h^6}{16} + \frac{265 \epsilon_f^2}{18 h^2}$} & {$(\frac{2}{3})^{3/8}265^{1/8}\sqrt[4]{\frac{\epsilon_f}{L_4}}$} & {$ \frac{1}{3} \sqrt[8]{\frac{2}{3}} 265^{3/8} L_4^{1/4} \epsilon_f^{3/4}$} \\
{$4$} & {$\frac{-25 f(t) + 48 f(t + h) - 36 f(t + 2h) + 16 f(t + 3h) - 3 f(t + 4h)}{12 h}$} & {$\frac{L_5^2 h^8}{25} + \frac{2245 \epsilon_f^2}{72 h^2}$} & {$\frac{5^{3/10}449^{1/10}}{\sqrt{2}\sqrt[5]{3}} \sqrt[5]{\frac{ \epsilon_f}{L_5}}$} & {$ \frac{5^{7/10}449^{2/5}}{4\cdot 3^{4/5}} L_5^{1/5} \epsilon_f^{4/5}$} \\
\hline
\end{tabular}
\end{center}
}
\end{table}

\begin{table}
{\scriptsize
\caption{Table containing the finite-difference formula, MSE error bound $\mathbb{E}[(f^{(1)}(t; h) - \phi^{(1)}(t))^2] \leq \sigma^2_g(h)$ for generic $h$, optimal $h^*$, and optimal error $\sigma_g(h^*)$ for central-difference schemes with number of points $m \in \{2, 4, 6\}$.} \label{tab:sto-cd}
\begin{center}
\begin{tabular}{ccccc}
\hline
{$p$} & {$f^{(1)}(t; h)$} & {$\sigma_g^2(h)$} & {$h^*$} & {$\sigma_g(h^*)$} \\
\hline
{$2$} & {$\frac{f(t + h) - f(t - h)}{2h}$} & {$ \frac{L_3^2 h^4}{36} + \frac{\epsilon_f^2}{2h^2}$} & {$\sqrt[3]{3} \sqrt[3]{\frac{\epsilon_f}{L_3}}$} & {$\frac{\sqrt[6]{3}}{2} L_3^{1/3} \epsilon_f^{2/3}$} \\
{$4$} & {$\frac{f(t - 2h) - 8 f(t - h) + 8f(t + h) - f(t + 2h)}{12 h}$} & {$\frac{L_5^2 h^8}{900} + \frac{65 \epsilon_f^2}{72 h^2}$} & {$\left(\frac{5}{2}\right)^{3/10} 13^{1/10} \sqrt[5]{\frac{\epsilon_f}{L_5}}$} & {$ \frac{5^{7 / 10} \cdot 13^{2 / 5}}{12 \cdot \sqrt[5]{2}} L_5^{1/5} \epsilon_f^{4/5}$} \\
{$6$} & {$\frac{-f(t - 3h) + 9f(t - 2h) - 45f(t - h) + 45 f(t + h) - 9f(t + 2h) + f(t + 3h)}{60h}$} & {$\frac{L_7^2 h^{12}}{140^2} + \frac{2107 \epsilon_f^2}{1800 h^2}$} & {$\frac{7^{2/7}43^{1/14}}{3^{3/14}}\sqrt[7]{\frac{\epsilon_f}{L_7}}$} & {$\frac{7^{17/14}43^{3/7}}{60 \cdot 3^{2/7}} L_7^{1/7} \epsilon_f^{6/7}$} \\
\hline
\end{tabular}
\end{center}
}
\end{table}
\end{landscape}

\section{Complete Experimental Results}
\label{app:experiments}

Here, we present the complete experimental results from Section \ref{sec:experiments}.

\subsection{Robustness to Different Noise Levels}

We test our procedure on a simple function $\phi(t) = \cos(t)$ for different noise levels using different schemes listed in Table \ref{tab:schemes_tested}. These are shown in Figure \ref{fig:cos_func_noise_level_app}. Detailed numerical results, including the number of iterations and relative error, are listed in Table \ref{tab:noise_level_exp}. 

Observe that our method is able to consistently achieve low relative error using a similar number of function evaluations across all tested noise levels. This is a desirable property, as it demonstrates that our initial choice of the interval $h$ and our method is consistent over different noise levels.

\begin{figure}
    \centering
    \includegraphics[width=0.49\linewidth]{figures/adafd/noise_level/FD.pdf}
	\includegraphics[width=0.49\linewidth]{figures/adafd/noise_level/CD.pdf}
    \\
    \includegraphics[width=0.49\linewidth]{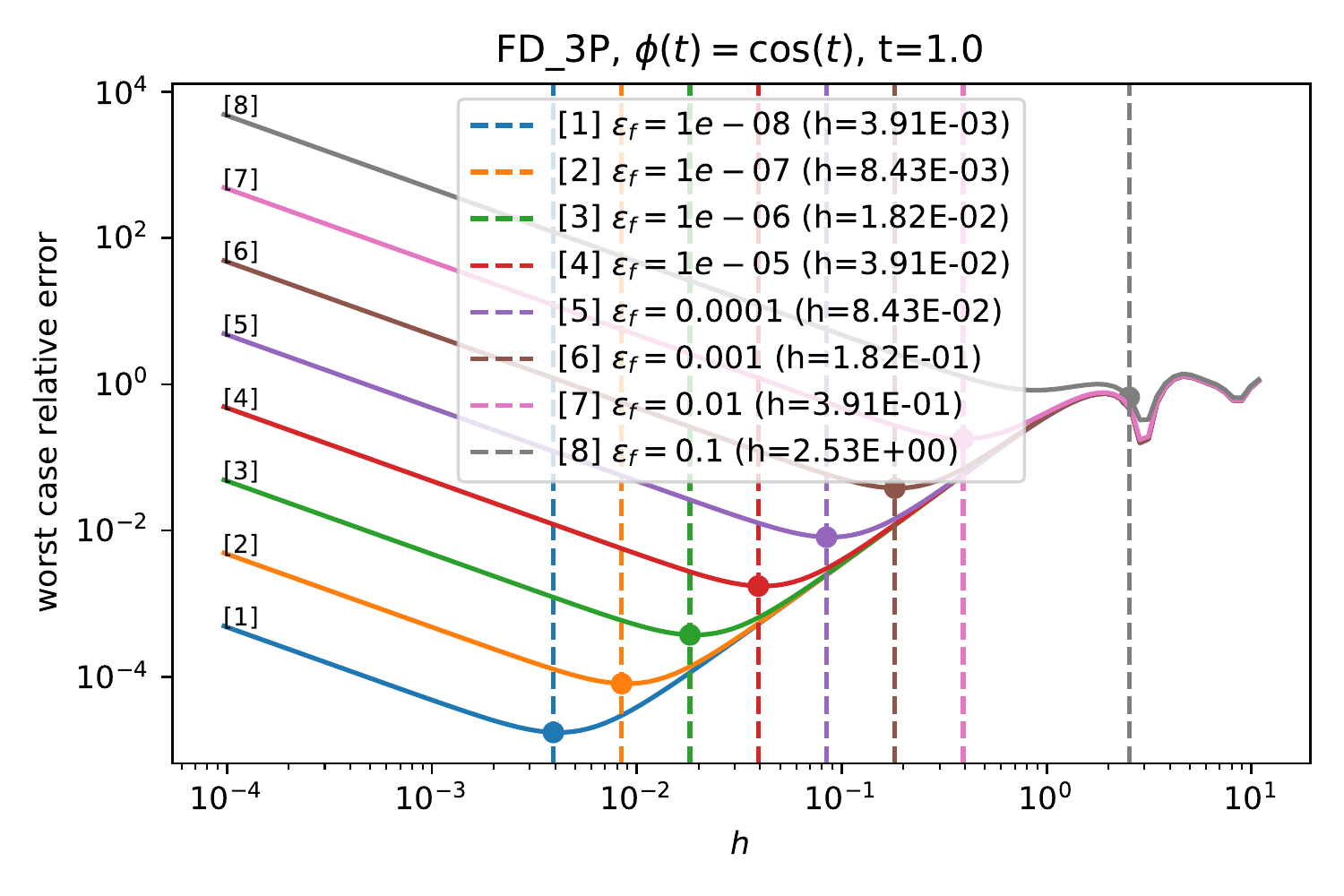}
	\includegraphics[width=0.49\linewidth]{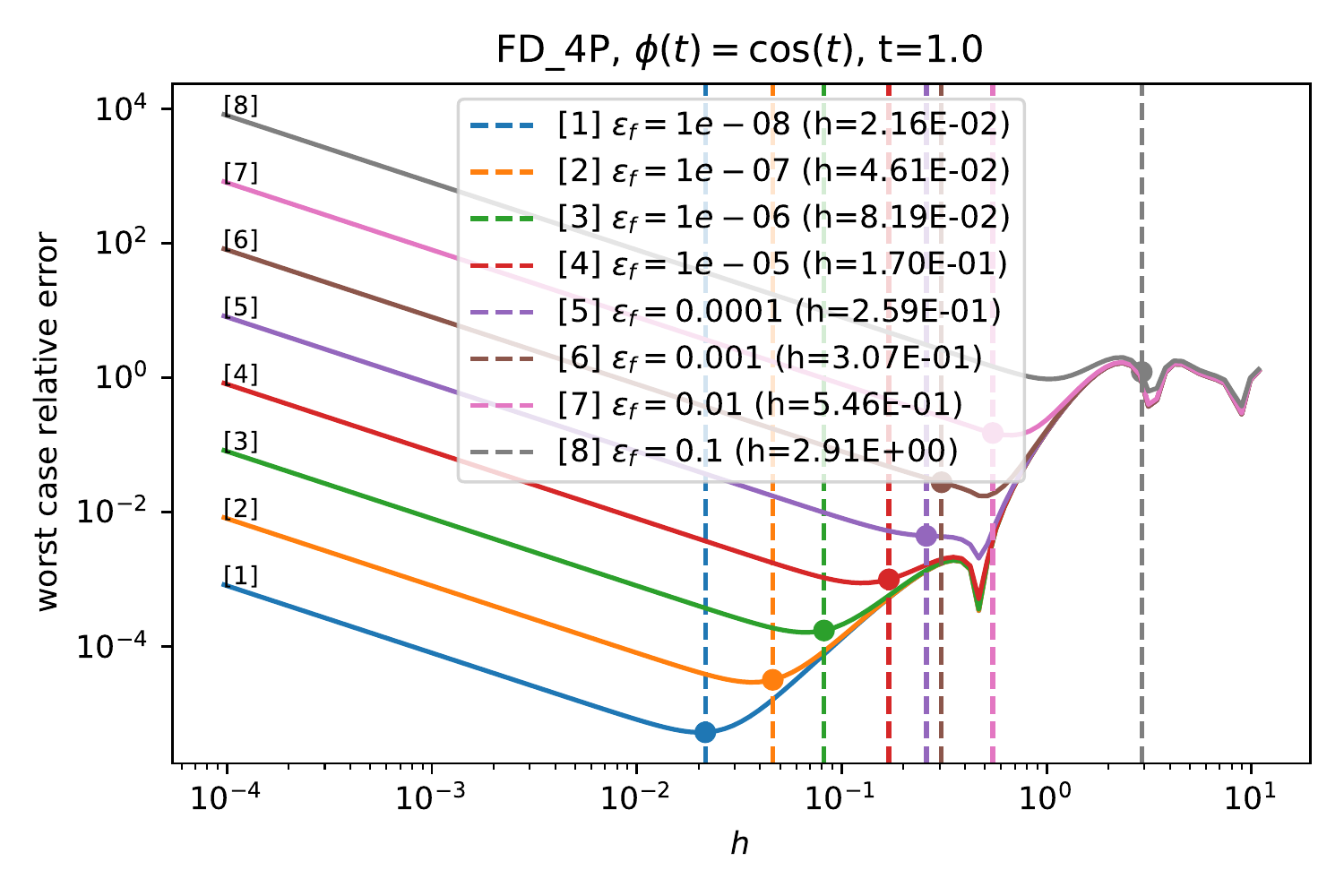}
	\\
	\includegraphics[width=0.49\linewidth]{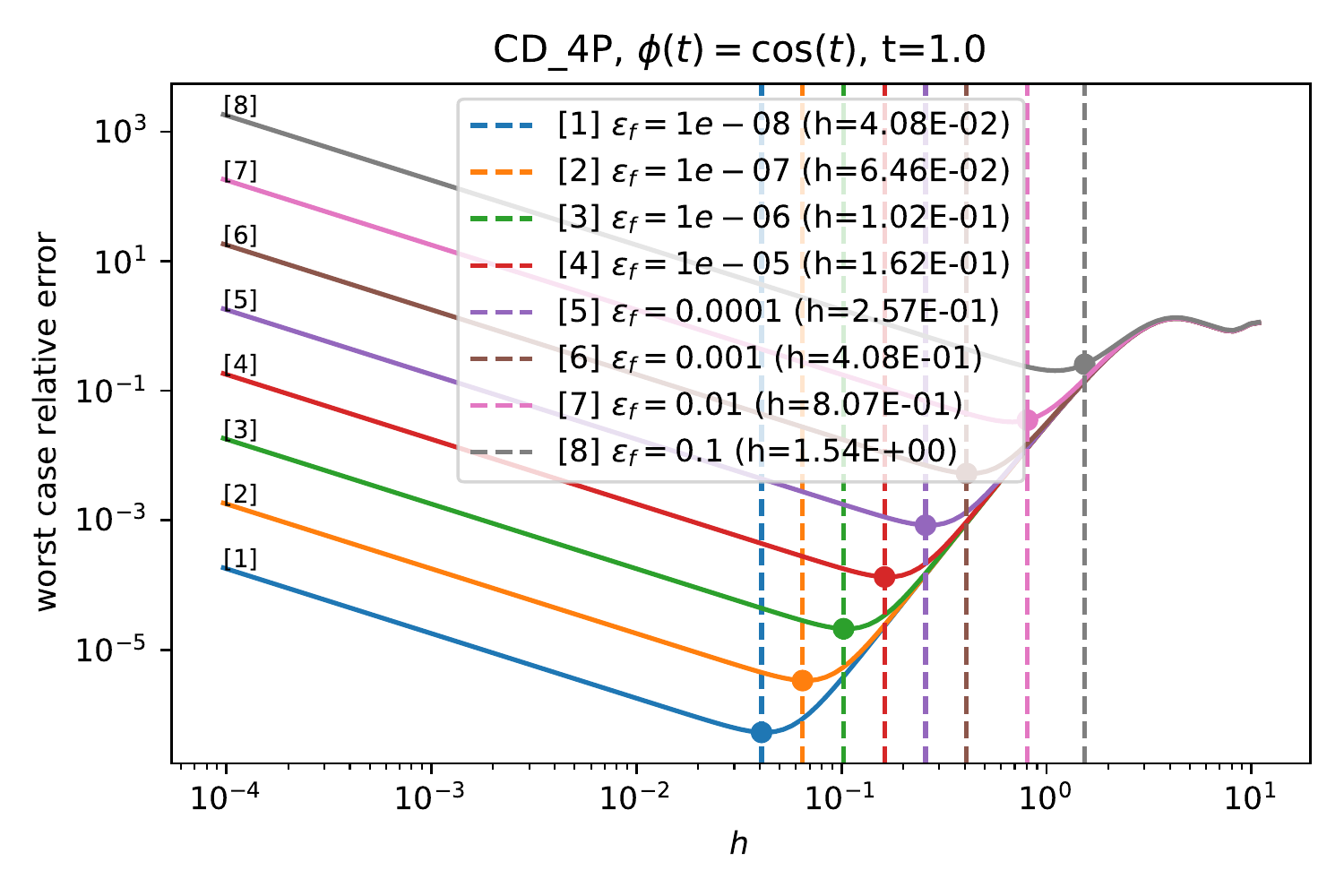}
	\includegraphics[width=0.49\linewidth]{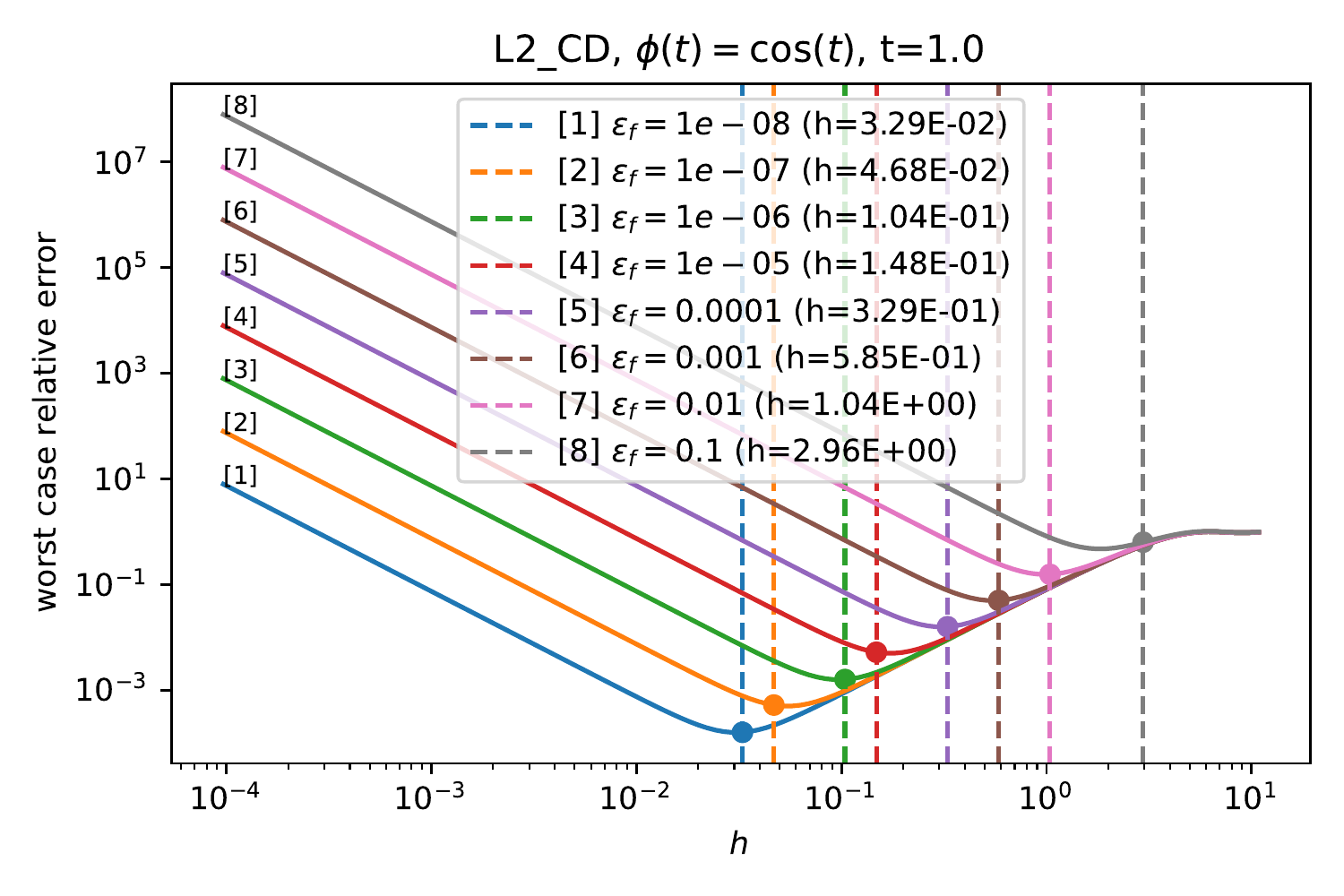}
	\\
    \caption{Worst case relative error $\delta_S(h; \phi, t, \epsilon_f)$ against $h$ on function $\phi(t) = \cos(t)$ with different noise levels; the vertical dashed line represents the $h_\dagger$ output by Algorithm \ref{algo:gen_est_proc}.
    }
    \label{fig:cos_func_noise_level_app}
\end{figure}

\begin{table}[tbhp]
{\scriptsize
\caption{Detailed results for $\phi(t) = \cos(t)$ with different noise levels; $r$ represents the final testing ratio; $h^*$ is the $h$ that minimizes $\delta_S(h; \phi, t, \epsilon_f)$ reported by \texttt{minimize\_scalar} function in \texttt{scipy.optimize} and could be unreliable.}\label{tab:noise_level_exp}
\begin{center}
\begin{tabular}{lrrrrrrr}
	\hline
		scheme &    $h_\dagger$ & $h^*$ &    $r$ &  \texttt{\#iters} &  \texttt{\#Evals} & relative error &    $\epsilon_f$ \\
		\hline
    \texttt{FD} & 2.00e-04 &  2.72e-04 &  2.08 &       1 &         3 &  1.86e-05 & 1.00e-08 \\
    \texttt{FD} & 6.32e-04 &  8.59e-04 &  2.00 &       1 &         3 &  4.21e-05 & 1.00e-07 \\
    \texttt{FD} & 3.50e-03 &  2.73e-03 &  4.79 &       4 &         8 &  1.17e-03 & 1.00e-06 \\
    \texttt{FD} & 6.32e-03 &  8.64e-03 &  1.77 &       1 &         3 &  1.72e-03 & 1.00e-05 \\
    \texttt{FD} & 2.00e-02 &  2.76e-02 &  2.00 &       1 &         3 &  1.69e-03 & 1.00e-04 \\
    \texttt{FD} & 6.32e-02 &  9.05e-02 &  1.73 &       1 &         3 &  5.07e-04 & 1.00e-03 \\
    \texttt{FD} & 5.00e-01 &  1.73e+00 &  3.89 &       3 &         6 &  9.86e-02 & 1.00e-02 \\
    \texttt{FD} & 6.32e-01 &  8.26e+00 &  1.52 &       1 &         3 &  2.97e-01 & 1.00e-01 \\
    \texttt{CD} & 3.11e-03 &  3.29e-03 &  2.40 &       1 &         4 &  1.89e-06 & 1.00e-08 \\
    \texttt{CD} & 6.69e-03 &  7.09e-03 &  2.66 &       1 &         4 &  5.39e-06 & 1.00e-07 \\
    \texttt{CD} & 1.44e-02 &  1.53e-02 &  2.72 &       1 &         4 &  9.35e-06 & 1.00e-06 \\
    \texttt{CD} & 3.11e-02 &  3.29e-02 &  2.05 &       1 &         4 &  3.34e-04 & 1.00e-05 \\
    \texttt{CD} & 6.69e-02 &  7.09e-02 &  2.18 &       1 &         4 &  1.33e-03 & 1.00e-04 \\
    \texttt{CD} & 1.44e-01 &  1.53e-01 &  2.55 &       1 &         4 &  3.32e-03 & 1.00e-03 \\
    \texttt{CD} & 3.11e-01 &  3.30e-01 &  1.89 &       1 &         4 &  3.84e-02 & 1.00e-02 \\
    \texttt{CD} & 6.69e-01 &  7.74e+00 &  2.01 &       1 &         4 &  5.71e-02 & 1.00e-01 \\
\texttt{FD\_3P} & 3.91e-03 &  4.14e-03 &  2.88 &       1 &         5 &  1.01e-05 & 1.00e-08 \\
\texttt{FD\_3P} & 8.43e-03 &  8.92e-03 &  3.24 &       1 &         5 &  2.01e-05 & 1.00e-07 \\
\texttt{FD\_3P} & 1.82e-02 &  1.92e-02 &  2.76 &       1 &         5 &  2.05e-04 & 1.00e-06 \\
\texttt{FD\_3P} & 3.91e-02 &  4.11e-02 &  3.28 &       1 &         5 &  5.82e-04 & 1.00e-05 \\
\texttt{FD\_3P} & 8.43e-02 &  8.77e-02 &  2.86 &       1 &         5 &  5.65e-03 & 1.00e-04 \\
\texttt{FD\_3P} & 1.82e-01 &  1.86e-01 &  3.10 &       1 &         5 &  2.29e-02 & 1.00e-03 \\
\texttt{FD\_3P} & 3.91e-01 &  2.99e+00 &  2.88 &       1 &         5 &  9.96e-02 & 1.00e-02 \\
\texttt{FD\_3P} & 2.53e+00 &  2.12e+01 &  5.75 &       2 &         7 &  4.17e-01 & 1.00e-01 \\
\texttt{FD\_4P} & 2.16e-02 &  2.04e-02 &  9.36 &       5 &        18 &  2.14e-06 & 1.00e-08 \\
\texttt{FD\_4P} & 4.61e-02 &  3.67e-02 & 16.51 &       4 &        13 &  1.14e-05 & 1.00e-07 \\
\texttt{FD\_4P} & 8.19e-02 &  4.76e-01 & 10.80 &       4 &        13 &  9.43e-05 & 1.00e-06 \\
\texttt{FD\_4P} & 1.70e-01 &  1.25e-01 &  4.40 &       5 &        18 &  6.15e-04 & 1.00e-05 \\
\texttt{FD\_4P} & 2.59e-01 &  3.28e+00 & 14.62 &       4 &        13 &  7.80e-04 & 1.00e-04 \\
\texttt{FD\_4P} & 3.07e-01 &  3.28e+00 &  4.23 &       1 &         6 &  5.46e-03 & 1.00e-03 \\
\texttt{FD\_4P} & 5.46e-01 &  8.78e+00 &  6.44 &       1 &         6 &  3.19e-02 & 1.00e-02 \\
\texttt{FD\_4P} & 2.91e+00 &  8.79e+00 &  4.28 &       2 &         8 &  9.55e-01 & 1.00e-01 \\
\texttt{CD\_4P} & 4.08e-02 &  4.22e-02 &  2.52 &       1 &         6 &  1.16e-07 & 1.00e-08 \\
\texttt{CD\_4P} & 6.46e-02 &  6.69e-02 &  2.04 &       1 &         6 &  8.34e-07 & 1.00e-07 \\
\texttt{CD\_4P} & 1.02e-01 &  1.06e-01 &  1.95 &       1 &         6 &  8.81e-06 & 1.00e-06 \\
\texttt{CD\_4P} & 1.62e-01 &  1.68e-01 &  1.79 &       1 &         6 &  4.29e-05 & 1.00e-05 \\
\texttt{CD\_4P} & 2.57e-01 &  2.67e-01 &  1.71 &       1 &         6 &  4.00e-04 & 1.00e-04 \\
\texttt{CD\_4P} & 4.08e-01 &  4.25e-01 &  1.89 &       1 &         6 &  8.32e-04 & 1.00e-03 \\
\texttt{CD\_4P} & 8.07e-01 &  7.97e+00 &  4.83 &       4 &        20 &  5.87e-03 & 1.00e-02 \\
\texttt{CD\_4P} & 1.54e+00 &  2.06e+01 &  4.25 &       3 &        14 &  2.34e-01 & 1.00e-01 \\
\texttt{L2\_CD} & 3.29e-02 &  3.07e-02 &  3.78 &       4 &        15 &  1.00e-04 & 1.00e-08 \\
\texttt{L2\_CD} & 4.68e-02 &  5.46e-02 &  1.89 &       1 &         5 &  8.55e-05 & 1.00e-07 \\
\texttt{L2\_CD} & 1.04e-01 &  9.71e-02 &  4.22 &       4 &        15 &  7.28e-04 & 1.00e-06 \\
\texttt{L2\_CD} & 1.48e-01 &  1.73e-01 &  1.90 &       1 &         5 &  1.06e-03 & 1.00e-05 \\
\texttt{L2\_CD} & 3.29e-01 &  3.07e-01 &  4.03 &       4 &        15 &  8.28e-03 & 1.00e-04 \\
\texttt{L2\_CD} & 5.85e-01 &  5.49e-01 &  3.81 &       4 &        15 &  2.84e-02 & 1.00e-03 \\
\texttt{L2\_CD} & 1.04e+00 &  9.87e-01 &  3.45 &       4 &        15 &  7.95e-02 & 1.00e-02 \\
\texttt{L2\_CD} & 2.96e+00 &  1.55e+01 &  5.45 &       2 &         7 &  5.40e-01 & 1.00e-01 \\
\hline
	\end{tabular}
\end{center}
}	
\end{table}

\subsection{Affine Invariance}

One advantage of our proposed method is that the testing ratio remains unchanged under affine transformations of the function. It is particularly obvious that our procedure is invariant when adding a constant to the function. Hence, we focus on transformations of the form $\phi(t) \rightarrow a \cdot \phi(b \cdot t)$ for some $a, b \neq 0$. 

To do this, we test Algorithm \ref{algo:gen_est_proc} on the function $\phi(t) = a \cdot \sin(b \cdot t)$ at $t = 0$ for various $a$ and $b$. We fix the noise level to be $\epsilon_f = 10^{-3}$. The results are shown in Figure \ref{fig:sin_func_scale}. Detailed results can be found in Table \ref{tab:scale_inv_exp} and \ref{tab:scale_inv_exp_2}. As seen in Figure \ref{fig:sin_func_scale}, our method is affine-invariant and can output consistently correct results for different $a$ and $b$.

\begin{figure}
    \centering
    \includegraphics[width=0.49\linewidth]{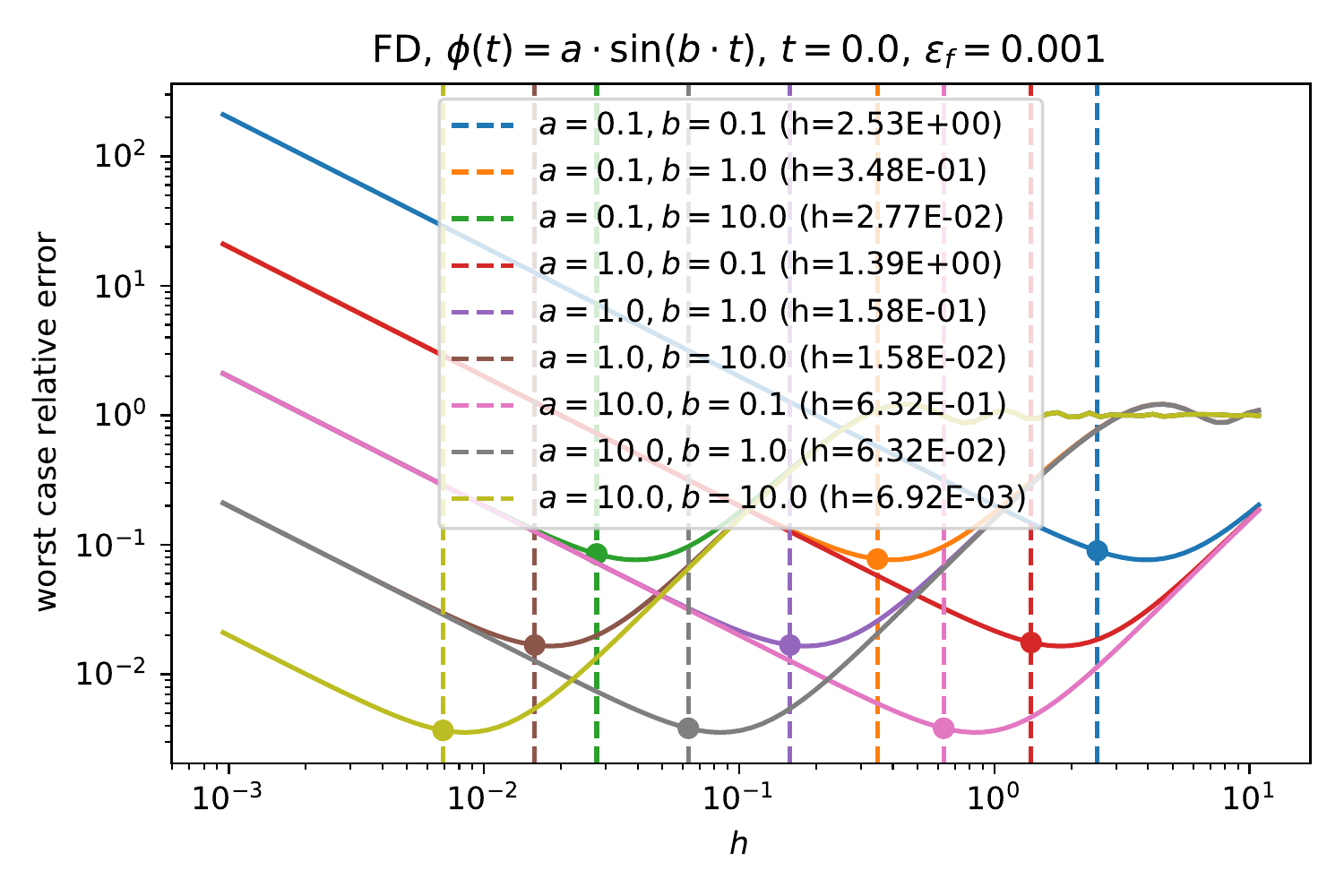}
	\includegraphics[width=0.49\linewidth]{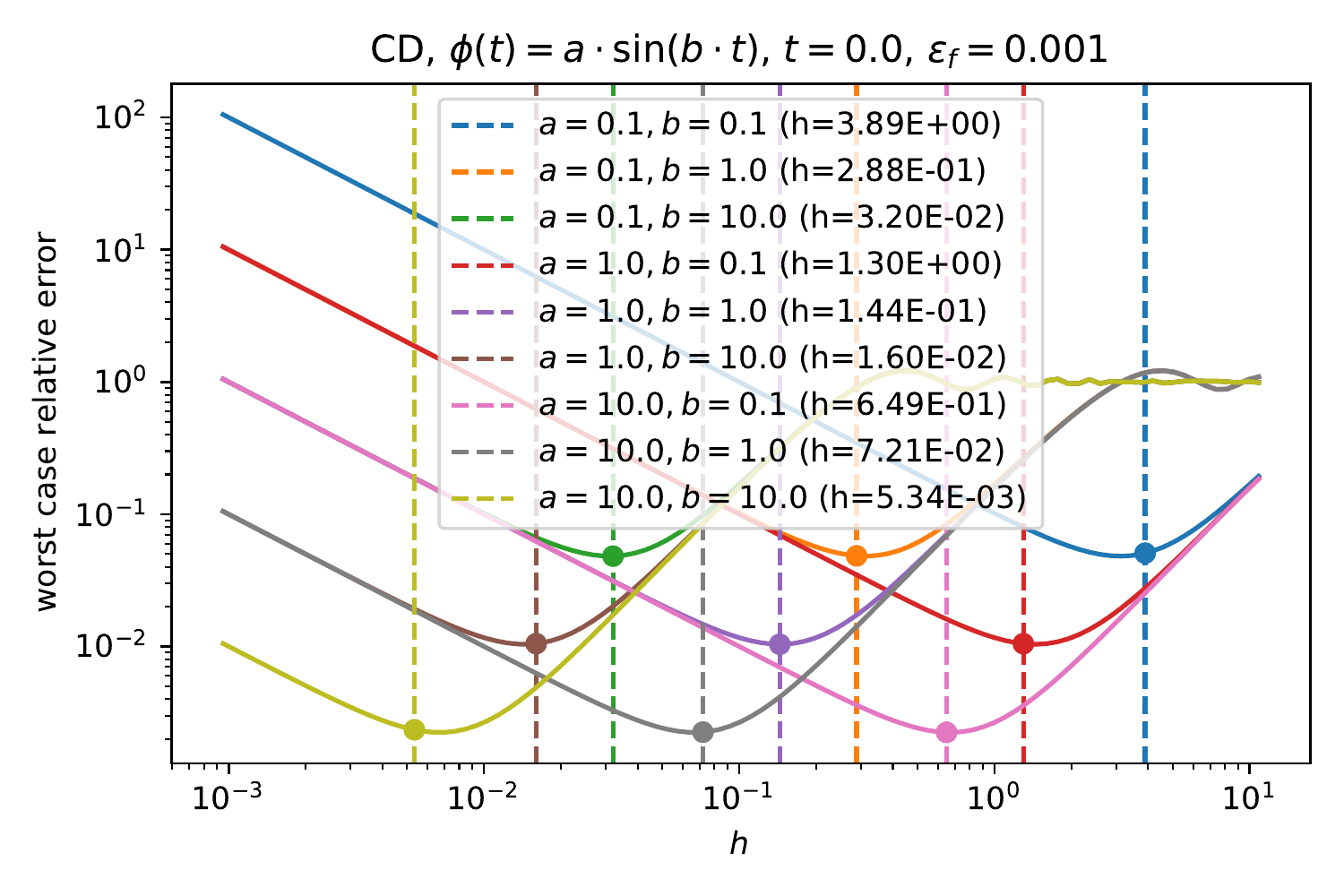}
    \\
    \includegraphics[width=0.49\linewidth]{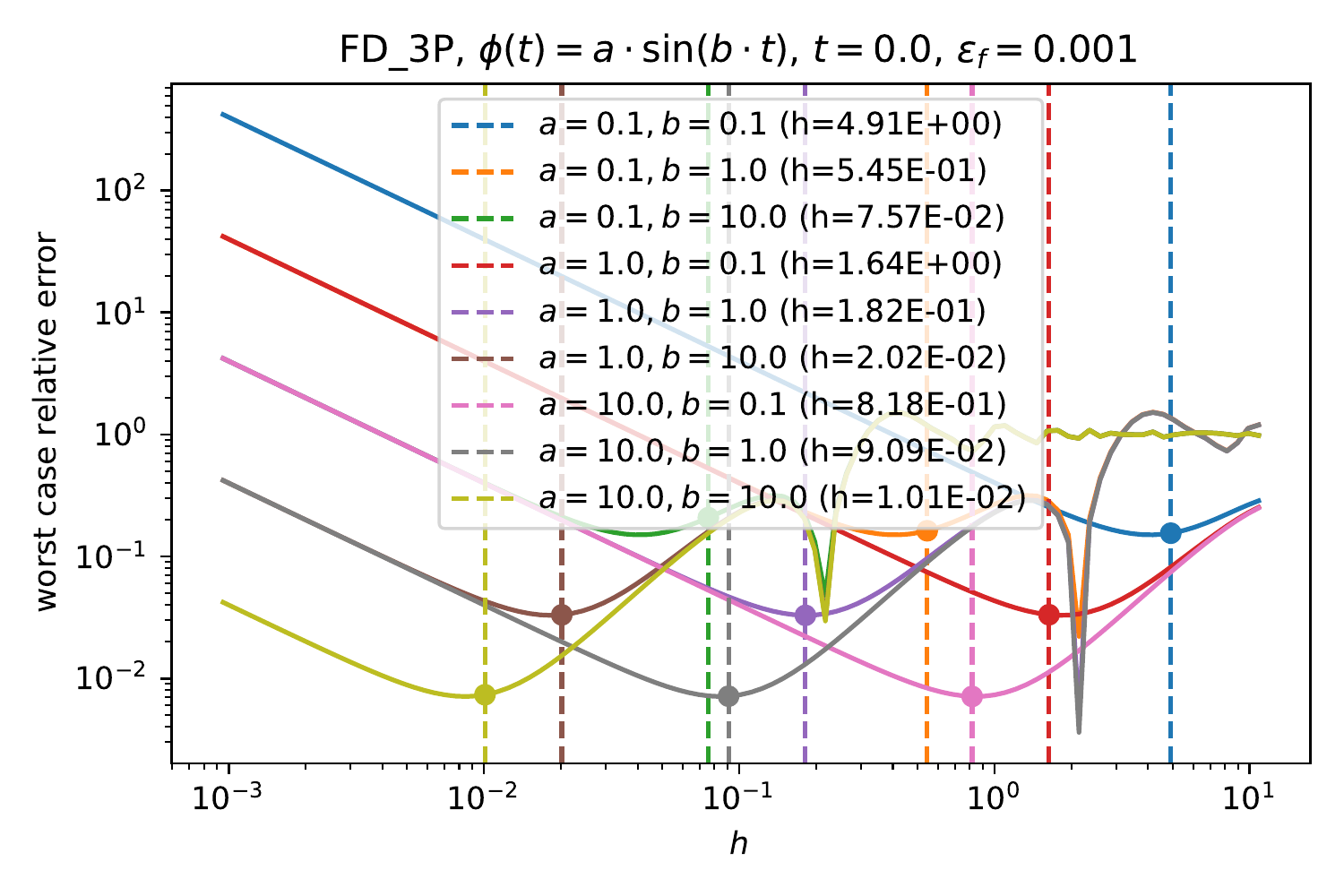}
	\includegraphics[width=0.49\linewidth]{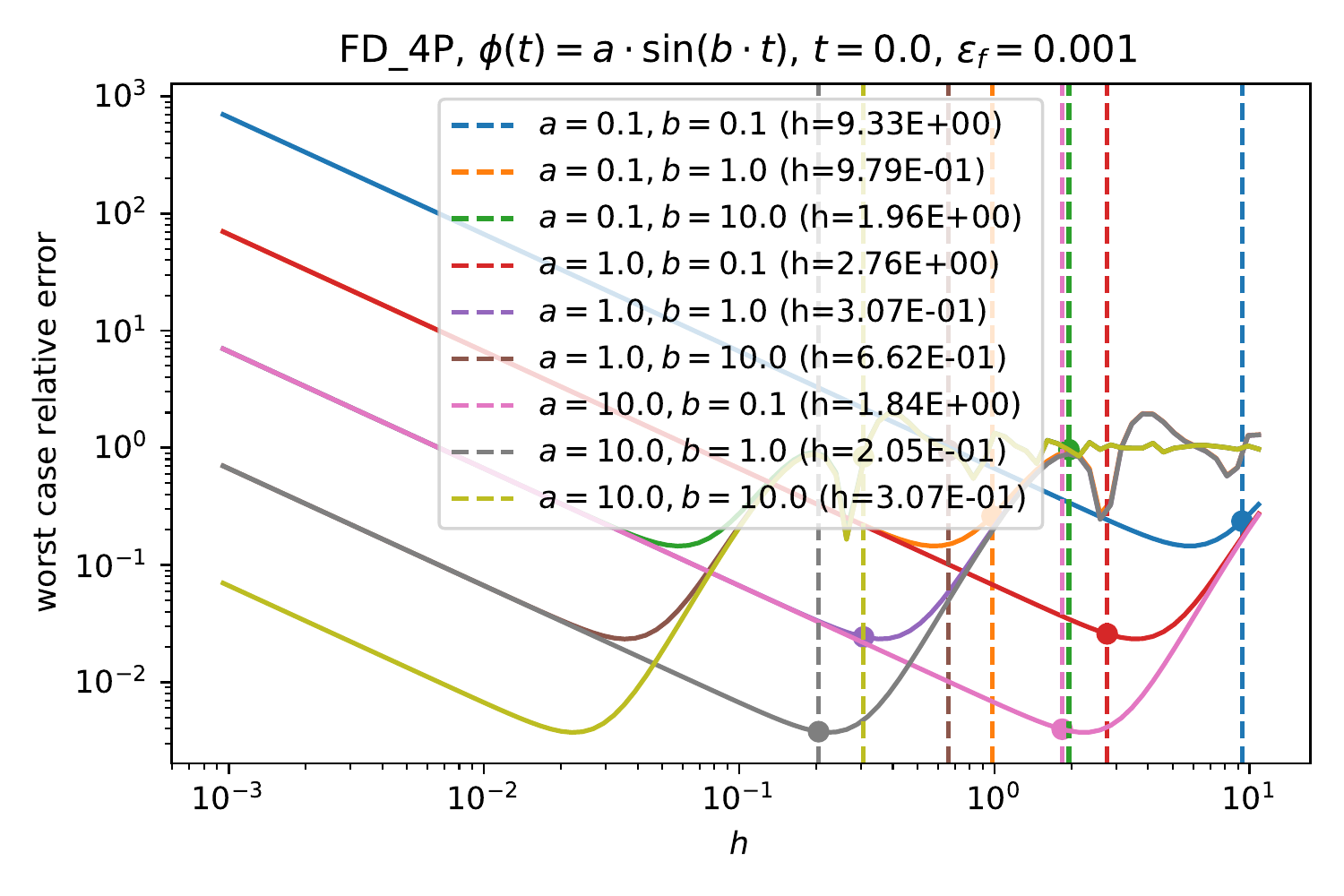}
	\\
	\includegraphics[width=0.49\linewidth]{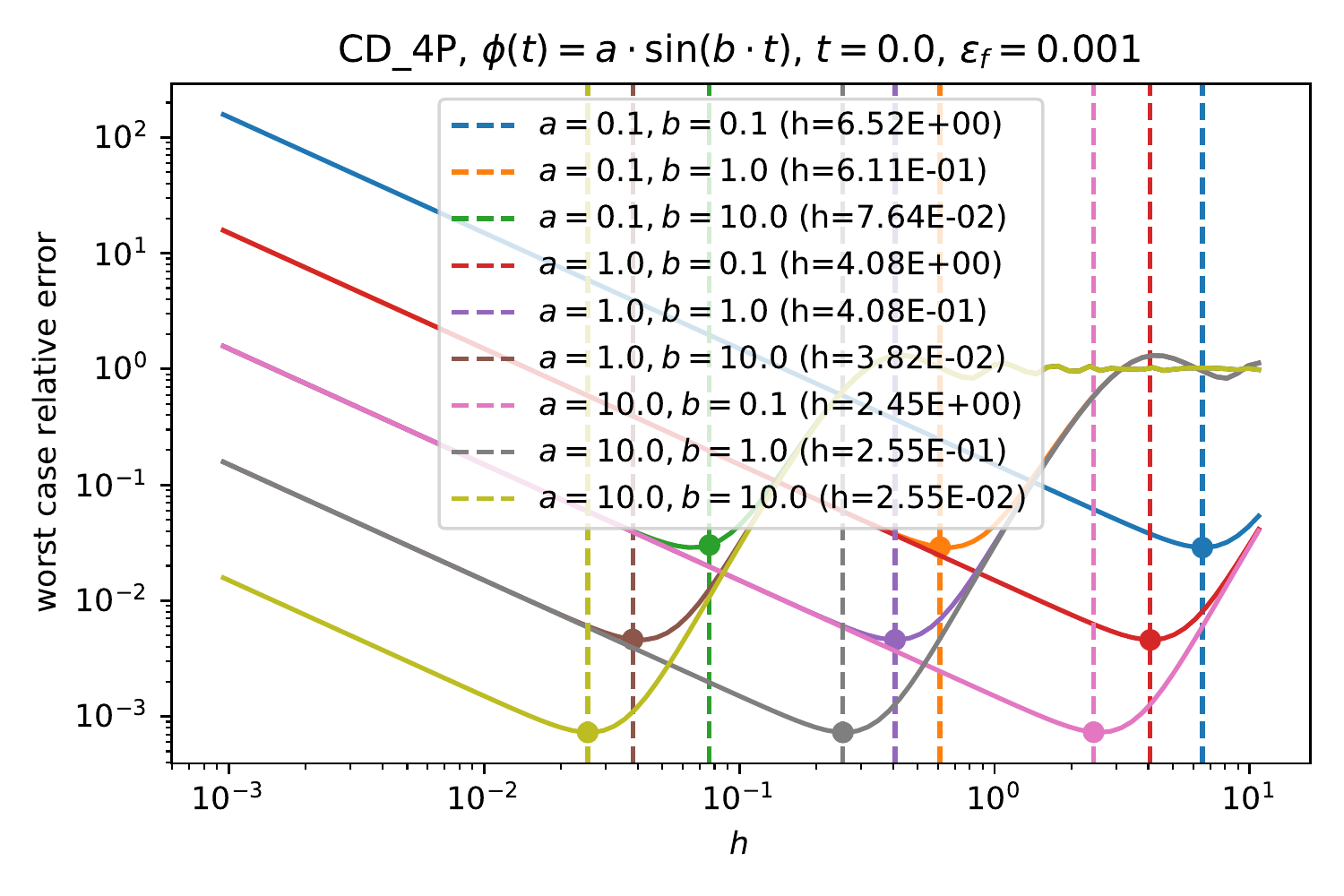}
	\includegraphics[width=0.49\linewidth]{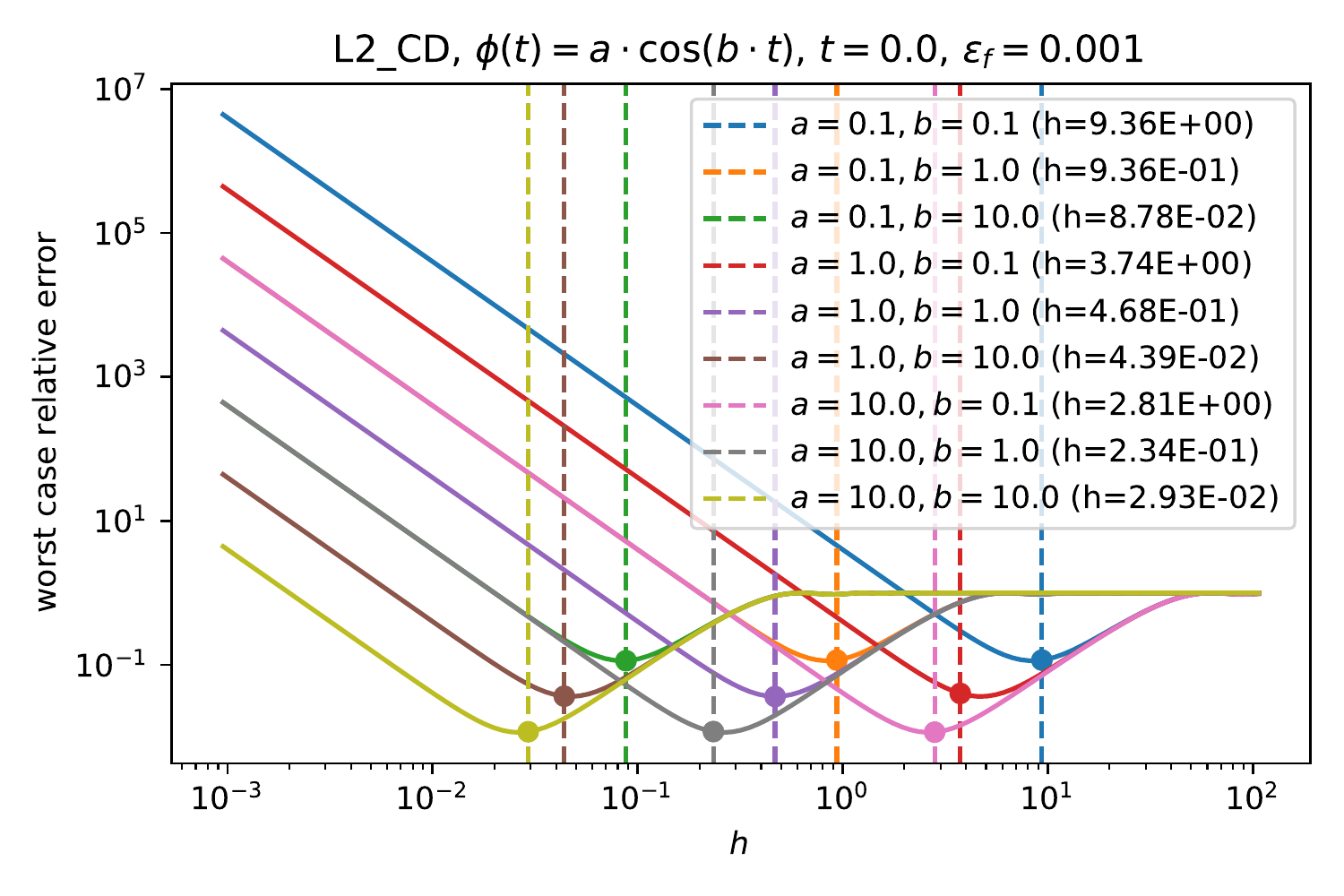}
	\\
    \caption{Worst case relative error $\delta_S(h; \phi, t, \epsilon_f)$ against $h$ on function $\phi(t) = a \cdot \sin(b \cdot t)$ for different $a$ and $b$; the vertical dashed line represents the $h_\dagger$ output by Algorithm \ref{algo:gen_est_proc}.
    }
    \label{fig:sin_func_scale}
\end{figure}

\begin{table}[tbhp]
	{\scriptsize
	\caption{Detailed results for $\phi(t) = a\cdot\sin(b \cdot t)$ with $\epsilon_f = \texttt{1E-3}$; $r$ represents the final testing ratio; $h^*$ is the $h$ that minimizes $\delta_S(h; \phi, t, \epsilon_f)$ reported by \texttt{minimize\_scalar} function in \texttt{scipy.optimize} and could be unreliable.}\label{tab:scale_inv_exp}
	\begin{center}
	\begin{tabular}{rrlrrrrrr}
		\hline
		$a$ & $b$ & scheme &    $h_\dagger$ & $h^*$ &    $r$ &  \texttt{\#iters} &  \texttt{\#Evals} & relative error \\
		\hline
 0.10 &  0.10 &     \texttt{FD} & 2.53e+00 &  3.94e+00 &  1.89 &       5 &         8 &  1.75e-02 \\
 0.10 &  1.00 &     \texttt{FD} & 3.48e-01 &  3.94e-01 &  4.34 &       6 &        11 &  5.29e-02 \\
 0.10 & 10.00 &     \texttt{FD} & 2.77e-02 &  3.94e-02 &  2.68 &       4 &         8 &  2.01e-05 \\
 1.00 &  0.10 &     \texttt{FD} & 1.39e+00 &  1.82e+00 &  2.61 &       7 &        12 &  1.30e-02 \\
 1.00 &  1.00 &     \texttt{FD} & 1.58e-01 &  1.82e-01 &  4.91 &       3 &         6 &  3.01e-03 \\
 1.00 & 10.00 &     \texttt{FD} & 1.58e-02 &  1.82e-02 &  4.58 &       2 &         4 &  6.97e-03 \\
10.00 &  0.10 &     \texttt{FD} & 6.32e-01 &  8.44e-01 &  3.23 &       4 &         7 &  4.76e-04 \\
10.00 &  1.00 &     \texttt{FD} & 6.32e-02 &  8.44e-02 &  3.38 &       1 &         3 &  2.83e-04 \\
10.00 & 10.00 &     \texttt{FD} & 6.92e-03 &  8.44e-03 &  3.34 &       5 &         9 &  2.92e-03 \\
 0.10 &  0.10 &     \texttt{CD} & 3.89e+00 &  3.12e+00 &  5.47 &       4 &        10 &  2.51e-02 \\
 0.10 &  1.00 &     \texttt{CD} & 2.88e-01 &  3.12e-01 &  2.30 &       3 &        10 &  1.38e-02 \\
 0.10 & 10.00 &     \texttt{CD} & 3.20e-02 &  3.12e-02 &  3.13 &       4 &        14 &  1.70e-02 \\
 1.00 &  0.10 &     \texttt{CD} & 1.30e+00 &  1.44e+00 &  2.17 &       3 &         8 &  2.81e-03 \\
 1.00 &  1.00 &     \texttt{CD} & 1.44e-01 &  1.44e-01 &  2.97 &       1 &         4 &  3.46e-03 \\
 1.00 & 10.00 &     \texttt{CD} & 1.60e-02 &  1.44e-02 &  4.06 &       3 &        10 &  4.27e-03 \\
10.00 &  0.10 &     \texttt{CD} & 6.49e-01 &  6.70e-01 &  2.73 &       5 &        16 &  7.02e-04 \\
10.00 &  1.00 &     \texttt{CD} & 7.21e-02 &  6.70e-02 &  3.74 &       4 &        14 &  8.66e-04 \\
10.00 & 10.00 &     \texttt{CD} & 5.34e-03 &  6.70e-03 &  1.52 &       4 &        12 &  4.75e-04 \\
 0.10 &  0.10 & \texttt{FD\_3P} & 4.91e+00 &  8.14e+01 &  2.93 &       4 &        11 &  2.13e-02 \\
 0.10 &  1.00 & \texttt{FD\_3P} & 5.45e-01 &  8.14e+00 &  2.45 &       2 &         7 &  6.48e-02 \\
 0.10 & 10.00 & \texttt{FD\_3P} & 7.57e-02 &  8.14e-01 &  4.90 &       5 &        15 &  1.61e-01 \\
 1.00 &  0.10 & \texttt{FD\_3P} & 1.64e+00 &  2.14e+01 &  2.25 &       3 &         9 &  8.79e-03 \\
 1.00 &  1.00 & \texttt{FD\_3P} & 1.82e-01 &  2.14e+00 &  3.63 &       1 &         5 &  1.86e-04 \\
 1.00 & 10.00 & \texttt{FD\_3P} & 2.02e-02 &  1.83e-02 &  4.43 &       3 &         9 &  1.15e-02 \\
10.00 &  0.10 & \texttt{FD\_3P} & 8.18e-01 &  8.45e-01 &  3.48 &       5 &        13 &  1.71e-03 \\
10.00 &  1.00 & \texttt{FD\_3P} & 9.09e-02 &  8.45e-02 &  4.58 &       4 &        11 &  1.77e-03 \\
10.00 & 10.00 & \texttt{FD\_3P} & 1.01e-02 &  8.45e-03 &  6.67 &       6 &        15 &  1.00e-03 \\
 0.10 &  0.10 & \texttt{FD\_4P} & 9.33e+00 &  8.44e+01 &  4.95 &       9 &        31 &  1.63e-01 \\
 0.10 &  1.00 & \texttt{FD\_4P} & 9.79e-01 &  8.44e+00 &  9.64 &       8 &        32 &  1.50e-01 \\
 0.10 & 10.00 & \texttt{FD\_4P} & 1.96e+00 &  1.78e+01 &  7.71 &       7 &        29 &  9.60e-01 \\
 1.00 &  0.10 & \texttt{FD\_4P} & 2.76e+00 &  3.59e+00 &  4.47 &       3 &        11 &  3.59e-03 \\
 1.00 &  1.00 & \texttt{FD\_4P} & 3.07e-01 &  3.59e-01 &  6.55 &       1 &         6 &  8.16e-03 \\
 1.00 & 10.00 & \texttt{FD\_4P} & 6.62e-01 &  5.88e+00 & 10.15 &       8 &        35 &  9.49e-01 \\
10.00 &  0.10 & \texttt{FD\_4P} & 1.84e+00 &  2.25e+00 &  6.57 &       4 &        14 &  3.68e-04 \\
10.00 &  1.00 & \texttt{FD\_4P} & 2.05e-01 &  2.71e+00 & 10.68 &       3 &        10 &  2.90e-04 \\
10.00 & 10.00 & \texttt{FD\_4P} & 3.07e-01 &  4.62e+00 & 16.47 &       1 &         6 &  8.38e-01 \\
 0.10 &  0.10 & \texttt{CD\_4P} & 6.52e+00 &  7.97e+01 &  2.12 &       5 &        14 &  5.73e-03 \\
 0.10 &  1.00 & \texttt{CD\_4P} & 6.11e-01 &  7.97e+00 &  1.57 &       3 &        14 &  4.45e-03 \\
 0.10 & 10.00 & \texttt{CD\_4P} & 7.64e-02 &  7.97e-01 &  4.32 &       5 &        18 &  1.06e-02 \\
 1.00 &  0.10 & \texttt{CD\_4P} & 4.08e+00 &  4.10e+00 &  2.30 &       7 &        26 &  9.02e-04 \\
 1.00 &  1.00 & \texttt{CD\_4P} & 4.08e-01 &  4.10e-01 &  2.30 &       1 &         6 &  9.02e-04 \\
 1.00 & 10.00 & \texttt{CD\_4P} & 3.82e-02 &  4.10e-02 &  1.68 &       6 &        20 &  6.98e-04 \\
10.00 &  0.10 & \texttt{CD\_4P} & 2.45e+00 &  2.58e+00 &  1.89 &       5 &        18 &  1.18e-04 \\
10.00 &  1.00 & \texttt{CD\_4P} & 2.55e-01 &  2.58e-01 &  2.31 &       4 &        20 &  1.39e-04 \\
10.00 & 10.00 & \texttt{CD\_4P} & 2.55e-02 &  2.58e-02 &  2.31 &       5 &        14 &  1.39e-04 \\
\hline
	\end{tabular}
	\end{center}
}	
\end{table}

\begin{table}[tbhp]
	{\scriptsize
	\caption{Detailed results for $\phi(t) = a\cdot\sin(b \cdot t)$ with $\epsilon_f = \texttt{1E-3}$; $r$ represents the final testing ratio; $h^*$ is the $h$ that minimizes $\delta_S(h; \phi, t, \epsilon_f)$ reported by \texttt{minimize\_scalar} function in \texttt{scipy.optimize} and could be unreliable.}\label{tab:scale_inv_exp_2}
	\begin{center}
	\begin{tabular}{rrlrrrrrr}
		\hline
		$a$ & $b$ & scheme &    $h_\dagger$ & $h^*$ &    $r$ &  \texttt{\#iters} &  \texttt{\#Evals} & relative error \\
\hline
0.10 &  0.10 & \texttt{L2\_CD} & 9.36e+00 &  8.42e+00 & 4.39 &       8 &        23 &  5.97e-02 \\
 0.10 &  1.00 & \texttt{L2\_CD} & 9.36e-01 &  8.42e-01 & 4.79 &       2 &         7 &  4.28e-02 \\
 0.10 & 10.00 & \texttt{L2\_CD} & 8.78e-02 &  8.99e-01 & 3.28 &       5 &        15 &  5.32e-02 \\
 1.00 &  0.10 & \texttt{L2\_CD} & 3.74e+00 &  4.70e+00 & 1.99 &       4 &        11 &  9.49e-03 \\
 1.00 &  1.00 & \texttt{L2\_CD} & 4.68e-01 &  4.70e-01 & 2.57 &       1 &         5 &  2.13e-02 \\
 1.00 & 10.00 & \texttt{L2\_CD} & 4.39e-02 &  4.70e-02 & 2.15 &       6 &        17 &  1.70e-02 \\
10.00 &  0.10 & \texttt{L2\_CD} & 2.81e+00 &  2.64e+00 & 3.59 &       5 &        15 &  7.63e-03 \\
10.00 &  1.00 & \texttt{L2\_CD} & 2.34e-01 &  2.64e-01 & 2.62 &       2 &         7 &  6.94e-04 \\
10.00 & 10.00 & \texttt{L2\_CD} & 2.93e-02 &  2.64e-02 & 5.02 &       5 &        13 &  3.94e-03 \\
\hline
	\end{tabular}
	\end{center}
}	
\end{table}

\subsection{Difficult and Special Examples}

Here, we present the full table of results for the examples listed in Section \ref{sec:experiments} in Table \ref{tab:special_exp} with $\epsilon_f = 10^{-3}$. For reference, the considered problems are:
\begin{enumerate}
	\item $\phi(t) = \bpa{e^t - 1}^2$, at $t = -8$.
	\item $\phi(t) = e^{100 t}$, at $t = 0.01$. 
	\item $\phi(t) = t^4 + 3 t^2 - 10 t$, at $t = 0.99999$. 
	\item $\phi(t) = 10000 t^3 + 0.01 t^2 + 5 t$, at $t = 10^{-9}$. 
\end{enumerate}

\begin{table}[tbhp]
	{\scriptsize
	\caption{Detailed results for special examples, with $\epsilon_f = \texttt{1E-3}$; $r$ represents the final testing ratio; $h^*$ is the $h$ that minimizes $\delta_S(h; \phi, t, \epsilon_f)$ reported by \texttt{minimize\_scalar} function in \texttt{scipy.optimize} and could be unreliable.}\label{tab:special_exp}
	\begin{center}
	\begin{tabular}{llrrrrrr}
		\hline
		$\phi(t)$ & scheme &    $h_\dagger$ & $h^*$ &    $r$ &  \texttt{\#iters} &  \texttt{\#Evals} & relative error \\
		\hline
  $\left(e^{t} - 1.0\right)^{2}$ &     \texttt{FD} & 1.01e+00 &  1.46e+00 &  4.49 &       3 &         5 &  1.02e+00 \\
  $\left(e^{t} - 1.0\right)^{2}$ &     \texttt{CD} & 1.30e+00 &  1.53e+00 &  3.38 &       3 &         8 &  5.73e-02 \\
  $\left(e^{t} - 1.0\right)^{2}$ & \texttt{FD\_3P} & 8.18e-01 &  3.82e+02 &  2.28 &       5 &        13 &  6.63e-01 \\
  $\left(e^{t} - 1.0\right)^{2}$ & \texttt{FD\_4P} & 9.21e-01 &  3.82e+02 &  4.14 &       2 &         8 &  4.15e+00 \\
  $\left(e^{t} - 1.0\right)^{2}$ & \texttt{CD\_4P} & 1.43e+00 &  3.82e+02 &  1.84 &       5 &        22 &  1.11e+00 \\
  $\left(e^{t} - 1.0\right)^{2}$ & \texttt{L2\_CD} & 2.34e+00 &  8.68e+00 &  3.03 &       6 &        19 &  3.90e-01 \\
                     $e^{100 t}$ &     \texttt{FD} & 4.32e-04 &  3.79e-04 &  3.72 &       7 &        11 &  2.74e-02 \\
                     $e^{100 t}$ &     \texttt{CD} & 1.19e-03 &  1.03e-03 &  4.29 &       7 &        20 &  3.09e-03 \\
                     $e^{100 t}$ & \texttt{FD\_3P} & 1.12e-03 &  3.82e+02 &  3.08 &       8 &        19 &  6.81e-03 \\
                     $e^{100 t}$ & \texttt{FD\_4P} & 1.90e-03 &  3.82e+02 &  6.54 &       8 &        23 &  4.97e-03 \\
                     $e^{100 t}$ & \texttt{CD\_4P} & 3.18e-03 &  3.82e+02 &  2.15 &       8 &        20 &  4.60e-04 \\
                     $e^{100 t}$ & \texttt{L2\_CD} & 3.66e-03 &  3.64e-03 &  3.01 &       8 &        19 &  1.18e-02 \\
        $t^{4} + 3 t^{2} - 10 t$ &     \texttt{FD} & 1.58e-02 &  1.48e-02 &  3.55 &       2 &         4 &  7.97e+02 \\
        $t^{4} + 3 t^{2} - 10 t$ &     \texttt{CD} & 4.81e-02 &  5.00e-02 &  2.94 &       2 &         6 &  2.15e+01 \\
        $t^{4} + 3 t^{2} - 10 t$ & \texttt{FD\_3P} & 6.06e-02 &  6.16e-02 &  3.64 &       2 &         7 &  2.38e+02 \\
        $t^{4} + 3 t^{2} - 10 t$ & \texttt{FD\_4P} & 1.54e-01 &  1.39e-01 & 11.90 &       4 &        13 &  1.93e+02 \\
        $t^{4} + 3 t^{2} - 10 t$ & \texttt{CD\_4P} & 9.39e+02 &  4.87e+03 &  1.62 &      16 &        48 &  2.71e-03 \\
        $t^{4} + 3 t^{2} - 10 t$ & \texttt{L2\_CD} & 2.34e-01 &  2.11e-01 &  4.53 &       2 &         7 &  5.38e-03 \\
$10000 t^{3} + 0.01 t^{2} + 5 t$ &     \texttt{FD} & 3.95e-03 &  4.64e-03 &  4.36 &       3 &         5 &  5.46e-02 \\
$10000 t^{3} + 0.01 t^{2} + 5 t$ &     \texttt{CD} & 3.56e-03 &  3.68e-03 &  2.63 &       6 &        18 &  4.02e-02 \\
$10000 t^{3} + 0.01 t^{2} + 5 t$ & \texttt{FD\_3P} & 4.49e-03 &  4.64e-03 &  3.65 &       6 &        15 &  1.17e-02 \\
$10000 t^{3} + 0.01 t^{2} + 5 t$ & \texttt{FD\_4P} & 6.72e+02 &  3.20e+03 & 11.72 &       8 &        22 &  2.37e-06 \\
$10000 t^{3} + 0.01 t^{2} + 5 t$ & \texttt{CD\_4P} & 8.35e+02 &  1.03e+04 &  1.95 &      12 &        28 &  1.99e-07 \\
$10000 t^{3} + 0.01 t^{2} + 5 t$ & \texttt{L2\_CD} & 9.59e+02 &  2.84e+03 &  1.95 &      12 &        27 &  7.49e-08 \\
\hline
\end{tabular}
\end{center}
}
\end{table}

\subsection{Comparison with Mor\'e-Wild Heuristic}

We compare our adaptive forward-difference procedure against the Mor\'e-Wild heuristic \cite{more2012estimating}, as described in Section \ref{subsec:comparison}. 

First, observe that if function $\phi$ has (near) central symmetry at $t$, then Mor\'e-Wild heuristic is very likely to fail. To demonstrate this, we test on $\phi(t) = \sin(t)$ with various value of $t$ close to $0$ and different noise levels $\epsilon_f$. The results are summarized in Table \ref{tab:more_wild_0}. 

\begin{table}[tbhp]
	{\scriptsize
	\caption{Comparison between the Mor\'e-Wild heuristic against our adaptive procedure on function $\phi(t) = \sin(t)$ with various $\epsilon_f$ and $t$. We use ``$--$'' to report the cases where Mor\'e-Wild heuristic fails. Subscript ``MW'' labels the results corresponding to Mor\'e-Wild heuristic, and subscript ``ada'' labels the results corresponding to our adaptive procedure; $\delta$ is the relative error, and $\overline{\delta}$ is the worst-case relative error.}\label{tab:more_wild_0}
	\begin{center}
\begin{tabular}{rr|rrr|rrr}
\hline
   $\epsilon_f$ &            $t$ & $h_\text{MW}$ & $\delta_\text{MW}$ & $\overline{\delta}_\text{MW}$ & $h_\text{ada}$ & $\delta_\text{ada}$ & $\overline{\delta}_\text{ada}$ \\
\hline
1.00e-08 & 1.00e-08 &         $--$ &                 $--$ &                       $--$ &   3.20e-03 &              0.000 &                    0.000 \\
1.00e-08 & 1.00e-06 &         $--$ &                 $--$ &                       $--$ &   3.20e-03 &              0.000 &                    0.000 \\
1.00e-08 & 1.00e-04 &    1.68e-02 &               0.000 &                     0.000 &   3.20e-03 &              0.000 &                    0.000 \\
1.00e-08 & 1.00e-02 &    1.68e-03 &               0.000 &                     0.000 &   2.00e-03 &              0.000 &                    0.000 \\
1.00e-08 & 0.00e+00 &         $--$ &                 $--$ &                       $--$ &   3.20e-03 &              0.000 &                    0.000 \\
1.00e-06 & 1.00e-08 &         $--$ &                 $--$ &                       $--$ &   1.40e-02 &              0.000 &                    0.000 \\
1.00e-06 & 1.00e-06 &         $--$ &                 $--$ &                       $--$ &   1.40e-02 &              0.000 &                    0.000 \\
1.00e-06 & 1.00e-04 &         $--$ &                 $--$ &                       $--$ &   1.40e-02 &              0.000 &                    0.000 \\
1.00e-06 & 1.00e-02 &    1.69e-02 &               0.000 &                     0.000 &   1.40e-02 &              0.000 &                    0.000 \\
1.00e-06 & 0.00e+00 &         $--$ &                 $--$ &                       $--$ &   1.40e-02 &              0.000 &                    0.000 \\
1.00e-04 & 1.00e-08 &         $--$ &                 $--$ &                       $--$ &   5.00e-02 &              0.001 &                    0.004 \\
1.00e-04 & 1.00e-06 &         $--$ &                 $--$ &                       $--$ &   6.50e-02 &              0.003 &                    0.004 \\
1.00e-04 & 1.00e-04 &         $--$ &                 $--$ &                       $--$ &   5.00e-02 &              0.001 &                    0.004 \\
1.00e-04 & 1.00e-02 &         $--$ &                 $--$ &                       $--$ &   5.00e-02 &              0.001 &                    0.005 \\
1.00e-04 & 0.00e+00 &         $--$ &                 $--$ &                       $--$ &   6.50e-02 &              0.000 &                    0.004 \\
1.00e-02 & 1.00e-08 &         $--$ &                 $--$ &                       $--$ &   2.00e-01 &              0.058 &                    0.107 \\
1.00e-02 & 1.00e-06 &    5.20e-01 &               0.067 &                     0.083 &   3.50e-01 &              0.018 &                    0.077 \\
1.00e-02 & 1.00e-04 &         $--$ &                 $--$ &                       $--$ &   3.50e-01 &              0.019 &                    0.077 \\
1.00e-02 & 1.00e-02 &         $--$ &                 $--$ &                       $--$ &   3.50e-01 &              0.029 &                    0.079 \\
1.00e-02 & 0.00e+00 &    6.10e-01 &               0.085 &                     0.094 &   3.50e-01 &              0.032 &                    0.077 \\
\hline
\end{tabular}
\end{center}
}
\end{table}

Next, we test our adaptive procedure and the Mor\'e-Wild heuristic on $\phi(t) = a \cdot \bpa{\exp(b\cdot t) - 1}$ at $t = 0$, with a fixed noise level: $\epsilon_f = \texttt{1E-3}$. We summarize our result in Table \ref{tab:more_wild_1}. Notice that Mor\'e-Wild heuristic may not be able to find a suitable estimation for $h$, in which case a failure is declared. In such cases, we will report the result as ``$--$''.

We can see that when the Mor\'e-Wild heuristic does not declare a failure, it usually outputs an interval $h$ that is quite close to our procedure and produces similar relative error as ours. However, there are many cases where Mor\'e-Wild heuristic fails, while our procedure works very robustly in all cases.

\begin{table}[tbhp]
	{\scriptsize
	\caption{Comparison between the Mor\'e-Wild heuristic against our adaptive procedure on function $\phi(t) = a \cdot (\exp(b\cdot t) - 1)$ with $\epsilon_f = \texttt{1E-3}$ at $t=0$. We use ``$--$'' to report the cases where Mor\'e-Wild heuristic fails. Subscript ``MW'' labels the results corresponding to Mor\'e-Wild heuristic, and subscript ``ada'' labels the results corresponding to our adaptive procedure; $\delta$ is the relative error, and $\overline{\delta}$ is the worst-case relative error.}\label{tab:more_wild_1}
	\begin{center}
\begin{tabular}{rr|rrr|rrr}
\hline
     $a$ &      $b$ & $h_\text{MW}$ & $\delta_\text{MW}$ & $\overline{\delta}_\text{MW}$ & $h_\text{ada}$ & $\delta_\text{ada}$ & $\overline{\delta}_\text{ada}$ \\
\hline
0.01 &   0.01 &         $--$ &                 $--$ &                       $--$ &   4.05e+01 &              0.135 &                    0.727 \\
  0.01 &   0.10 &         $--$ &                 $--$ &                       $--$ &   4.05e+00 &              0.145 &                    0.727 \\
  0.01 &   1.00 &         $--$ &                 $--$ &                       $--$ &   4.43e-01 &              0.275 &                    0.710 \\
  0.01 &  10.00 &    4.68e-02 &               0.177 &                     0.702 &   3.95e-02 &              0.084 &                    0.732 \\
  0.01 & 100.00 &         $--$ &                 $--$ &                       $--$ &   3.95e-03 &              0.123 &                    0.732 \\
  0.10 &   0.01 &         $--$ &                 $--$ &                       $--$ &   1.62e+01 &              0.032 &                    0.209 \\
  0.10 &   0.10 &         $--$ &                 $--$ &                       $--$ &   1.77e+00 &              0.074 &                    0.207 \\
  0.10 &   1.00 &    1.64e-01 &               0.072 &                     0.209 &   1.58e-01 &              0.095 &                    0.210 \\
  0.10 &  10.00 &    1.62e-02 &               0.096 &                     0.209 &   1.58e-02 &              0.055 &                    0.210 \\
  0.10 & 100.00 &         $--$ &                 $--$ &                       $--$ &   1.73e-03 &              0.078 &                    0.207 \\
  1.00 &   0.01 &         $--$ &                 $--$ &                       $--$ &   7.08e+00 &              0.041 &                    0.065 \\
  1.00 &   0.10 &         $--$ &                 $--$ &                       $--$ &   6.32e-01 &              0.034 &                    0.064 \\
  1.00 &   1.00 &    5.26e-02 &               0.029 &                     0.065 &   6.32e-02 &              0.036 &                    0.064 \\
  1.00 &  10.00 &    5.28e-03 &               0.012 &                     0.065 &   6.92e-03 &              0.014 &                    0.064 \\
  1.00 & 100.00 &         $--$ &                 $--$ &                       $--$ &   6.18e-04 &              0.009 &                    0.064 \\
 10.00 &   0.01 &         $--$ &                 $--$ &                       $--$ &   2.53e+00 &              0.012 &                    0.021 \\
 10.00 &   0.10 &    1.76e-01 &               0.011 &                     0.020 &   2.53e-01 &              0.009 &                    0.021 \\
 10.00 &   1.00 &    1.68e-02 &               0.006 &                     0.020 &   1.58e-02 &              0.005 &                    0.021 \\
 10.00 &  10.00 &    1.68e-03 &               0.002 &                     0.020 &   2.47e-03 &              0.014 &                    0.021 \\
 10.00 & 100.00 &         $--$ &                 $--$ &                       $--$ &   2.47e-04 &              0.010 &                    0.021 \\
100.00 &   0.01 &         $--$ &                 $--$ &                       $--$ &   6.32e-01 &              0.003 &                    0.006 \\
100.00 &   0.10 &    5.39e-02 &               0.002 &                     0.006 &   6.32e-02 &              0.004 &                    0.006 \\
100.00 &   1.00 &    5.31e-03 &               0.001 &                     0.006 &   6.92e-03 &              0.001 &                    0.006 \\
100.00 &  10.00 &    5.31e-04 &               0.000 &                     0.006 &   6.18e-04 &              0.001 &                    0.006 \\
100.00 & 100.00 &         $--$ &                 $--$ &                       $--$ &   6.18e-05 &              0.001 &                    0.006 \\

\hline
\end{tabular}
\end{center}
}
\end{table}

\subsection{Finite-Difference L-BFGS}

We present the total number of function evaluations and final optimality gap $\phi(x) - \phi^*$ used by each method in Tables \ref{tab:fd l-bfgs 1e-1}--\ref{tab:cd l-bfgs 1e-7}. 

In general, our adaptive procedure is more robust to different noise levels. Our method only fails when the initial choice of $h$ is not sufficiently small to initially identify the local behavior of the function. This can be seen, for example, with the \texttt{BOX2} example. On the other hand, Mor\'e and Wild's heuristic frequently fails when the noise level is large (for example, with $\epsilon_f = 10^{-1}$). This is due both to the case where $\phi(x) \approx 0$ and hence \eqref{eq:more-wild 2} fails, as well as the case where two iterations are insufficient to find an $h$ that satisfies their conditions. In both cases, we denote a failure case with $^*$. As expected, using a fixed interval is always efficient, but may perform poorly when the Hessian in the function changes, as described in Section \ref{sec:experiments}.

\begin{table}[tbhp]
	{\scriptsize
	 \caption{Total number of function evaluations used and final accuracy achieved by forward-difference L-BFGS method with different choices of the finite-difference interval.}
    \label{tab:fd l-bfgs 1e-1}
    \begin{center}
\begin{tabular}{lrr|rr|rr|rr}
		\hline
		& & & \multicolumn{2}{c|}{\textbf{Fixed Interval}} & \multicolumn{2}{c|}{\textbf{Mor\'e-Wild}} & \multicolumn{2}{c}{\textbf{Adaptive}} \\
        Problem & $n$ & $\epsilon_f$ & \#Evals & $\phi(x) - \phi^*$ & \#Evals & $\phi(x) - \phi^*$ & \#Evals & $\phi(x) - \phi^*$ \\
        \hline
        \texttt{AIRCRFTB} & $5$ & \num{1e-01} & $1373$ & \num{ 7.651e+00} & $1457$ & \num{ 4.777e-01} & $559$ & \num{ 4.950e-01} \\
        \texttt{ALLINITU} & $4$ & \num{1e-01} & $656$ & \num{ 2.767e+00} & $972$ & \num{ 3.708e-01} & $686$ & \num{ 1.003e-01} \\
        \texttt{ARWHEAD} & $100$ & \num{1e-01} & $2513$ & \num{ 2.027e-01} & $8784$ & \num{ 3.040e-01} & $2811$ & \num{ 2.458e-01} \\
        \texttt{BARD} & $3$ & \num{1e-01} & $650$ & \num{ 8.311e-01} & $602$ & \num{ 7.581e-02} & $670$ & \num{ 3.942e-01} \\
        \texttt{BDQRTIC} & $100$ & \num{1e-01} & $2601$ & \num{ 6.278e+00} & $9213$ & \num{ 5.902e+00} & $5246$ & \num{ 4.438e+00} \\
        \texttt{BIGGS3} & $3$ & \num{1e-01} & $650$ & \num{ 1.195e+00} & $726$ & \num{ 1.348e+00} & $674$ & \num{ 1.545e+00} \\
\hline
    \end{tabular}
  
 \end{center}   
}
\end{table}

\begin{table}[tbhp]
	{\scriptsize
	\caption{Total number of function evaluations used and final accuracy achieved by forward-difference L-BFGS method with different choices of the finite-difference interval.}
    \label{tab:fd l-bfgs 1e-3}
    \begin{center}
	\begin{tabular}{lrr|rr|rr|rr}
		\hline
		& & & \multicolumn{2}{c|}{\textbf{Fixed Interval}} & \multicolumn{2}{c|}{\textbf{Mor\'e-Wild}} & \multicolumn{2}{c}{\textbf{Adaptive}} \\
        Problem & $n$ & $\epsilon_f$ & \#Evals & $\phi(x) - \phi^*$ & \#Evals & $\phi(x) - \phi^*$ & \#Evals & $\phi(x) - \phi^*$ \\
        \hline
        \texttt{AIRCRFTB} & $5$ & \num{1e-03} & $1008$ & \num{ 3.772e-01} & $817$ & \num{ 3.313e-01} & $560$ & \num{ 3.286e-01} \\
        \texttt{ALLINITU} & $4$ & \num{1e-03} & $656$ & \num{ 2.354e-03} & $563$ & \num{ 1.858e-03} & $686$ & \num{ 1.172e-03} \\
        \texttt{ARWHEAD} & $100$ & \num{1e-03} & $2513$ & \num{ 4.510e-02} & $7331$ & \num{ 4.407e-02} & $4951$ & \num{ 4.160e-02} \\
        \texttt{BARD} & $3$ & \num{1e-03} & $650$ & \num{ 2.093e-03} & $602$ & \num{ 3.467e-03} & $538$ & \num{ 3.273e-03} \\
        \texttt{BDQRTIC} & $100$ & \num{1e-03} & $6017$ & \num{ 9.588e-02} & $8733$ & \num{ 1.151e-01} & $8551$ & \num{ 9.804e-02} \\
        \texttt{BIGGS3} & $3$ & \num{1e-03} & $650$ & \num{ 8.010e-04} & $712$ & \num{ 1.375e-02} & $672$ & \num{ 2.297e-02} \\
\hline
    \end{tabular}
    \end{center}
}
\end{table}

\begin{table}[tbhp]
	{\scriptsize
	\caption{Total number of function evaluations used and final accuracy achieved by forward-difference L-BFGS method with different choices of the finite-difference interval.}
    \label{tab:fd l-bfgs 1e-5}
    \begin{center}
	\begin{tabular}{lrr|rr|rr|rr}
		\hline
		& & & \multicolumn{2}{c|}{\textbf{Fixed Interval}} & \multicolumn{2}{c|}{\textbf{Mor\'e-Wild}} & \multicolumn{2}{c}{\textbf{Adaptive}} \\
        Problem & $n$ & $\epsilon_f$ & \#Evals & $\phi(x) - \phi^*$ & \#Evals & $\phi(x) - \phi^*$ & \#Evals & $\phi(x) - \phi^*$ \\
        \hline
        \texttt{AIRCRFTB} & $5$ & \num{1e-05} & $662$ & \num{ 3.023e-01} & $1880$ & \num{ 2.674e-04} & $2014$ & \num{ 4.084e-04} \\
        \texttt{ALLINITU} & $4$ & \num{1e-05} & $656$ & \num{ 1.835e-05} & $706$ & \num{ 1.146e-05} & $682$ & \num{ 6.468e-06} \\
        \texttt{ARWHEAD} & $100$ & \num{1e-05} & $3832$ & \num{ 1.827e-04} & $11900$ & \num{ 3.262e-04} & $7570$ & \num{ 5.117e-04} \\
        \texttt{BARD} & $3$ & \num{1e-05} & $650$ & \num{ 1.894e-03} & $602$ & \num{ 1.878e-03} & $674$ & \num{ 1.903e-03} \\
        \texttt{BDQRTIC} & $100$ & \num{1e-05} & $10225$ & \num{ 3.227e-03} & $14968$ & \num{ 2.305e-04} & $9860$ & \num{ 5.248e-04} \\
        \texttt{BIGGS3} & $3$ & \num{1e-05} & $631$ & \num{ 2.826e-04} & $712$ & \num{ 2.404e-04} & $670$ & \num{ 1.581e-04} \\
\hline
    \end{tabular}
    \end{center}
}
\end{table}

\begin{table}[tbhp]
	{\scriptsize
	\caption{Total number of function evaluations used and final accuracy achieved by forward-difference L-BFGS method with different choices of the finite-difference interval.}
    \label{tab:fd l-bfgs 1e-7}
    \begin{center}
	\begin{tabular}{lrr|rr|rr|rr}
		\hline
		& & & \multicolumn{2}{c|}{\textbf{Fixed Interval}} & \multicolumn{2}{c|}{\textbf{Mor\'e-Wild}} & \multicolumn{2}{c}{\textbf{Adaptive}} \\
        Problem & $n$ & $\epsilon_f$ & \#Evals & $\phi(x) - \phi^*$ & \#Evals & $\phi(x) - \phi^*$ & \#Evals & $\phi(x) - \phi^*$ \\
        \hline
        \texttt{AIRCRFTB} & $5$ & \num{1e-07} & $662$ & \num{ 9.768e-05} & $2985$ & \num{ 1.199e-05} & $1334$ & \num{ 1.001e-05} \\
        \texttt{ALLINITU} & $4$ & \num{1e-07} & $656$ & \num{ 2.172e-07} & $706$ & \num{ 1.061e-07} & $543$ & \num{ 7.839e-08} \\
        \texttt{ARWHEAD} & $100$ & \num{1e-07} & $4326$ & \num{ 1.643e-06} & $11900$ & \num{ 5.117e-07} & $5842$ & \num{ 1.049e-06} \\
        \texttt{BARD} & $3$ & \num{1e-07} & $650$ & \num{ 2.376e-05} & $706$ & \num{ 7.175e-07} & $660$ & \num{ 1.273e-07} \\
        \texttt{BDQRTIC} & $100$ & \num{1e-07} & $8011$ & \num{ 1.802e-05} & $14968$ & \num{ 5.277e-06} & $12535$ & \num{ 4.346e-06} \\
        \texttt{BIGGS3} & $3$ & \num{1e-07} & $650$ & \num{ 2.348e-07} & $726$ & \num{ 3.990e-07} & $1018$ & \num{ 1.386e-06} \\
\hline
    \end{tabular}
    \end{center}
    }
\end{table}

\begin{table}[tbhp]
	{\scriptsize
	\caption{Total number of function evaluations used and final accuracy achieved by central-difference L-BFGS method with different choices of the finite-difference interval.}
    \label{tab:cd l-bfgs 1e-1}
    \begin{center}
	\begin{tabular}{lrr|rr|rr}
		\hline
		& & & \multicolumn{2}{c|}{\textbf{Fixed Interval}} & \multicolumn{2}{c}{\textbf{Adaptive}} \\
        Problem & $n$ & $\epsilon_f$ & \#Evals & $\phi(x) - \phi^*$ & \#Evals & $\phi(x) - \phi^*$ \\
        \hline
        \texttt{AIRCRFTB} & $5$ & \num{1e-01} & $602$ & \num{ 4.455e-01} & $3074$ & \num{ 4.451e-01} \\
        \texttt{ALLINITU} & $4$ & \num{1e-01} & $848$ & \num{ 7.698e-01} & $1371$ & \num{ 3.378e-03} \\
        \texttt{ARWHEAD} & $100$ & \num{1e-01} & $4001$ & \num{ 4.137e-02} & $5021$ & \num{ 5.157e-02} \\
        \texttt{BARD} & $3$ & \num{1e-01} & $669$ & \num{ 2.616e-03} & $947$ & \num{ 1.535e-02} \\
        \texttt{BDQRTIC} & $100$ & \num{1e-01} & $8594$ & \num{ 2.047e-01} & $11555$ & \num{ 1.918e-01} \\
        \texttt{BIGGS3} & $3$ & \num{1e-01} & $669$ & \num{ 2.607e-01} & $1150$ & \num{ 1.154e-02} \\
\hline
    \end{tabular}
    \end{center}
    }
\end{table}

\begin{table}[tbhp]
	{\scriptsize
	    \caption{Total number of function evaluations used and final accuracy achieved by central-difference L-BFGS method with different choices of the finite-difference interval.}
    \label{tab:cd l-bfgs 1e-3}
    \begin{center}
	\begin{tabular}{lrr|rr|rr}
		\hline
		& & & \multicolumn{2}{c|}{\textbf{Fixed Interval}} & \multicolumn{2}{c}{\textbf{Adaptive}} \\
        Problem & $n$ & $\epsilon_f$ & \#Evals & $\phi(x) - \phi^*$ & \#Evals & $\phi(x) - \phi^*$ \\
        \hline
        \texttt{AIRCRFTB} & $5$ & \num{1e-03} & $594$ & \num{ 4.451e-01} & $2306$ & \num{ 4.451e-01} \\
        \texttt{ALLINITU} & $4$ & \num{1e-03} & $574$ & \num{ 3.325e-03} & $1377$ & \num{ 1.847e-05} \\
        \texttt{ARWHEAD} & $100$ & \num{1e-03} & $7175$ & \num{ 5.590e-04} & $8379$ & \num{ 6.125e-04} \\
        \texttt{BARD} & $3$ & \num{1e-03} & $669$ & \num{ 1.905e-03} & $713$ & \num{ 1.883e-03} \\
        \texttt{BDQRTIC} & $100$ & \num{1e-03} & $11220$ & \num{ 3.667e-03} & $31078$ & \num{ 4.314e-04} \\
        \texttt{BIGGS3} & $3$ & \num{1e-03} & $669$ & \num{ 1.009e-03} & $1363$ & \num{ 1.314e-04} \\
\hline
    \end{tabular}
\end{center}
}
\end{table}

\begin{table}[tbhp]
	{\scriptsize
	    \caption{Total number of function evaluations used and final accuracy achieved by central-difference L-BFGS method with different choices of the finite-difference interval.}
    \label{tab:cd l-bfgs 1e-5}
    \begin{center}
	\begin{tabular}{lrr|rr|rr}
		\hline
		& & & \multicolumn{2}{c|}{\textbf{Fixed Interval}} & \multicolumn{2}{c}{\textbf{Adaptive}} \\
        Problem & $n$ & $\epsilon_f$ & \#Evals & $\phi(x) - \phi^*$ & \#Evals & $\phi(x) - \phi^*$ \\
        \hline
        \texttt{AIRCRFTB} & $5$ & \num{1e-05} & $612$ & \num{ 4.451e-01} & $2558$ & \num{ 3.293e-01} \\
        \texttt{ALLINITU} & $4$ & \num{1e-05} & $681$ & \num{ 7.277e-06} & $2035$ & \num{ 7.767e-09} \\
        \texttt{ARWHEAD} & $100$ & \num{1e-05} & $7487$ & \num{ 4.138e-06} & $7477$ & \num{ 1.108e-06} \\
        \texttt{BARD} & $3$ & \num{1e-05} & $669$ & \num{ 1.867e-03} & $1500$ & \num{ 1.678e-07} \\
        \texttt{BDQRTIC} & $100$ & \num{1e-05} & $11220$ & \num{ 1.539e-06} & $19106$ & \num{ 4.834e-06} \\
        \texttt{BIGGS3} & $3$ & \num{1e-05} & $669$ & \num{ 4.019e-06} & $1088$ & \num{ 3.251e-07} \\
\hline
    \end{tabular}
\end{center}
}
\end{table}

\begin{table}[tbhp]
	{\scriptsize
	    \caption{Total number of function evaluations used and final accuracy achieved by central-difference L-BFGS method with different choices of the finite-difference interval.}
    \label{tab:cd l-bfgs 1e-7}
    \begin{center}
	\begin{tabular}{lrr|rr|rr}
		\hline
		& & & \multicolumn{2}{c|}{\textbf{Fixed Interval}} & \multicolumn{2}{c}{\textbf{Adaptive}} \\
        Problem & $n$ & $\epsilon_f$ & \#Evals & $\phi(x) - \phi^*$ & \#Evals & $\phi(x) - \phi^*$ \\
        \hline
        \texttt{AIRCRFTB} & $5$ & \num{1e-07} & $1585$ & \num{ 6.912e-12} & $2929$ & \num{ 2.059e-22} \\
        \texttt{ALLINITU} & $4$ & \num{1e-07} & $463$ & \num{ 1.644e-08} & $1493$ & \num{ 3.211e-10} \\
        \texttt{ARWHEAD} & $100$ & \num{1e-07} & $5391$ & \num{ 1.192e-08} & $8298$ & \num{ 2.636e-09} \\
        \texttt{BARD} & $3$ & \num{1e-07} & $669$ & \num{ 2.756e-09} & $1357$ & \num{ 4.989e-09} \\
        \texttt{BDQRTIC} & $100$ & \num{1e-07} & $11220$ & \num{ 1.885e-08} & $27324$ & \num{ 6.486e-09} \\
        \texttt{BIGGS3} & $3$ & \num{1e-07} & $669$ & \num{ 1.115e-08} & $1357$ & \num{ 5.031e-10} \\
\hline
    \end{tabular}
\end{center}
}
\end{table}

\end{document}